%% file: main.tex
\documentclass{amsart}
\usepackage[utf8]{inputenc}
\usepackage[margin=01.5 in]{geometry}
\usepackage{dirtytalk}
\usepackage[dvipsnames]{xcolor}
\usepackage{float}

\usepackage{amsmath}
\usepackage{amssymb}
\usepackage{amsthm}
\usepackage{tikz-cd}
\usepackage{tikz}
\usetikzlibrary{knots, snakes}
\usetikzlibrary{arrows.meta}
\usepgflibrary[arrows.meta]
\usetikzlibrary{angles, quotes, decorations.markings, decorations.pathmorphing, decorations.pathreplacing, decorations.shapes}
\usepackage{graphicx}
\usepackage{mathtools}
\usepackage{extarrows}
\usepackage{booktabs}
\usepackage{pythonhighlight}
\usepackage{hyperref}
\usepackage[all]{xy}

\usepackage{amsmath,amsthm,amssymb,latexsym,amscd}
\usepackage{mathrsfs}

\usepackage{hyperref}
\hypersetup{pdftitle={A pair of Jones polynomial positivity obstructions},pdfauthor={Lizzie Buchanan}}
\hypersetup{colorlinks=true,linkcolor=blue,anchorcolor=blue,citecolor=blue}
\numberwithin{equation}{section}
\theoremstyle{plain}
 
\newtheorem{lemma}[equation]{Lemma} 
\newtheorem*{lemma*}{Lemma}

\newtheorem{proposition}[equation]{Proposition}
\newtheorem*{proposition*}{Proposition}

\newtheorem{theorem}[equation]{Theorem}
\newtheorem*{theorem*}{Theorem}

\newtheorem{corollary}[equation]{Corollary}
\newtheorem*{corollary*}{Corollary}

\newtheorem*{conjecture*}{Conjecture}

\def\tcc{\textcolor{cyan}}

\definecolor{purp}{RGB}{148,30,238}
\definecolor{pinky}{RGB}{255, 174, 252}

\definecolor{purp1}{RGB}{120,30,238}
\definecolor{pinky1}{RGB}{243, 162, 252}
\definecolor{pinky2}{RGB}{252, 165, 249}
\definecolor{pinky3}{RGB}{247, 174, 255}

\definecolor{purp2}{RGB}{189, 144, 249}
\definecolor{purp3}{RGB}{203, 172, 243}

\theoremstyle{remark}
\newtheorem{claim}[equation]{Claim}

\newtheorem{remark}[equation]{Remark}

\theoremstyle{definition}

\newtheorem{definition}[equation]{Definition}

\newenvironment{enumalph}
{\begin{enumerate}}
{\end{enumerate}}

\DeclareMathOperator{\coeff}{coeff}
\DeclareMathOperator{\lead}{lead}
\DeclareMathOperator{\term}{term}

\title{A Pair of Jones Polynomial Positivity Obstructions}
\author{Lizzie Buchanan}
\address{Department of Mathematics, Dartmouth College, Hanover, NH 03755}
\email{elizabeth.m.buchanan.gr@dartmouth.edu}
\urladdr{\url{http://math.dartmouth.edu/~ebuchanan/}}

\begin{document}

\maketitle

\begin{abstract}
We provide a new bound on the maximum degree of the Jones polynomial of a positive link with second Jones coefficient equal to $\pm 1$ or $\pm 2$. This builds upon the result of our previous work, in which we found such a bound for positive fibered links. We also show that each of these three bounds obstruct positivity for infinitely many almost-positive diagrams, allowing us to classify infinitely many knots as almost-positive.

\end{abstract}

\section{Introduction}

A positive knot diagram hints at its own existence by forcing certain properties on its knot polynomials and other invariants. For example, the Conway polynomial of a positive knot is positive (\cite{Cromwell}) and the signature of a positive knot is negative (\cite{przytycki2009positive}). So, any of these properties can be used as a positivity obstruction: if a given knot cannot satisfy these conditions, the knot cannot be positive. 

Most of our positivity obstructions fail, however, to distinguish almost-positive knots from positive knots, because many results about positive knots have been shown to be true for almost-positive knots as well. For example, the Conway polynomial of an almost-positive knot is positive (\cite{Cromwell}) and the signature of an almost-positive knot is negative (\cite{przytycki2009almostpositive}). It is also now known that both positive and almost-positive knots are strongly quasipositive (\cite{Feller_2022}). We refer the reader to \cite{Feller_2022} and \cite{StI} for more information on the many overlapping properties of positive and almost-positive knots.

In \cite{Buchanan_2022}, we found a condition on the Jones polynomial of a fibered positive knot that can, in some cases, show that a given knot is almost-positive rather than positive. 

\begin{theorem} \cite{Buchanan_2022}\label{main result for coeff 0}
    The Jones polynomial of a fibered positive $n$-component link $L$ satisfies
    $$\max \deg V_L \leq 4 \min \deg V_L + \frac{n-1}{2}.$$
\end{theorem}

Among positive links, a fibered positive link is distinguished by having $0$ as the second coefficient of its Jones polynomial (\cite{Stoimenow}, \cite{futer2012guts}). In this paper, we focus on two more classes of positive links: those characterized by having $\pm 1$ or $\pm 2$ as the second coefficient of their Jones polynomial. We prove the following theorem. 

\begin{theorem*}$\mathbf{\ref{main result coeff 1 and 2}}$
    Let $L$ be a positive link with $n$ link components, Jones polynomial $V_L$, and Conway polynomial $\nabla_L$. 
    \begin{enumalph}
       \item  If the second coefficient of $V_L$ is $\pm 1$, then 
    $$\max \deg V_L \leq 4 \min \deg V_L + \frac{n-1}{2} + 2\lead \coeff \nabla_L -2. $$
        \item If the second coefficient of $V_L$ is $\pm 2$, then 
    $$\max \deg V_L \leq 4 \min \deg V_L + \frac{n-1}{2} + \lead \coeff \nabla_L.$$ 
     \end{enumalph}
\end{theorem*}

In Section \ref{Application section}, we present two examples of infinite families of knots, and use the two parts of this theorem to prove that these knots cannot be positive, and are instead almost-positive. We also present a related infinite family of knots whose non-positivity can be shown using the result of Theorem \ref{main result for coeff 0}. This demonstrates that each of our theorems obstruct positivity for infinitely many knots. 

\section{Background}

In everything that follows, we always assume we are working with non-split links. We will use the following notation. 
\begin{itemize}
    \item $c(D)$ is the number of crossings in diagram $D$
    \item $n(D)$ is the number of link components
    \item $s(D)$ is the number of Seifert circles
    \item $A_D$ is the number of $A$-circles (sometimes we may just use $A$ if there is only one diagram in question)
    \item $B_D$ (or just $B$) is the number of $B$-circles
    \item $\nabla$ is the Conway polynomial, and $\nabla_D$ (or $\nabla(D)$, $\nabla_L$, or $\nabla(L)$) is the Conway polynomial of link $L$ represented by $D$.
    \item $V$ is the Jones polynomial, and $V_D$ (or $V(D)$, $V_L$, or $V(L)$) is the Jones polynomial of the link $L$ represented by diagram $D$
\end{itemize}

\begin{definition}
    The \textbf{\textit{second coefficient of the Jones polynomial}} is the coefficient of the term with exponent one higher than the minimal degree. For example, the Jones polynomial of the trefoil is $V(3_1) = t + t^3 - t^4$, so we have $\min \deg V = 1$, $\max \deg V = 4$, and the \textbf{\textit{second Jones coefficient}} is equal to $0$. 
\end{definition}

\begin{definition}
A (oriented) link diagram is called \textbf{positive} if every crossing in that diagram is positive. (See Figure \ref{positive and negative crossing}.) A link is a \textbf{positive link} if it has a positive diagram.  
\end{definition}

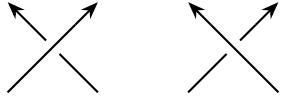
\begin{figure}[h]
\centering
    \begin{tikzpicture}[knot gap=9, scale=0.6]
        
        \begin{scope}[xshift=5cm]
        \draw[thick, -{Stealth}] (0.5,-0.5) -- (-1,1);
        \draw[thick, -{Stealth}] (-1,-1) -- (1,1);
        \draw[thick, knot] (1,-1) -- (-0.5,0.5);
        \draw[thick, knot] (-1,-1) -- (0.5,0.5);
        \end{scope}
        
        \begin{scope}[xshift=9cm]
        \draw[thick, -{Stealth}] (0.5,-0.5) -- (-1,1);
        \draw[thick, -{Stealth}] (-1,-1) -- (1,1);
        \draw[thick, knot] (-1,-1) -- (0.5,0.5);
        \draw[thick, knot] (1,-1) -- (-0.5,0.5);
        \end{scope}
        
    \end{tikzpicture} 
    \caption{ A positive crossing (left) and a negative crossing (right)} \label{positive and negative crossing}
\end{figure}

\begin{definition}
    At any crossing in any link diagram $D$, we can perform an \textbf{$A$-smoothing} or a \textbf{$B$-smoothing} (see Figure \ref{A and B smoothing}). Performing $A$-smoothings on every crossing in the diagram results in an arrangement of circles known as the \textbf{$A$-state}. The set of circles in the $A$-state are called the \textbf{$A$-circles}. 
We then can consider the associated \textbf{$A$-state graph} of the diagram $D$: Every $A$-circle in $D$ corresponds to a vertex in the $A$-state graph, and every crossing in $D$ corresponds to an edge in the graph (see Figure \ref{reduced A-state graph example}). The \textbf{reduced $A$-state graph} is the $A$-state graph with duplicate edges removed. 
\end{definition}

\input{Figures/A_and_B_smoothings}

\input{Figures/reduced_A-state_graph_example}

\begin{remark}
In a positive diagram, $A$-smoothings are equivalent to smoothing according to Seifert's algorithm. So in a positive diagram, $A_D=s(D)$. 
\end{remark}

A light discussion of the Kauffman state-sum model of the Jones polynomial appears in the Appendix (Section \ref{Appendix}). From this approach to the Jones polynomial, it becomes clear that for any link diagram of any link, the minimum degree of the Jones polynomial is bounded below by the degree of contribution of the $A$-state \cite{Kauffman}. For positive diagrams in particular, the lowest degree term is contributed solely by the $A$-state, and can be neatly expressed in terms of the diagram's crossing number and number of $A$-circles. 

\begin{equation}\label{min degree pos}
    \min \deg V_L = \frac{c(D) - A_D + 1}{2}
\end{equation}

Similarly, for any link diagram of any link, the maximum degree of the Jones polynomial is bounded above by the contribution of the $B$-state. For positive diagrams in particular, this bound can be neatly expressed in terms of the crossing number and number of $B$-circles. 

\begin{equation}\label{standard bound}
    \max \deg V_L \leq c(D) + \frac{B_D -1 }{2}.
\end{equation}

Explanation of both of these facts appears in the Appendix.

In this paper, we focus on two kinds of positive link diagrams and try to improve the bound \ref{standard bound} for these diagrams. Specifically, we focus on Balanced diagrams of type $1$, which represent certain positive links with second Jones coefficient equal to $\pm 1$, and Balanced and Oddly Balanced diagrams of type $2$, which represent certain positive links with second Jones coefficient equal to $\pm 2$. These diagrams are formally defined in Section \ref{Definitions}. In Theorems \ref{Bal type 1 B=n} and \ref{Bal type 2 B=n}, we prove that in a Balanced diagram of type $1$ or $2$, the number of $B$-circles is equal to the number of link components. In Theorem \ref{Oddly Bal B almost equal n}, we prove that in an Oddly Balanced diagram of type $2$, the number of $B$-circles either is equal to the number of link components, or the two quantities differ by $\pm 2$. As such, for Balanced (and Oddly Balanced) diagrams we are able to replace the diagram-dependent quantity $B_D$ in \ref{standard bound} with a diagram-independent quantity. 

Burdened diagrams are the positive diagrams which can be sanded down into Balanced diagrams by smoothing $m$ crossings, for some non-negative integer $m$. This means that, as we saw in the case of fibered positive links, our bound \ref{standard bound} can be rewritten in terms of this $m$. By delving deeper into the structural properties of Balanced and Burdened diagrams, we are able to replace all of those diagram-dependent quantities (crossing number, number of $B$-circles and smoothing number $m$) with diagram-independent quantities (minimal degree of the Jones polynomial, number of link components, and leading coefficient of the Conway polynomial). We claim that any positive link with second Jones coefficient equal to $\pm 1$ or $\pm 2$ has a Burdened (or Oddly Burdened) diagram, and thus we are able to develop Theorem \ref{main result coeff 1 and 2}.

\section{Balanced and Burdened Diagrams}\label{Definitions}


\subsection{Balanced Diagrams}

Examples of each of the following appear in Figures \ref{Balanced type 0 example}, \ref{fig:Balanced type 1 example 2}, \ref{fig:Balanced type 2 example}, and \ref{fig:Oddly Balanced type 2 example}. A \say{hole} in a graph is an interior face.

\begin{definition}\label{Balanced type 0 def}
A \underline{\textbf{Balanced diagram of type $\mathbf{0}$}} is a (non-split) positive link diagram $D$ for which:
\begin{enumerate}
\item The reduced $A$-state graph of $D$ is a tree, and 
\item Every pair of $A$-circles share exactly $0$ or $2$ crossings.
\end{enumerate}
\end{definition}

\begin{definition}\label{Balanced type 1 def}
A \underline{\textbf{Balanced diagram of type $\mathbf{1}$}} is a (non-split) positive link diagram $D$ for which:
\begin{enumerate}
\item The reduced $A$-state graph of $D$ ($G_D'$) has exactly $1$ hole (interior face), and 
\item For each pair $v,w$ of $A$-circles, exactly one of the following is true:
\begin{enumerate}
    \item $v$ and $w$ share $0$ crossings, 
    \item $v$ and $w$ share exactly $1$ crossing, and the edge in $G_D'$ corresponding to that crossing is part of a cycle, or 
    \item $v$ and $w$ share exactly $2$ crossings, and the edge in $G_D'$ corresponding to those crossings is not part of a cycle. 
\end{enumerate}
\end{enumerate}
\end{definition}

\begin{definition}\label{Balanced type 2 def}
A \underline{\textbf{Balanced diagram of type $\mathbf{2}$}} is a (non-split) positive link diagram $D$ for which:
\begin{enumerate}
\item The reduced $A$-state graph of $D$ ($G_D'$) has exactly $2$ holes, and 
\item For each pair $v,w$ of $A$-circles, exactly one of the following is true:
\begin{enumerate}
    \item $v$ and $w$ share $0$ crossings, 
    \item $v$ and $w$ share exactly $1$ crossing, and the edge in $G_D'$ corresponding to that crossing is part of a cycle, or 
    \item $v$ and $w$ share exactly $2$ crossings, and the edge in $G_D'$ corresponding to those crossings is not part of a cycle. 
\end{enumerate}
\item An even number of edges in $G_D'$ are part of cycles. 
\end{enumerate}
\end{definition}

\begin{definition}\label{Oddly Balanced type 2 def}
If $D$ is a positive link diagram that satisfies criteria $(1)$ and $(2)$ in Definition \ref{Balanced type 2 def} but an odd number of edges in its reduced $A$-state graph are part of cycles, then we call $D$ an \underline{\textbf{Oddly Balanced diagram of type $\mathbf{2}$}}.
\end{definition}

\input{Figures/Balanced_type_0_example}
\input{Figures/Balanced_type_1_example_2}
\input{Figures/Balanced_type_2_example}
\input{Figures/Oddly_Balanced_example}

\subsection{Burdened Diagrams}

To generalize to a larger class of positive link diagrams, we consider \textit{Burdened diagrams}. We picture Burdened diagrams as would-be Balanced diagrams that are burdened down by extra crossings, making them violate criteria (b) or (c) in the above definitions. Examples of Burdened diagrams appear in Figure \ref{Burdened type 0 example} and Figure \ref{fig:Burdened type 1 example}.

\begin{definition}\label{Burdened type 0 def}
A \underline{\textbf{Burdened diagram of type $\mathbf{0}$}} is a (non-split) positive link diagram $D$ for which:
\begin{enumerate}
\item The reduced $A$-state graph of $D$ ($G_D'$) is a tree, and 
\item Every pair of $A$-circles share $0$ or \underline{at least} $2$ crossings.
\end{enumerate}
\end{definition}

\newpage

\begin{definition}\label{Burdened type 1 def}
A \underline{\textbf{Burdened diagram of type $\mathbf{1}$}} is a (non-split) positive link diagram $D$ for which:
\begin{enumerate}
\item The reduced $A$-state graph of $D$ ($G_D'$) has exactly $1$ hole (interior face), and 
\item For each pair $v,w$ of $A$-circles, exactly one of the following is true:
\begin{enumerate}
    \item $v$ and $w$ share $0$ crossings, 
    \item $v$ and $w$ share \underline{at least} $1$ crossing, and the edge in $G_D'$ corresponding to that crossing is part of a cycle, or 
    \item $v$ and $w$ share \underline{at least} $2$ crossings, and the edge in $G_D'$ corresponding to those crossings is not part of a cycle. 
\end{enumerate}
\end{enumerate}
\end{definition}

\begin{definition}\label{Burdened type 2 def}
A \underline{\textbf{Burdened diagram of type $\mathbf{2}$}} is a (non-split) positive link diagram $D$ for which:
\begin{enumerate}
\item The reduced $A$-state graph of $D$ ($G_D'$) has exactly $2$ holes, and 
\item For each pair $v,w$ of $A$-circles, exactly one of the following is true:
\begin{enumerate}
    \item $v$ and $w$ share $0$ crossings, 
    \item $v$ and $w$ share \underline{at least} $1$ crossing, and the edge in $G_D'$ corresponding to that crossing is part of a cycle, or 
    \item $v$ and $w$ share \underline{at least} $2$ crossings, and the edge in $G_D'$ corresponding to those crossings is not part of a cycle. 
\end{enumerate}
\item An even number of edges in $G_D'$ are part of cycles. 

\end{enumerate}
\end{definition}

\input{Figures/Burdened_type_0_example}
\input{Figures/Burdened_type_1_example}

\begin{definition}\label{Oddly Burdened type 2 def}
If $D$ is a positive link diagram that satisfies criteria $(1)$ and $(2)$ in Definition \ref{Burdened type 2 def} but  an odd number of edges in $G_D'$ are part of cycles, then we call $D$ an \underline{\textbf{Oddly Burdened diagram of type $\mathbf{2}$}}. 
\end{definition}

\begin{definition}
A \underline{\textbf{$\mathbf{k}$-Burdened diagram}} is a Burdened diagram whose reduced $A$-state graph contains one hole (interior face), and that hole is bounded by $k$ edges (where we ignore any cut edges that might \say{protrude} into the hole. 

A \underline{\textbf{$\mathbf{(k_1, k_2)}$-Burdened (or Oddly Burdened) diagram}} is a Burdened (or Oddly Burdened) diagram whose reduced $A$-state graph contains $2$ holes, with one hole bounded by $k_1$ edges (again ignoring any cut edges) and the other bounded by $k_2$ edges (note that the two sets of edges are not necessarily disjoint).  
\end{definition}

\begin{remark}
    For links with $k$-Burdened diagrams of type $1$, the value of $k$ is unique. That is, if $D$ is a $k$-Burdened diagram and $D'$ is a $k'$-Burdened diagram for $k\neq k'$, then $D$ and $D'$ do not represent the same link. (We will see later on that this follows from Lemma \ref{lead coeff k-Bur}, which says that this $k$ will be twice the leading coefficient of the Conway polynomial of the link.)  However, this is not true for diagrams of type $2$. In Figure \ref{fig:k-sequence is not invariant} we see an example of a $(6,4)$-Balanced diagram that is equivalent to an $(8,4)$-Oddly Balanced diagram. Not only are the values of $k_1$ and $k_2$ not unique, the set of knots with Balanced type $2$ diagrams and the set of knots with Oddly Balanced type $2$ diagrams are not mutually exclusive.
\end{remark}

\input{Figures/k-sequence_is_not_invariant}

In this paper, we develop the following bounds on the Jones polynomial of Burdened diagrams. 

\begin{theorem}\label{main result for Bur type 1}
    Let $L$ be a link with a Burdened diagram of type $1$. Then 
    $$\max \deg V_L \leq 4 \min \deg V_L + \frac{n-1}{2} + 2\lead \coeff \nabla_L -2,$$ where $V$ is the Jones polynomial, $\nabla$ is the Conway polynomial, and $n$ is the number of link components. 
\end{theorem}

\begin{theorem}\label{main result for Bur and Odd Bur type 2}
    Let $L$ be a link with a Burdened diagram of type $2$ or an Oddly Burdened diagram of type $2$. Then 
    $$\max \deg V_L \leq 4 \min \deg V_L + \frac{n-1}{2} + \lead \coeff \nabla_L,$$ where $V$ is the Jones polynomial, $\nabla$ is the Conway polynomial, and $n$ is the number of link components.
\end{theorem}

This allows to generalize to larger classes of positive links, because the second Jones coefficient tells us about the number of holes in the reduced $A$-state graph of a positive link diagram. 

\begin{theorem}\label{second Jones coeff counts holes}(Stoimenow)\cite{Stoimenow} 

Let $L$ be a positive link with positive diagram $D$. Then the second coefficient $V_1$ of the Jones polynomial satisfies:
$$(-1)^{n(L)-1}V_1=s(D)-1-\#(\text{pairs of Seifert circles that share at least one crossing}).$$ 

\end{theorem}

As Stoimenow notes, this means the absolute value of second Jones coefficient is exactly the first Betti number of the reduced Seifert graph. Since in positive diagrams the Seifert circles are the $A$-circles, this means that the absolute value of second coefficient of the Jones polynomial does indeed count the number of holes in the reduced $A$-state graph of a positive diagram $D$. 
Since the Jones polynomial is a diagram-independent link invariant, this means that \underline{every} positive diagram of a positive link with second Jones coefficient $0, \pm 1,$ or $\pm 2$ will have $0, 1, $ or $2$ (respectively) holes in its reduced $A$-state graph.  

\newpage 
\begin{corollary}\label{pos links must have bur diagrams}
    \begin{enumalph}
        \item Every reduced positive diagram of a link with second Jones coefficient equal to $0$ is a Burdened diagram of type $0$. 
        \item Every reduced positive diagram of a link with second Jones coefficient equal to $\pm 1$ is a Burdened diagram of type $1$. 
        \item Every reduced positive diagram of a link with second Jones coefficient equal to $\pm 2$ is either a Burdened diagram of type $2$ or an Oddly Burdened diagram of type $2$. 
    \end{enumalph}
\end{corollary}

\begin{proof}
    Let $L$ be a positive link with second Jones coefficient equal to $0, \pm 1, $ or $\pm 2$. Let $D$ be a reduced positive link diagram of $L$. By \ref{second Jones coeff counts holes}, the reduced $A$-state graph of $D$ must contain exactly $0, 1, $ or $2$ (respectively) holes. If any pair of $A$-circles $v,w$ share exactly one crossing, then the corresponding edge in the reduced $A$-state graph is part of a cycle (otherwise, the diagram $D$ is not reduced). 

    Thus $D$ indeed satisfies definition \ref{Burdened type 0 def}, \ref{Burdened type 1 def}, \ref{Burdened type 2 def}, or \ref{Oddly Burdened type 2 def}. 
\end{proof}

\begin{theorem}\label{main result coeff 1 and 2}
    Let $L$ be a positive link with $n$ link components, Jones polynomial $V_L$, and Conway polynomial $\nabla_L$. 
    \begin{enumalph}
       \item  If the second coefficient of $V_L$ is $\pm 1$, then 
    $$\max \deg V_L \leq 4 \min \deg V_L + \frac{n-1}{2} + 2\lead \coeff \nabla_L -2. $$
        \item If the second coefficient of $V_L$ is $\pm 2$, then 
    $$\max \deg V_L \leq 4 \min \deg V_L + \frac{n-1}{2} + \lead \coeff \nabla_L.$$ 
     \end{enumalph}
\end{theorem}

\begin{proof}
    This follows directly from Corollary \ref{pos links must have bur diagrams} and Theorems \ref{main result for Bur type 1} and \ref{main result for Bur and Odd Bur type 2}. 
\end{proof}

This can be used as a positivity obstruction. In Section \ref{Application section}, we present infinite families of knots which can be shown not to be positive using these theorems.

\section{Clasping Positive Diagrams}

In this section, we explore some properties of positive diagrams as we prepare to generalize the result of our last paper: 

\begin{theorem}\label{B=n for Balanced type 0}\cite{Buchanan_2022}
Let $D$ be a Balanced diagram of type $0$ with $n$ link components. Then 
$$B_D = n.$$
(The number of $B$-circles is equal to the number of link components. This justifies our use of the word Balanced.)
\end{theorem}

Whenever we refer to \say{an arc} of a diagram, we mean a portion of a strand that goes between two crossings, so an arc ends when it reaches any crossing, not just an undercrossing.

\begin{proposition}\label{even cycles}
    Every cycle in the reduced $A$-state graph of a positive diagram is made of an even number of edges. 
\end{proposition}

\begin{proof}
    It suffices to prove that this is true for any diagram $D$ whose reduced $A$-state graph $G$ contains no cut edges, since cut edges will never be part of a cycle. 

    Let $(v_1, v_2)$ be an edge in $G$. Since $G$ contains no cut edges, this edge is contained in some cycle $P=(v_1, v_2, \dots, v_r, v_1)$. We know that there are no duplicate edges in $G$, so $r\geq 3$. 

    In the diagram $D$, $A$-circles $A_1$ and $A_2$ (corresponding to vertices $v_1$ and $v_2$) are connected by a crossing, and $A_r$ and $A_1$ are also connected by a crossing. Either (without loss of generality) $A_2$  and $A_r$ are both drawn inside $A_1$, or none of the circles are drawn inside each other. 
    If two circles share a crossing and neither is contained in the other, then one has a clockwise orientation and the other has a counterclockwise orientation. If two circles share a crossing and one is contained in the other, then they have the same orientation. 
    
    Assume $A_1$ has counter-clockwise orientation. We can note this with a \say{$+$} and form a sequence that indicates the orientations of each circle: $$(A_1, A_2, A_3,  \dots, A_r, A_1) : (+, -, +, \dots, -, +)$$ if no circle contains the other, or $$(A_1, A_2, A_3, \dots, A_r, A_1): (+ , + , - , \dots, +, +)$$ if the circles are contained in $A_1$. In either case, notice that this sequence must have $r$ even for the rest of the entries to alternate between $+$ and $-$.  
\end{proof}

Because of this fact, the reduced $A$-state graph of any positive diagram of any positive link with second Jones coefficient $\neq 0$ must contain a cycle of length $\geq 4$. Meaning, the smallest possible Balanced diagram of type $1$ has at least four crossings. 

\input{Figures/clasp_example_1}

In Figure \ref{fig:clasp example 1} we see a Balanced diagram of type $1$ that we are able to transform into a Balanced diagram of type $0$ by adding a clasp. Let $A_1, A_2, A_3, A_4$ be the $A$-circles of this diagram, with corresponding vertices $v_1, v_2, v_3, v_4$ forming a cycle in the $A$-state graph. We take arcs in the diagram that belong to some $A$-circles $A_1$ and $A_3$, and interlock them. What happened in the $A$-state graph? All edges of the form $(v_i, v_3)$ are now of the form $(v_i, v_1)$, and we have added two copies of the new edge $(v_1, v_3)$. This clasping did not change the number of $A$-circles in $D$ (or the number of vertices of $G$), but it did change how the circles (and their corresponding edges in $G$) are arranged relative to one another. This is shown in a little more generality in Figure \ref{fig:clasp move 0}, which demonstrates how we can view this as one $A$-circle swallowing the other. This kind of clasping clearly does not change the number of link components. But more interestingly, we can see in Figure \ref{fig:clasp move 1} that this does not change the number of $B$-circles either. Hopefully, by repeatedly forming these clasps, we can kill a hole in a reduced $A$-state graph and preserve the number of link components and the number of $B$-circles as we do so. 

Since the $A$-state graph will tell us information to help us classify a diagram as Balanced or not, we would like to be able to think about performing \textit{clasp moves} directly from the information contained in the graph, and not have to have a diagram to start with.

\subsection{The Clasp Move}

Let $D$ be a reduced positive link diagram with the following properties: Every cut edge in the reduced $A$-state graph of $D$ corresponds to exactly two crossings in $D$, and every cycle edge in the reduced $A$-state graph corresponds to exactly one crossing in $D.$

Suppose there is a path $(v_1, v_2, v_3)$ in the reduced $A$-state graph that is part of a cycle, and that $v_2$ has degree $2$. Then there are arcs $a_1$ and $a_3$ in the diagram $D$ (that are part of the $A$-circles corresponding to vertices $v_1$ and $v_3$) that can be clasped together with positive crossings. We saw this in Figure \ref{fig:clasp example 1}.

Now, suppose that: (1) $(v_1, v_2, v_3)$ is part of a cycle in the reduced $A$-state graph, (2) Cutting edges $(v_1, v_2)$ and $(v_2, v_3)$ disconnects the graph, and (3) this disconnects it so that the component containing $v_2$ is a tree. So, this means that $v_2$ may not have degree $2$ in the graph, but from the perspective of any vertex in the graph that is not part of that tree rooted at $v_2$, $v_2$ might as well have degree $2$ -- every path from $v_2$ to any other vertex (that is not part of the tree) must use edge $(v_1, v_2)$ or $(v_2, v_3)$. Then, as before, we can clasp together arcs in the link diagram corresponding to $v_1$ and $v_3$. This is shown in Figure \ref{fig:v2 cycle degree 2}, and will be called a \textit{clasp move}.

If a reduced positive link diagram $D$ satisfies the criteria listed above for performing a clasp move, then we say that $D$ is \textit{claspable}:

\begin{definition} \label{claspable def}
A link diagram $D$ is \textbf{\underline{claspable}} if it satisfies the following: \begin{enumalph}
    \item $D$ is a reduced positive link diagram 
    \item Every cut edge in the reduced $A$-state graph of $D$ corresponds to exactly two crossings in $D$, and every cycle edge in the reduced $A$-state graph corresponds to exactly one crossing in $D$
    \item The reduced $A$-state graph contains a cycle $(v_1, v_2, v_3, \dots, v_1)$ such that: 
    \begin{enumalph}
        \item Edges $(v_1, v_2)$ and $(v_2, v_3)$ form a cut set in the graph
        \item Cutting those two edges disconnects the graph so that the component containing $v_2$ is a tree
    \end{enumalph}
\end{enumalph}

\end{definition}

\begin{definition} \label{clasp move def}
We perform a \textbf{\underline{clasp move}} on a claspable diagram $D$ (and its associated $A$-state graph $G$) by clasping arcs $a_1$ and $a_3$ together with positive crossings (transferring all edges incident to $v_3$ to be incident to $v_1$, and then adding in two copies of the edge $(v_1, v_3)$). 
\end{definition}

\input{Figures/clasp_move_0}
\input{Figures/clasp_move_1}
\input{Figures/v2_cycle_degree_2}

\begin{proposition} \label{leaves dont matter}
    Let $D$ be a claspable diagram, and let $D'$ be the diagram obtained by performing a clasp move. Let $G$ be the $A$-state graph of $D$, and let $G'$ be that of $D'$. If vertex $v_i$ is incident to only one vertex in $G$, then $v_i$ is still incident to only one vertex in $G'$. 
\end{proposition}

\begin{proof}

    Suppose for contradiction that $v_i$ is incident to more than one vertex in $G'$. 
      The clasp move takes all edges in the $A$-state graph $G$ of the form $(v_3, v_j)$ and makes them edges of the form $(v_1, v_j)$ in $G'$, and all other edges remain exactly the same. 

      So the only way a vertex could have a different set of neighbors in $G'$ than in $G$ is if the vertex itself were $v_3$, or if one of its neighbors were $v_3$. Since by assumption $v_3$ is part of a cycle in the reduced $A$-state graph and $v_i$ is not, it must be the case that $v_i$ was incident to $v_3$ in $G$. But then now $v_i$ is only incident to $v_1$ in $G'$. 
\end{proof}

\begin{lemma}\label{clasp moves preserve B and n}
    Clasp moves do not change the number of link components or the number of $B$-circles in a diagram: If $D'$ is obtained by performing a clasp move on $D$, then $B_{D'} = B_D$ and $n(D')=n(D)$. 
\end{lemma}

\begin{proof}
Obviously, clasp moves do not change the number of link components in a diagram, and in Figure \ref{fig:clasp move 0} we saw that a clasp moves preserve the number of $A$-circles. Now in Figure \ref{fig:clasp move 1}, we see that clasp moves also preserve the number of $B$-circles, regardless of if the arcs belonged to the same $B$-circle or to different $B$-circles in the original diagram $D$.
\end{proof}

\newpage

\subsection{Clasp Moves on a Balanced diagram of type $\mathbf{1}$}

\begin{proposition}
    Every $k$-Balanced diagram $D$ is claspable. 
\end{proposition}

\begin{proof}

    This follows immediately from Definition \ref{Balanced type 1 def}.
\end{proof}

\begin{proposition}\label{k-Bal are forever claspable}
If $k=4$, then performing a clasp move on a $4$-Balanced diagram results in a Balanced diagram of type $0$. 
   If $k\geq 6$, then performing a clasp move on a $k$-Balanced diagram results in a $(k-2)$-Balanced diagram. 
\end{proposition}

\begin{proof}
We know that $D$ is claspable. And, we note we can perform a clasp at any $(v_1, v_2, v_3)$ path segment of the cycle in this reduced $A$-state graph. Since Proposition \ref{leaves dont matter} tells us that a leaf in the reduced $A$-state graph of $D$ will correspond to a leaf in the reduced $A$-state graph of $D'$, it suffices to prove our claim for the case where the reduced $A$-state graph contains no leaves, and therefore contains only cycle edges. 
We already saw the case of $k=4$ demonstrated in Figure \ref{fig:clasp example 1}. 
    The proof of the second statement is demonstrated in Figure \ref{fig:clasping k-balanced}, where the dashed edge represents either one edge, or a path of three new edges (adding two vertices in the process, so that all vertices in the dashed path have degree $2$). We perform the clasp move, and obtain a $(k-2)$-Balanced diagram.  

\end{proof}

\input{Figures/clasping_k-balanced}

\begin{theorem}\label{Bal type 1 B=n}
    Let $D$ be a Balanced diagram of type $1$. Then the number of $B$-circles in $D$ is equal to the number of link components: 
    $B_D = n.$
\end{theorem}

\begin{proof}
    If $D$ is a Balanced type $1$ diagram, then it is $k$-Balanced for some even integer $k$. By induction on $k$, using the result of Proposition \ref{k-Bal are forever claspable}, we have that any $k$-Balanced diagram $D$ can be turned into a Balanced diagram $D'$ of type $0$ by performing $\frac{k-2}{2}$ clasp moves. 

    Then 
    \begin{align*}
        B_{D} &= B_{D'} \text{ by Lemma \ref{clasp moves preserve B and n}},\\
        & = n(D') \text{ by Theorem \ref{B=n for Balanced type 0}, since $D'$ is Balanced type $0$},\\
        & = n(D) \text{ by Lemma \ref{clasp moves preserve B and n}}.
    \end{align*}

\end{proof}

\subsection{Clasp Moves on a Balanced diagram of type $\mathbf{2}$}

The general idea will be that if $D$ is a claspable diagram that can be transformed (via clasp moves) into a diagram for which we know that $B=n$, then we have that $B=n$ in $D$ also. 

\begin{theorem} \label{Bal type 2 B=n}
    Let $D$ be a Balanced diagram of type $2$. Then $B_D = n$. 
\end{theorem}

\begin{proof}

    Let $D$ be a $(k_1, k_2)$-Balanced diagram. We claim that $D$ is claspable, and that we can choose to perform clasp moves on $D$ such that the resulting diagram $D'$ is also claspable. In doing so we can transform $D$ into a Balanced diagram of type $1$ or type $0$, which we already know has the property that $B=n$. 

    By definition of Balanced type $2$, $D$ will satisfy the first two criteria of the claspable definition (Definition \ref{claspable def}). As before, by Proposition \ref{leaves dont matter} it suffices to prove the rest of our claims on a diagram $D$ whose reduced $A$-state graph does not contain any leaves. 

    Let $x$ be the number of edges in the reduced $A$-state graph that border both holes. From the definition of Balanced type $2$ (Definition \ref{Balanced type 2 def}), this means we have an even number of edges that are part of cycles, and it follows that $x$ must also be even. We proceed by considering the three cases of $x=0$, $x=2$, and $x \geq 4$. \\

\underline{Case: $x=0$}

    First, suppose $x=0$. Then either the reduced $A$-state graph of $D$ looks like two cycles with one shared vertex $v$, or it looks like two cycles each with with a vertex ($v$ or $w$) such that there is a single path connecting them. In either situation, we cut off the second cycle (and the connecting path, if it exists) and focus on the first cycle of length $k_1$. Let $v=v_1$, and then choose $v_2$ in either direction along the cycle and we will have a cycle-segment $(v_1, v_2, v_3)$ that satisfies the remaining criteria of being claspable. By performing $\frac{k_1-2}{2}$ clasp moves, we transform the diagram into a Balanced diagram of type $0$. We get a Balanced diagram $D'$ of type $1$ when we reattaching the second cycle that we cut off earlier. Meaning, we could have chosen to perform the clasp move at $(v=v_1, v_2, v_3)$ at the beginning, and obtained the very same $D'$ without needing to cut. Since clasp moves do not change the number of the $B$-circles or the number of link components, by our previous case for Balanced type $1$ diagrams (Theorem \ref{Bal type 1 B=n}) we see that $$B_D = B_{D'} = n(D') = n(D).$$

    \input{Figures/clasping_Bal_type_2_x=2}

\underline{Case: $x=2$}

    Now that $x> 0$, the restriction of the reduced $A$-state graph having no leaves means that the reduced $A$-state graph and the $A$-state graph are identical. 
    
    In Figure \ref{fig:clasping Bal type 2 x=2}, we see on left the general form that every such $A$-state graph must have, where a dashed edge represents either a single edge or a string of an odd number of edges and even number of vertices, each with degree $2$. 
 In the middle, we have connected $v_1$ and $v_3$ with an orange line, indicating where our clasp will occur. On the right we show the result of clasping.
 
 If each dashed line represents one edge each, the graph on the right corresponds to a Balanced diagram of type $0$. If exactly one of the dashed lines represents a path of at least three edges, then the graph on the right corresponds to a Balanced diagram of type $1$. If both of the dashed lines represent paths of at least three edges, then the graph corresponds to a Balanced diagram of type $2$ in which $x=0$. We have already shown that for any of these possibilities, we have $B=n$. So since clasp moves do not change the number of link components or the number of $B$-circles (Proposition \ref{clasp moves preserve B and n}), it follows that $B=n$ in our original diagram $D$, corresponding to the graph on the left. \\

    \input{Figures/clasping_Bal_type_2_x_geq_4}

    \underline{Case: $x \geq 4$}

   We see in Figure \ref{fig:clasping Bal type 2 x geq 4} that we can choose to perform a clasp move on the diagram that will transform it into a Balanced diagram of type $2$, and have $x$ has decreas by $2$ in the process. So, we can perform $\frac{x-2}{2}$ total clasp moves to transform it into a Balanced diagram of type $2$ in which $x=2$. 
    By our usual clasping argument, this means that $B=n$ for Balanced diagrams of type $2$ when $x\geq 4$. 

    And thus $B=n$ for any Balanced diagram of type $2$. 
\end{proof}

\subsection{Clasp Moves on an Oddly Balanced diagram of type $\mathbf{2}$}

We proceed in much the same manner for Oddly Balanced diagrams. However, we run into an obstacle. In Figure \ref{fig:Oddly Balanced B not equal to n} we see a diagram of knot $7_4$ which is the smallest possible Oddly Balanced diagram. But this is a diagram with $1$ link component and $3$ $B$-circles, so clearly we cannot say that $B=n$ for Oddly Balanced diagrams.  However, we will see that this is as bad as it gets, and while the number of $B$-circles might not always be equal to the number of link components, it will only ever be a little bit off-balance. 

\input{Figures/Oddly_Balanced_B_not_equal_to_n}

\begin{theorem}\label{Oddly Bal B almost equal n}
    Let $D$ be an Oddly Balanced diagram of type $2$. Then $B=n$ or $B= n \pm 2$. 
\end{theorem}

\begin{proof}
    As in the last section, we show that all Oddly Balanced diagrams of type $2$ are claspable, and that by performing a sequence of clasp moves, we can shrink the holes. An Oddly Balanced diagram has an odd number of cycle edges in its reduced $A$-state graph, and therefore has an odd number of edges $x$ that bound both holes in the reduced $A$-state graph. We look at two special cases first, and then look at the three general cases of $x=1$, $x=3$, and $x \geq 5$.  As before, by Proposition \ref{leaves dont matter} it suffices to look at the $A$-state graphs of diagrams $D$ whose reduced $A$-state graphs have no leaves. (That is, it suffices to look at $A$-state graphs $G$ where every edge is part of a cycle.)\\

    \input{Figures/clasping_Oddly_Bal_type_2_special_case_x=1}
\input{Figures/clasping_Oddly_Bal_type_2_special_case_x=3}

\underline{Two Special Cases}

In Figure \ref{fig:clasping Oddly Bal type 2 special case x=1} we see a graph $G$ that is the $A$-state graph of a $(4,4)$-Oddly Balanced diagram $D$. While it is claspable, we see that any choice of where to clasp the diagram results in a Burdened diagram $D'$ of type $1$ with one extra edge preventing it from being Balanced. By Corollary \ref{bound on B-circles Bur} we have that $B_{D'} = n(D')$ or $B_{D'} = n(D') \pm 2$. Since clasp moves do not change the number of link components or the number of $B$-circles, this means that $B_D = B_{D'} = n(D') = n(D)$, or $B_D = B_{D'} = n(D') \pm 2 = n(D) \pm 2$. 

In Figure \ref{fig:clasping Oddly Bal type 2 special case x = 3}, we see a graph $G$ that is the $A$-state graph of a $(4, 6)$-Oddly Balanced diagram $D$. While it is claspable, any choice of where to clasp the diagram results in a Burdened diagram of type $1$ with one edge preventing it from being Balanced. In the same manner as the first special case, it follows by Corollary \ref{bound on B-circles Bur} that the number of $B$-circles in the resulting diagram $D'$ differs from the number of link components by at most $2$. And the result follows.\\

\underline{Case: $x= 1$}

If $x=1$ and (without loss of generality) $k_2 \geq 6$, then we can perform a clasp move to transform the $(k_1, k_2)$-Oddly Balanced diagram into a $(k_1, k_2-2)$-Oddly Balanced diagram. As in the case of Balanced type $2$ diagrams, we see (Figure \ref{fig:clasping Oddly Bal type 2 x=1 can do it}) that the resulting diagram is still claspable. We can perform $\frac{k_2-4}{2}$ clasp moves to tranform the diagram into a $(k_1, 4)$-Oddly Balanced diagram, and then perform $\frac{k_1-4}{2}$ clasp moves to transform it into a $(4,4)$-Oddly Balanced diagram, $D'$. Then $D$ had the same number of link components and $B$-circles as $D'$, and we just showed that $D'$ has $B=n$ or $B= n \pm 2$, so the result follows. \\

\input{Figures/clasping_Oddly_Bal_type_2_x_=_1_can_do_it}

\underline{Case: $x=3$}

If $x=3$ and $k_1, k_2 \geq 6$, then we can perform a clasp move (shown in Figure \ref{fig:clasping Oddly Bal type 2 x = 3}) that transforms the diagram into a $(k_1-2, k_2-2)$-Oddly Balanced diagram with $x=1$. Whatever value we have for $k_1-2$ and $ k_2-2$, we have already dealt with it in the previous cases, and the result follows. \\

\input{Figures/clasping_Oddly_Bal_type_2_x=3}

\input{Figures/clasping_Oddly_Bal_type_2_x_geq_5}

\underline{Case: $x=5$}

If $x=5$, observe that at least one of $k_1, k_2$ is $\geq x+1$. (If they were both equal to $x+1$, then we would have a repeated edge that is part of a nontrivial cycle, which violates the definition of an Oddly Balanced diagram.) Figure \ref{fig:clasping Oddly Bal type 2 x geq 5} shows a how we can perform a clasp move to transform the diagram into a $(k_1 -2, k_2-2)$-Oddly Balanced diagram in which $x$ has decreased by $2$. If we instead performed $\frac{x-3}{2}$ clasp moves of this form, then we get an Oddly Balanced diagram with $x=3$, and this was the previous case. Again, since clasp moves do not change the number of link components or the number of $B$-circles, the result follows. 

Now we have proven the result for all Oddly Balanced diagrams of type $2$. 

\end{proof}

\section{Burdened and Oddly Burdened Diagrams}

In every Burdened (or Oddly Burdened) diagram of type $0, 1,$ or $2$, we can smooth away some crossings to obtain a Balanced (or Oddly Balanced) diagram of type $0, 1,$ or $2$. In Figure \ref{fig:smoothing number vs burdening number} we see an example in which smoothing just one crossing transforms a Burdened diagram of type $1$ into a Balanced diagram of type $0$. In general, we may not always be able to quickly obtain a Balanced diagram of a smaller type. What we \textit{can} always do, though, is smooth away some of the crossings of a Burdened (or Oddly Burdened) diagram of type $r$ to obtain a Balanced (or Oddly Balanced, respectively) diagram of the same type $r$. 

\begin{definition}
    The number of crossings that must be smoothed away in a Burdened (or Oddly Burdened) diagram of type $r$ to produce a Balanced (or Oddly Balanced) diagram of the same type $r$ is called the \underline{\textbf{burdening number}}. It counts the number of crossings that burden the diagram, preventing it from being a Balanced (or Oddly Balanced) diagram of the same type. We denote the burdening number by $m$. 
\end{definition}

  The burdening number is the least upper bound on the number of crossings that must be smoothed in order to obtain a Balanced (or Oddly Balanced) diagram of any type. It is also the greatest number of crossings than can be smoothed away to result in a Balanced (or Oddly Balanced) diagram. 

 \input{Figures/smoothing_number_vs_burdening_number}

\begin{corollary} [Corollary to Theorems \ref{Bal type 1 B=n} and \ref{Bal type 2 B=n}]
\label{bound on B-circles Bur}
    Let $D$ be a Burdened diagram of type $1$ or $2$. Let $B_D$ be the number of $B$-circles in the diagram, let $n(D)$ be the number of link components, and let $m$ be the Burdening number. Then 
    $$B_D \leq n(D) + 2m.$$
\end{corollary}

\begin{proof}
Let $D$ be a Burdened diagram of type $r$. Smooth $m$ crossings to obtain a Balanced diagram $D'$ of the same type $r$. Each smoothing changes the number of $B$-circles in the diagram by $\pm 1$ and changes the number of link components by $\pm 1$. So, $B_D \leq B_{D'} + m $ and $n(D) \leq n(D') + m$. (This follows from Proposition 2.9 in \cite{Buchanan_2022}.) Then by Theorems  \ref{Bal type 1 B=n} and \ref{Bal type 2 B=n}, we can say that 
\begin{align*}
    B_D &\leq B_{D'} + m \\
    & = n(D') + m \\
    &\leq n(D) + 2m.
\end{align*}

\end{proof}

\begin{corollary}
\label{bound on B-circles Oddly Bur}
    Let $D$ be an Oddly Burdened diagram of type $2$. Let $B_D$ be the number of $B$-circles, let $n(D)$ be the number of link components, and let $m$ be the burdening number. Then 
    $$B_D \leq n(D) + 2(m + 1).$$
\end{corollary}

\begin{proof}

By the same argument as the proof of Corollary \ref{bound on B-circles Bur}, but using Theorem \ref{Oddly Bal B almost equal n} for Oddly Balanced diagrams, we can say that 
\begin{align*}
    B_D &\leq B_{D'} + m \\
    & = {\begin{cases}
        n(D') + 2 + m \\
         n(D') + m\\
         n(D') - 2 +m\\
    \end{cases}}\\
    &\leq n(D) + 2(m + 1).
\end{align*}

\end{proof}

\begin{corollary}\label{improved standard bound Bur} \label{improved standard bound Oddly Bur}
    For any Burdened type $0, 1, $ or $2$ diagram $D$ of a link $L$ with $n$ link components and burdening number $m$,
    $$\max \deg V_L \leq c(D) + \frac{n-1}{2} + m.$$ 

    For any Oddly Burdened type $2$ diagram $D$,
    $$\max \deg V_L \leq c(D) + \frac{n-1}{2} + m + 1.$$ 
\end{corollary}

\begin{proof}
    For Burdened diagrams, the follows directly from \ref{standard bound} and Corollary \ref{bound on B-circles Bur}: 

    $$ \max \deg V_L \leq c(D) + \frac{B_D-1}{2}  \leq c(D) + \frac{n + 2m - 1}{2} = c(D) + \frac{n-1}{2}+m. $$

For Oddly Burdened diagrams, the result follows directly from \ref{standard bound} and Corollary \ref{bound on B-circles Oddly Bur}: 

    $$ \max \deg V_L \leq c(D) + \frac{B_D-1}{2}  \leq c(D) + \frac{n + 2(m +1) - 1}{2} = c(D) + \frac{n-1}{2}+m+1. $$
\end{proof}

We started off with the standard bound on the maximum degree of the Jones polynomial of a link with positive diagram $D$: $\max \deg V \leq c(D) +\frac{B_D - 1}{2}$ (\ref{standard bound}). For links with Balanced diagrams, we now have replaced this diagram-dependent quantity $B_D$ with the diagram-independent $n$. But we still have the diagram-dependent quantities $c(D)$ and $m$ to contend with. In \cite{Buchanan_2022}, we found that in every Burdened diagram of type $0$, we can express the burdening number as $m= 4\min \deg V - c(D)$. This allowed us to replace $c(D)$ and $m$ with the diagram-independent quanitity $4\min \deg V$ and find a bound on the maximum degree of the Jones polynomial for fibered positive links. 

In the following subsections, we find similar expressions for the burdening number of Burdened diagrams of type $1$, Burdened diagrams of type $2$, and Oddly Burdened diagrams of type $2$. We will see that these expressions allow us to replace $c(D)$ and $m$ in the bound given by Corollary \ref{improved standard bound Bur} with diagram-independent quantities.

\subsection{Burdened Diagrams of Type $1$}

\begin{proposition}\label{m for coeff pm 1}
Let $D$ be a $k$-Burdened diagram of a link $L$. Then the burdening number $m$ (the number of crossings that must be smoothed to transform $D$ into a $k$-Balanced diagram) can be expressed as 
    $$m = 4 \min \deg V_{L} - c(D) + k -2.$$
\end{proposition}

\begin{proof}

We first consider a $k$-Balanced diagram $D$. Let $G$ be the $A$-state graph of a $k$-Balanced diagram $D$ and $G'$ be the reduced $A$-state graph of $D$. Since $D$ is $k$-Balanced, there is a hole in $G'$ that is bounded by $k$ edges. Since $G'$ is a graph with one cycle, the number of edges in $G'$ is $A_{D}-1 +1 = A_{D}$. By definition of Balanced, every edge in $G'$ that is part of the cycle corresponds to one edge in $G$, and any other edge in $G'$ corresponds to two edges in $G$. Thus the number of edges in $G$ is $k + 2(A_{D}-k)=2A_{D} - k$. As the number of crossings in $D$ is equal to the number of edges in $G$, this means that 
\begin{equation}\label{number of crossings in k-Balanced}
    c(D) = 2A_{D} - k.
\end{equation}

Now, let $D$ be a $k$-Burdened diagram with $m$ crossings that must be smoothed to obtain $k$-Balanced $D'$. 
As such, we know that $c(D) = c(D')+m$. Since smoothing those crossings does not change the number of $A$-circles, we also know that $A_D = A_{D'}$. It follows from \ref{number of crossings in k-Balanced} that 
\begin{equation}\label{number of crossings in k-Burdened}
    c(D)=2A_{D} - k + m. 
\end{equation}

Since $D$ is a positive diagram, we have from \ref{min degree pos} that $4\min \deg V_L = 2(c(D) - A_D + 1)$. It follows from \ref{number of crossings in k-Burdened} that 
\begin{align*}
    m & = c(D) - 2A_D + k \\
    & = 4\min \deg V_{L} - c(D) + k -2.
\end{align*}
\end{proof}

Putting this information into the bound in Corollary \ref{improved standard bound Bur} will eradicate $c(D)$ and $m$, but will introduce another diagram-dependent quantity: $k$. However, in section \ref{lead conway coeff coeff 1 section} we prove that

\begin{lemma*}[$\mathbf{\ref{lead coeff k-Bur}}$]
    For a link $L$ with a $k$-Burdened diagram $D$ and Conway polynomial $\nabla_L$,  $$\lead \coeff \nabla_L = \frac{k}{2}.$$
\end{lemma*}

So, we can replace the diagram-dependent quantity $k$ with the diagram-dependent quanitity $2\lead\coeff\nabla_L$.

\begin{theorem} \label{V_1 = 1 bound}
    Let $D$ be a Burdened diagram of type $1$. Then 

    $$\max\deg V_D \leq 4\min \deg V_D + \frac{n-1}{2} + 2\lead \coeff \nabla_D - 2,$$ where $V_D$ is the Jones polynomial and $\nabla_D$ is the Conway polynomial.
\end{theorem}

\begin{proof}

    Let $D$ be a Balanced type $1$ diagram. Then $D$ is a $k$-Balanced diagram for some $k$, we can improve our standard bound \ref{standard bound} as follows: 

       \begin{align*}
        \max \deg V_D &\leq c(D) + \frac{B_D - 1}{2} &&\text{ (by \ref{standard bound})}\\
        & = c(D) + \frac{n-1}{2} + m &&\text{ (by Corollary \ref{improved standard bound Bur})}\\
        & = c(D) + \frac{n-1}{2}+ 4\min \deg V_L - c(D) + k - 2 &&\text{ (by Proposition \ref{m for coeff pm 1})}\\
        & = 4\min \deg V_L + \frac{n-1}{2} + 2\lead \coeff \nabla_L - 2 &&\text{ (by Lemma \ref{lead coeff k-Bur})}.
    \end{align*}
\end{proof}

\subsection{Burdened (and Oddly Burdened) Diagrams of Type $2$}

This section finds a similar bound on the maximum degree of the Jones polynomial for Burdened (and Oddly Burdened) diagrams of type $2$. Using our notation introduced earlier, we are considering $(k_1, k_2)$-Burdened (and Oddly Burdened) diagrams. 

\begin{proposition}\label{m for coeff 2}
    Let $D$ be a $(k_1, k_2)-$Burdened diagram. Then $m$, the number of crossings that must be smoothed to transform $D$ into a $(k_1,k_2)-$Balanced diagram, can be expressed as 
    $$m = 4\min \deg V_L - c(D) + (k_1 + k_2 -x -4),$$ 
    where $x$ is the number of of edges shared by the two holes in the reduced $A$-state graph of $D$. 
\end{proposition}

\begin{proof}
As in the Burdened type $1$ case, we begin by finding the number of crossings in a Balanced (or Oddly Balanced) diagram first.

    Let $D$ be a $(k_1, k_2)-$Balanced (or Oddly Balanced) diagram. Then its reduced $A$-state graph $G'$ contains exactly $2$ holes, where the total number of edges in $G'$ is to $A_D +1$ (the number of vertices of the graph plus one). One hole is bounded by $k_1$ edges, the other by $k_2$ edges, and there are $x$ edges shared between them (where $x\geq 0$). So, the number of edges that are part of a cycle is $k_1 + k_2 - x$, and the number of cut edges is $A_D + 1 - (k_1 + k_2 -x).$
     
     By definition of Balanced (and Oddly Balanced), every edge that is part of a cycle in the reduced $A$-state graph corresponds to one crossing in the diagram $D$, and every cut edge in $G'$ corresponds to exactly $2$ crossings in $D$. 

    Thus
    \begin{align*} \label{number of crossings in 2 Balanced}
   c(D) &= \hspace{27pt} 2(\#\text{ cut edges}) \hspace{27pt} + (\#\text{ edges involved in a cycle}) \\
   &= 2\Big(A_D + 1 - (k_1 + k_2 -x)\Big) + (k_1 + k_2 - x) \\
   &= 2(A_D+1) - (k_1 + k_2 - x). 
    \end{align*}

That is the number of crossings in a $(k_1, k_2)$-Balanced (or Oddly Balanced) diagram. Now, let $D$ be $(k_1, k_2)$-\textit{Burdened} diagram. We know it can be smoothed into some Balanced (or Oddly Balanced) diagram $D'$. Thus $A_{D'}=A_D$ and $c(D) = c(D') + m,$  where $m$ is the burdening number. Then

    \begin{equation}\label{number of crossings in 2 Burdened}
        c(D) = 2(A_D + 1) - (k_1 + k_2 -x) + m.
    \end{equation}

    Since $D$ is a positive diagram, we have from \ref{min degree pos} that $4 \min \deg V_L = 2(c(D) - A_D +1)$. It follows from \ref{number of crossings in 2 Burdened} that 

    \begin{align*}
        m & = c(D) - 2A_D - 2 + (k_1 + k_2 -x)\\
        & = 4\min \deg V_L - c(D) + (k_1 + k_2 - x) - 4. 
    \end{align*}

\end{proof}

\begin{corollary}
\label{bound with k's Bur and Oddly Bur}
   Let $D$ be a $(k_1, k_2)-$Burdened diagram of with $n$ link components. Then 
 $$\max \deg V_D \leq 4\min \deg V_D + \frac{n-1}{2} + k_1 + k_2 - x - 4,$$ where $x$ is the number of boundary edges shared by both of the two holes in the reduced $A$-state graph of $D$. If $D$ is a $(k_1, k_2)$-Oddly Burdened diagram with $n$ link components, then 
 $$\max \deg V_D \leq 4\min \deg V_D + \frac{n-1}{2} + (k_1 + k_2 - x - 4) + 1.$$ 
\end{corollary}

\begin{proof}

    Let $D$ be a $(k_1, k_2)$-Burdened diagram. Then inserting the information from the above Proposition \ref{m for coeff 2} and Lemma \ref{bound on B-circles Bur} into \ref{standard bound}, we obtain: 
       \begin{align*}
        \max \deg V_D &\leq c(D) + \frac{B_D - 1}{2}\\
        & \leq c(D) + \frac{n + 2m -1}{2} \text{(by Prop. \ref{bound on B-circles Bur})}\\
        & = c(D) + \frac{n-1}{2}+ 4\min \deg V_L - c(D) + k_1 + k_2 -x -4 \text{(by Prop. \ref{m for coeff 2})}\\
        & = 4\min \deg V_L + \frac{n-1}{2} + k_1 + k_2 - x - 4.
    \end{align*}
\end{proof}

If instead $D$ is a $(k_1, k_2)$-Oddly Burdened diagram, then Proposition \ref{bound on B-circles Oddly Bur} tells us that $B_D \leq 2(m+1)$ (instead of $B_D \leq 2m$ as in the Burdened case), and the result for Oddly Burdened follows from the same argument.

\begin{theorem}\label{lead coeff for Bur type 2}
    Let $D$ be a $(k_1, k_2)$-Burdened diagram with Conway polynomial $\nabla.$ Then 
    $$\lead \coeff \nabla = \frac{k_1k_2 -x^2}{4},$$ where $x$ is the number of edges that bound both of the holes in the reduced $A$-state graph.
\end{theorem}

\begin{theorem}\label{lead coeff for type 2 Oddly Bur}
    Let $D$ be a $(k_1, k_2)$-Oddly Burdened diagram with Conway polynomial $\nabla.$ Then 
    $$\lead \coeff \nabla = \frac{k_1k_2 -x^2+1}{4},$$ where $x$ is the number of edges that bound both of the holes in the reduced $A$-state graph.
\end{theorem}

The proof of these appears in section \ref{lead conway coeff coeff 2 section}.

\begin{lemma}\label{result for type 2}
    Let $D$ be a $(k_1, k_2)$-Burdened diagram or a $(k_1, k_2)$-Oddly Burdened diagram. Then 
    $$\max \deg V \leq 4 \min \deg V + \frac{n-1}{2} + \lead \coeff \nabla. $$
\end{lemma}

\begin{proof}

\underline{Case 1: Burdened diagram}\\

Let $D$ be a $(k_1, k_2)$-Burdened diagram. 

By Corollary \ref{bound with k's Bur and Oddly Bur} and Theorem \ref{lead coeff for Bur type 2}, it suffices to show that 
\begin{equation}\label{suffices to show inequal Bur}
    k_1 + k_2 - x - 4 \leq \frac{k_1k_2 - x^2}{4}.
\end{equation}

To aid our computations, we make the following substitutions: let $y = k_1-x$ and $z = k_2 -x$. Observe that then $x,y,z$ are all even. We claim that 
\begin{align}\label{bur comp}
    0 &\leq  xy + xz + yz - 4(x+y+z) + 16\\
    & = (x-2)(y-2) + (x-2)(z-2) + (y-2)(z-2) + 4.\notag
\end{align} 

The equality on the second line is clear, and it remains to show the inequality on the first line. At most one of $x,y,z$ can be $0$. If $x=0$, then $y\geq 4$ and $z\geq 4$, so the right side of \ref{bur comp} is $yz - 4(y+z) + 16 = (y-4)(z-4) \geq 0$, as desired. If none of $x,y,z$ are $0$, then all are $\geq 2$, so the right side second line of \ref{bur comp} is at least $4$, and the inequality is satisfied. 

Inequality \ref{bur comp} is equivalent to 
\begin{equation}
    x+y+z-4 \leq \frac{xy+xz+yz}{4},
\end{equation} which, via our substitution, is exactly Inequality \ref{suffices to show inequal Bur}, and thus we have proved the Burdened case.\\

\underline{Case 2: Oddly Burdened diagram}\\

Let $D$ be a $(k_1, k_2)$-Oddly Burdened diagram. By Corollary \ref{bound with k's Bur and Oddly Bur}, it suffices to show that 
\begin{equation}\label{suffices to show inequal Oddly Bur 2}
    (k_1 + k_2 - x - 4) + 1 \leq  \frac{k_1k_2 - x^2 +1 }{4}.
\end{equation}

As before, we make the following substitutions: let $y = k_1-x$, and let $z = k_2 -x$. Observe that in this case, $x,y,z$ are all odd. We claim that 
\begin{align}\label{oddly bur comp}
    0 &\leq xy + xz + yz - 4(x+y+z) + 13\\
    & =(x-2)(y-2) + (x-2)(z-2) + (y-2)(z-2) + 1. \notag
\end{align} 

The equality on the bottom line is clear, and we must prove the inequality on the top. At most one of $x,y,z$ can be $1$, otherwise we would have a \say{hole} bounded in the reduced $A$-state graph that is by only two edges, which is impossible (this would mean the graph had duplicate edges). If $x=1$, then $y\geq 3$ and $z\geq 3$, so the right side of \ref{oddly bur comp} is $y + z + yz - 4(1+y+z) + 13 = (y-3)(z-3) \geq 0$, as desired. If none of $x,y,z$ are $1$, then all are $\geq 3$, so the right side of \ref{oddly bur comp} is at least $4$, and the inequality is satisfied.

Inequality \ref{oddly bur comp} is equivalent to 
\begin{equation}
    (x+y+z-4)+1 \leq \frac{xy+xz+yz + 1}{4},
\end{equation} which, via our substitution, is exactly Inequality \ref{suffices to show inequal Oddly Bur 2}, and thus we have proved the Oddly Burdened case. 

\end{proof}

\section{Conway Polynomial}

In this section, we gather whatever tools we can about positive and positive links, in preparation to find the leading Conway coefficient of a generic Burdened link. We build up resources here that will always allow us to \say{prune away} any leaves (or other foliage), and instead find that we can compute the leading Conway coefficient just from the parts of the diagram that correspond to cycle edges in the reduced $A$-state graph.

\begin{proposition}[Cromwell]\cite{Cromwell}\label{Cromwell pos Conway degree}
    Let $L$ be a positive link with positive diagram $D$. Then 
    $$1- \chi(D) = 1-\chi(L) = \max\deg \nabla_L = 2\min\deg V_L.$$
\end{proposition}

\begin{proposition}[Cromwell]\label{pos and almost pos links have pos conway}\cite{Cromwell}
    Positive links have positive Conway polynomials, and almost-positive links have positive Conway polynomials. 
\end{proposition}

\begin{proposition}[Stoimenow, \cite{Stoimenow2014MinimalGA}]\label{Stoimenow almost pos Conway degree min genus}
    Let $L$ be an almost-positive link. Then 
    $$\max \deg \nabla_L = 2\min \deg V_L = 1-\chi(L)$$ 
\end{proposition} 

The following propositions and theorems involve considering two types of almost-positive diagrams: one type in which the negative crossing and a positive crossing both connect the same pair of Seifert circles, and another type in which no positive crossing connects the same pair of Seifert circles as the negative crossing. Discussion of these two separate situations appears in the work of Feller, Lewark, and Lobb (\cite{Feller_2022}, notions of \textit{parallel crossings} and \textit{type I and type II diagrams}); Ito and Stoimenow (\cite{Stoimenow, StI} notions of \textit{Seifert equivalent crossings} and \textit{type I and type II diagrams}, and \textit{good and bad crossings} and \textit{good successively $k$-almost positive diagrams}); and Tagami (\cite{Tagami_2014}). We recall that for positive crossings, performing an $A$-smoothing is the same as smoothing according to Seifert's algorithm, so in a positive the $A$-state circles are exactly the same as the Seifert circles.

\begin{lemma}[Stoimenow, \cite{StI}]\label{stoimenow's lemma}
    Let $L$ be a link represented by an almost-positive diagram $D$ with negative crossing $q$. If $D$ is of type $I$ (there is no other crossing $p$ which connects the same pair of Seifert circles as $q$), then $\chi(L) = \chi(D)$. If $D$ is of type $II$ (there is another crossing $p$ which connects the same pair of Seifert circles as $q$), then $\chi(L) - 2  = \chi(D)< \chi (L)$. 
\end{lemma}

\begin{remark}
    We note that then the first case means (for a positive diagram $D_+$ related to an almost-positive diagram $D_-$ by one crossing change) that $2\min \deg V(L_+) = 1-\chi(L_+) = 1-\chi(D_+) = 1 - \chi(D_-) = 1-\chi(L_-)$, which (by Proposition \ref{Stoimenow almost pos Conway degree min genus} if the link is almost-positive, or by Proposition \ref{Cromwell pos Conway degree} if the link is positive) is equal to $2\min \deg V(L_-)$, so $\min \deg V(L_+) = \min \deg V(L_-)$. And in the other case, we have that $\min \deg V(L_+) = \min \deg V (L_-) + 1. $
\end{remark}

\begin{proposition}\label{another crossing doesnt affect conway degree}
    Let $D_+$ be a positive diagram with one distinguished crossing $q$, and let $D_-$ be the result of making $q$ negative. If there is another crossing in $D_+$ connecting the same two $A$-circles as $q$, then $$\deg\nabla_- < \deg \nabla_+.$$
\end{proposition}


\begin{proof}

Recall the Conway skein relation: $\nabla_+ - \nabla_- = z \nabla_0$. 

If $D_+$ is a positive diagram, and $D_-$ is the result of changing one crossing to be negative, and $D_0$ is the result of smoothing that crossing, then the skein relation and \ref{pos and almost pos links have pos conway} tells us that $$\deg \nabla_- \leq \deg \nabla_+.$$

Let $D_+$ be a positive diagram, $D_-$ the result of changing one crossing ($q$) to be negative, and $D_0$ the result of smoothing $q$. 

Since $D_+$ and $D_0$ are positive, we have that (where $n$ is the number of link componets in $D_+$, and $*=\min\deg V_+=\frac{c(D_+)-s(D_+)+1}{2}= 1- \chi(D_+)$, and $r$ is the absolute value of its second Jones coefficient):

\begin{itemize}
    \item $V_+ = (-1)^{n+1}\Big (t^* -(r) t^{*+1} + \dots \Big)$ where the rest of the terms are of higher degree 
    \item If another crossing in $D_+$ connects the same pair of Seifert circles as $q$, then smoothing $q$ does not change the number of holes in the reduced $A$-state graph and so 
    $V_0 = (-1)^n \Big ( t^{*-1/2} - (r) t^{*+1/2} + \dots \Big )$ where the rest of the terms are of higher degree
 
\end{itemize}

 Recall the skein relation: $t^{-1}V_+ - tV_- = (t^{1/2} - t^{-1/2})V_0$,

 And consider the following:

\begin{align}
    (t^{1/2} - t^{-1/2})V_0 & = (-1)(t^{-1/2} - t^{1/2})V_0 \notag \\
     & = (-1)^{n+1} \Big( t^{*-1}  - (r)t^{*} - t^{*} + (r)t^{*+1} + \dots \Big) \label{maybe same degree terms} \\
    & = \underbrace{(-1)^{n+1}  \Big ( t^{*-1} -(r+1)t^* + ( r)t^{*+1} + \dots \Big)}_{:=\tcc{**}} \notag
\end{align}

and then by the skein relation,
\begin{align*}
   V_- &= t^{-1}\Big( t^{-1} V_+ - (t^{1/2} - t^{-1/2})V_0 \Big) \\
    & =  t^{-1}\Bigg( (-1)^{n+1}\Big(t^{*-1} - (r)t^{*} + \dots \Big) - (\tcc{**}) \Bigg)\\
    & = t^{-1} (-1)^{n+1} \Big (t^{*} - (r)t^{*+1} + \dots\Big)
\end{align*}
And so $\min \deg V_- = * - 1 = (\min \deg V_+) - 1$.

Thus we have that $$\underbrace{\deg \nabla_- = 2 \min \deg V_-}_{\text{by Thm. \ref{Stoimenow almost pos Conway degree min genus}}} = 2\min\deg V_+ -2 < \underbrace{2\min \deg V_+ = \deg \nabla_+ }_{\text{by Thm. \ref{Cromwell pos Conway degree}}}.$$

\end{proof}

\begin{remark}
    The remaining terms not explictly written out in line \ref{maybe same degree terms} are of degree $*+1$ or higher, and this trickles down to the final line. So we do not claim that the second coefficient of $V_-$ must be equal to $\pm r.$
\end{remark}

The preceeding argument is mentioned in Stoimenow's work in the proofs for Proposition \ref{Stoimenow almost pos Conway degree min genus} and Lemma \ref{stoimenow's lemma} \cite{Stoimenow}.

    As we will see, what this means for us is that if two crossings connect the same pair of $A$-circles, we can smooth one of them away and not change the degree or leading coefficient of the Conway polynomial. This also appears in the work of Stoimenow and Ito \cite{StI}.

    \begin{corollary}\label{another crossing doesnt affect leading term}
    Let $D_+$ be an almost-positive diagram with one distinguished crossing $q$, let $D_-$ be the result of making $q$ negative, and let $D_0$ be the result of smoothing it. If there is another crossing in $D_+$ connecting the same two $A$-circles as $q$, then $$\lead \term \nabla_+ = z \lead \term \nabla_0.$$
    
\end{corollary}

\begin{proof}
 
The Conway skein relation tells us that $\nabla_+ = \nabla_- + z\nabla_0$.  

$D_+$ and $D_0$ are both positive diagrams, have the same number of $A$-circles and their crossing number differs by 1, so  by Proposition \ref{Cromwell pos Conway degree}, 
$$\deg \nabla_+ = c(D_+) - A_{D_+} + 1 = c(D_0) - A_{D_0} + 1 + 1 = \deg \nabla_0 + 1.$$

That is, $\nabla_0$ definitely contributes to the leading term of $\nabla_+$. What about $\nabla_-$?

Since $D_+, D_-,$ and $D_0$ are all positive or almost-positive diagrams, they all represent positive or almost-positive links, and therefore by Proposition \ref{pos and almost pos links have pos conway}, all of $\nabla_+, \nabla_-,$ and $\nabla_0$ are positive. 

It follows that if $\deg \nabla_- < \deg \nabla_+$, then $\lead \term \nabla_+ = z \lead \term \nabla_0.$

Similarly, if $\deg \nabla_- = \deg \nabla_+,$ then $\lead \term \nabla + = \lead \term \nabla_- + z \lead \term \nabla_0.$ 

By Proposition \ref{another crossing doesnt affect conway degree}, we are done.

\end{proof}

    \begin{lemma}\label{smooth away, keep lead coeff}
        Let $D$ be a positive diagram. Let $D_m$ be the result of adding $m$ crossings $D$ such that at every intermediate diagram $D_i$ (for $0\leq i\leq m$), the reduced $A$-state graph of $D_i$ is exactly the reduced $A$-state graph of $D$. Then $$\lead \term \nabla_{D_m} = z^m \lead \term \nabla_{D}.$$
    \end{lemma}

    \begin{proof}

    We proceed by induction on $m$. If $m=0$, then $D_0=D$ and we are done. 

    If $m = 1$, then we have added a single crossing $q$ to $D$ to create $D_1$, and we have not changed the underlying reduced $A$-state graph structure. Therefore, there exists another crossing $p$ in $D_1$ that connects the same pair of $A$-circles as $q$. By Corollary \ref{another crossing doesnt affect leading term}, and letting $D_+ = D_1$ and $D_0 = D$, we have that 
$$\lead \term \nabla_{D_1} = z \lead \term \nabla_{D} $$. 

    Now assume there is some value of $m$ for which the statement holds, and consider a diagram $D_{m+1}$. This is the result of adding $m+1$ crossings to $D$, so is also the result of adding $1$ crossing to $D_m$ while preserving the underlying graph structure. Therefore, by the same argument as in the base case, 
    $$\lead \term \nabla_{D_{m+1}} = z \lead \term \nabla_{D_m}.$$ By inductive hypothesis, $\lead \term \nabla_{D_m} = z^m \lead \term \nabla_{D},$ so in total 

    $$\lead \term \nabla_{D_{m+1}} = z \Big( z^m \lead \term \nabla_{D} \Big) = z^{m+1} \lead \term \nabla_{D}.$$
    
    \end{proof}

    \begin{lemma}\label{smooth away burdens, keep lead coeff}
       Let $D_m$ be:
       \begin{itemize}
      \item  a Burdened diagram of type $1$ in which $m$ crossings can be smoothed to obtain Balanced type $1$ diagram $D$, 
         \item a Burdened diagram of type $2$ in which $m$ crossings can be smoothed to obtain a Balanced type $2$ diagram $D$, or
         
         \item an Oddly Burdened diagram of type $2$ in which $m$ crossings can be smoothed to obtain an Oddly Balanced type $2$ diagram $D$.
         
         Then $$\lead \term \nabla_{D_m} = z^m \lead \term \nabla_D. $$
        \end{itemize}
    \end{lemma}

    \begin{proof}
        This follows immediately from Lemma \ref{smooth away, keep lead coeff}.  
    \end{proof}

    So, to find the leading coefficient of the Conway polynomial of Burdened diagram, it suffices to find the leading coefficient of a Balanced one.

    \begin{lemma}\label{contract cut edges, same result}
    Let $D$ be a positive link diagram with reduced $A$-state graph $G$. Let $G'$ be the result of contracting all cut edges in $G$. Then for any positive link diagram $D'$ whose reduced $A$-state graph is $G'$, 
    $$\lead \coeff \nabla_{D'} = \lead \coeff \nabla_{D}.$$
    \end{lemma}

    \begin{proof}

    Let $D$ be a positive link diagram whose reduced $A$-state graph $G$ has $t$ cut edges (edges that are not part of cycles).
    By Lemma \ref{smooth away, keep lead coeff}, it suffices to consider the case where every edge in $G$ corresponds to exactly one crossing in $D$. 

    Then there are $t$ nugatory crossings in $D$. For each such crossing, we can lift it up and over half of the diagram
    without changing any other part of the diagram, and the only effect this has is to contract the corresponding edge in $G$. This new diagram is actually equivalent the old one and thus has the same Conway polynomial, not just the same leading Conway coefficient. As we said before, the result then follows by Lemma \ref{smooth away, keep lead coeff}. 
    \end{proof}

    \begin{remark}
        This means that to find the leading Conway coefficient of a positive link $D$, we can instead:
        \begin{itemize}
            \item Draw its reduced $A$-state graph $G$, 
            \item Contract all cut edges to obtain a new graph $G'$,
            \item Draw a positive link diagram $D'$ whose reduced $A$-state graph is $G'$ and whose crossings are in one-to-one correspondence with the edges of $G'$, and then 
            \item Find the leading Conway coefficient of $D'$. 
        \end{itemize}
    \end{remark}

    \subsection{Leading Conway Coefficient for Balanced diagrams of type $1$} \label{lead conway coeff coeff 1 section}

    We begin with the easiest case. The standard alternating diagram of the $T(2,2p)$ torus link (given positive orientation) is a Balanced diagram of type $1$. It is known that 
    \begin{proposition}\label{torus links conway polynomial}
        The Conway polynomial of the positive torus link $T(2,2p)$ is $\nabla = pz$. 
    \end{proposition}
    \begin{proof}
        This can be seen using the Conway skein relation, and the observation that smoothing one of the crossings gives the unknot, while changing it gives an almost-positive diagram of the link $T(2,2p-2)$. So 
        $$\nabla_{T(2,2p)} = \nabla_{T(2,2p-2)} + z(1) = \dots = \nabla_{T(2,2)} + (p-1) z = pz.$$
    \end{proof}

    \begin{lemma}\label{lead coeff k-Bur}
        Let $D$ be a $k$-Burdened diagram. Then the leading coefficient of its Conway polynomial is $\lead \coeff \nabla_D = \frac{k}{2}$
    \end{lemma}

    \begin{proof}
        This follows directly from Lemma \ref{contract cut edges, same result} and Proposition \ref{torus links conway polynomial}. 
    \end{proof}

    \subsection{Leading Conway Coefficient for Balanced diagrams of type $2$}\label{lead conway coeff coeff 2 section}

    As in the last section, it suffices to consider link diagrams whose reduced $A$-state graphs contain no cut edges. 
For a graph to be the reduced $A$-state graph of a Balanced type $2$ link diagram, $x$ (the number of edges that bound both holes of the graph) must be even.

    We will consider the case when $x=0$ and the case when $x \geq 2$. \\

    \underline{Case: $x = 0$}

    For $k_1$ and $k_2$ even and $\geq 4$, the connect sum of the standard alternating $T(2,k_1)$ and $T(2,k_2)$ torus link diagrams form a positive link diagram whose reduced $A$-state graph is such that $x=0$. The leading Conway coefficient of this composite link is $\frac{k_1}{2}\frac{k_2}{2}=\frac{k_1k_2}{4}.$

    So by Lemma \ref{contract cut edges, same result}, we just proved: 

    \begin{proposition}\label{lead coeff Bur type 2 x=0}
        Let $D$ be a $(k_1, k_2)$-Burdened diagram with reduced $A$-state graph $G$. If the two holes of $G$ are not bounded by any of the same edges, then the leading coefficient of its Conway polynomial is
        $$\lead \coeff \nabla_D = \frac{k_1k_2}{4}.$$
    \end{proposition}

    Now, we consider the case if there are edges in $G$ that bound both holes.\\

   \underline{Case: $x \geq 2$}

    We observe that the standard alternating diagram of certain pretzel links will have this kind of reduced $A$-state graph. 

 We observe that the standard alternating diagram of a positive pretzel link $P(-p,-q,-r)$ with even $p, q,r \geq 2$ is a Balanced diagram of type $2$. The pretzel with all counter-clockwise twists is the one which actually corresponds to the link in which all crossings are positive. By convention, a counterclockwise twist is indicated with a negative sign, and we choose to follow the standard convention of naming our pretzel links in the following computations. 

     \begin{lemma}\label{pretzel conway all evens}
        Let $p, q,r$ even integers $\geq 2$. Let $\nabla(-p,-q,-r)$ be the Conway polynomial for the $P(-p,-q,-r)$ positive pretzel link. Then 
        $$\nabla(-p,-q,-r) =  \Big( \frac{pq + pr + qr}{4}\Big)z^2.$$
    \end{lemma}

    \begin{proof}
        This follows directly from the Conway skein relation, and the following facts: 
        \begin{itemize}
            \item Changing one of the crossings in the $p$-twist column gives us an almost-positive diagram of the positive link $P(-(p-2), -q, -r)$.
            \item Smoothing one of the crossings in the $p$-twist wipes out that whole column, and gives us torus link $T(2,q+r)$, which by Lemma \ref{torus links conway polynomial} has Conway polynomial $\nabla = \frac{q+r}{2}z$. 
            \item The $P(0, -q, -r)$ pretzel link is actually the connect sum of torus links $T(2,q)$ and $T(2,r)$. By Lemma \ref{torus links conway polynomial} and multiplicativity of the Conway polynomial, this means $\nabla(0,-q, -r) = \nabla(T(2,q))\cdot \nabla(T(2,r)) = \frac{q}{2}z\cdot \frac{r}{2}z = \frac{qr}{4}z^2$.
        \end{itemize}

        Therefore, 
        \begin{align*}
            \nabla(-p,-q,-r) &= \nabla(0, -q, -r) + \Big(\frac{p}{2}\Big) z \nabla(T(2,q+r))\\
            & = \frac{qr}{4}z^2 + \Big(\frac{p}{2}\Big)\frac{q+r}{2}z^2\\
            & = \frac{pq+pr+qr}{4}z^2.
        \end{align*}

        \end{proof}

\begin{theorem} \label{lead coeff for bur type 2}
    Let $D$ be a $(k_1, k_2)$-Burdened diagram. Then the leading coefficient of its Conway polynomial is 
    $$\lead \coeff \nabla_D = \frac{k_1k_2 - x^2}{4}.$$
\end{theorem}

\begin{proof}
    We have already proven the case if $x=0$ in our Proposition \ref{lead coeff Bur type 2 x=0}. 

    If $x \geq 2$, then by Lemma \ref{contract cut edges, same result} it suffices to find the leading Conway coefficient for any $(k_1, k_2)$-Burdened diagram with that value of $x$. With our notation for Balanced diagrams of type $2$, the pretzel link $P(-p,-q,-r)$ for $p,q,r$ even is a $\Big((p+q), (r+q)\Big)$-Balanced link, where  $k_1 = p+q, k_2 = q+r,$ and $x=q$. 

    So, let us put that information into the expression above: 
    \begin{align*}
        \frac{k_1k_2 - x^2}{4} &= \frac{(p+q)(r+q) - (q)^2}{4} \\
        & = \frac{pq + pr + qr}{4} \\
        & = \lead \coeff \nabla(-p, - q, -r)\\
        & = \lead \coeff \nabla_D,
    \end{align*}

     as desired. 
\end{proof}

\subsection{Oddly Burdened type 2}

As in the previous case, we can look to pretzel links for a place to start.

\begin{theorem}
    Let $D$ be a $(k_1, k_2)$-Oddly Burdened diagram. Then the leading coefficient of its Conway polynomial is 
    $$\lead \coeff \nabla_D = \frac{k_1k_1 - x^2 + 1}{4}.$$
\end{theorem}

\begin{proof}

Let $D$ be a $(k_1, k_2)$-Oddly Burdened diagram with reduced $A$-state graph $G$ in which the two holes are bounded by $x$ shared edges, with odd $x\geq 1$. 

By our Lemma \ref{contract cut edges, same result}, we know that the leading Conway coefficient of $D$ will be the same as the leading Conway coefficient of the pretzel link $P(-p, -q, -r)$ where $p=k_1-x$, $q=x$, and $r=k_2-x$. (We also note that $P$ is invariant under any permutation of the three arguments, so it really does not matter which value we choose to be $x$.)

So, it suffices to find that leading coefficient, and show it is equal to the one given above. 

    \begin{claim}
        Let $p, q,r$ be odd integers $\geq 1$. Let $\nabla(-p,-q,-r)$ be the Conway polynomial for the $P(-p,-q,-r)$ pretzel link, given a positive orientation. Then 
        $$\nabla(-p,-q,-r) = 1 + \Big( \frac{pq + pr + qr + 1}{4}\Big) z^2.$$
    \end{claim}

    \begin{proof}
        This follows directly from the Conway skein relation, and the following facts: 
        \begin{itemize}
            \item Changing one of the crossings in the $p$-twist column gives us an almost-positive diagram of the $P(-(p-2), -q, -r)$ pretzel link. 
            \item Smoothing one of the crossings in the $p$-twist wipes out that whole column, and gives us the $T(2, q+r)$ torus link, which by Lemma \ref{torus links conway polynomial} has Conway polynomial $\frac{q+r}{2}$. 
            \item $P(-1,-1,-1)$ is the trefoil, with $\nabla = 1+z^2$. 
        \end{itemize}

        Thus 
        \begin{align*}
            \nabla(-p, -q, -r)
            & = \nabla(-1, -q, -r) + \Big(\frac{p-1}{2}\Big) z \nabla(T(2, q+r))\\
            & = \nabla(-1, -1, -r) + \Big(\frac{q-1}{2}\Big) z \nabla(T(2, 1+r)) + \Big(\frac{p-1}{2}\Big) \Big(\frac{q+r}{2}\Big)z^2\\
            & = \nabla(-1,-1,-1) + \Big(\frac{r-1}{2}\Big) z \nabla(T(2,2)) + \Big(\frac{q-1}{2}\Big) \Big(\frac{r+1}{2}\Big) z^2 + \Big(\frac{p-1}{2}\Big) \Big(\frac{q+r}{2}\Big)z^2\\
            & = 1 + z^2 + \Big(\frac{r-1}{2}\Big) z^2 + \Big(\frac{q-1}{2}\Big) \Big(\frac{r+1}{2}\Big) z^2 + \Big(\frac{p-1}{2}\Big) \Big(\frac{q+r}{2}\Big) z^2 \\
            & = 1 + \frac{pq + pr + qr + 1}{4}z^2.
        \end{align*}
    \end{proof}

Now, let us substitute: $k_1 = p+q$, $k_2= r+ q$, and $x=q$. Then:
\begin{align*}
    \frac{k_1k_2 - x^2 + 1}{4 } & = \frac{(p+q)(r+q) - (q)^2 + 1}{4}\\
    & = \frac{pr + pq + qr + q^2 - q^2 + 1}{4}\\
    & = \frac{pq + pr + qr + 1}{4},
\end{align*}
 as desired. 

 \end{proof}





    \section{Application}\label{Application section}

In \cite{Buchanan_2022}, we showed that the positivity obstruction developed for the case of fibered positive knots could be used to show that seven particular knots are not positive, and are instead almost-positive. These were the last remaining knots of crossing number $\leq 12$ for which positivity was unknown. Independently and around the same time, Stoimenow compiled a list \cite{StListofPosKnots} of all non-alternating positive knots up to $15$, also settling the question of those remaining seven knots through a different method.

In Section \ref{coeff 0 family section} we give another example of an almost-positive knot diagram whose non-positivity is proved by failing the test of Theorem \ref{main result for coeff 0}. From this $16$-crossing diagram, we can construct an infinite family of examples of almost-positive knots, whose non-positivity is shown by failing the test of Theorem \ref{main result for coeff 0} in the exact same way. 

In Sections \ref{coeff 1 family section} and \ref{coeff 2 family section} we similarly construct infinite families of knots whose non-positivity is shown by failing the tests of Theorem \ref{main result coeff 1 and 2}. 

We would like to mention that KnotFolio was an extremely helpful tool for constructing such examples \cite{KnotFolio}.

    \subsection{An Infinite Family of Almost-Positive Knots with Second Jones Coefficient equal to $0$}\label{coeff 0 family section}

    \input{Figures/Knot_and_Fam_coeff_0}

    In Figure \ref{fig:Knot coeff 0}, we see a $16$-crossing almost-positive link diagram.  We can compute the Jones polynomial and the HOMFLY polynomial (provided for the interested reader) with Regina \cite{Regina}, and find that they are:

    \begin{itemize}
        \item Jones polynomial: $V = t^4 -3t^6 + 12t^7 - 24t^8 +38t^9 - 49t^{10} + 56t^{11} -56t^{12} + 48t^{13} -37t^{14} + 23t^{15} -12t^{16} + 5t^{17} - t^{18}$
        \item HOMFLY polynomial $P = \alpha^{-8} z^8 + 8 \alpha^{-8} z^6 + 17 \alpha^{-8} z^4 + 13 \alpha^{-8} z^2 + 3 \alpha^{-8} + 3 \alpha^{-10} z^6 + 5 \alpha^{-10} z^4 - 2 \alpha^{-10} + 4 \alpha^{-12} z^6 + 10 \alpha^{-12} z^4 + 11 \alpha^{-12} z^2 + 6 \alpha^{-12} - 8 \alpha^{-14} z^4 - 18 \alpha^{-14} z^2 - 11 \alpha^{-14} + 5 \alpha^{-16} z^2 + 6 \alpha^{-16} - \alpha^{-18}$
    \end{itemize}

    The second Jones coefficient is $0$. If this diagram represented a positive link, then Theorem \ref{main result for coeff 0} tells us that we would have $\max \deg V \leq 4\min \deg V$. However, 
    $$\max \deg V = 18 \nleq 16 = 4\min \deg V,$$ so the knot (identified by the KnotFinder tool as $16_n125409$ \cite{knotfinder}) cannot be positive.

      We can use this knot to generate an infinite family of knots that can be shown not to be positive by using Theorem \ref{main result for coeff 0} in the same way. We create a new diagram $D_w$ by adding $w$ copies of a three-crossing loop to the upper left corner of our original diagram $D_0$, as shown in Figure \ref{fig:Family coeff 0}, so that the crossing number of our new diagram is $c(D_w) = c(D_0) + 3w.$

      \begin{claim}\label{claims for family coeff 0}
       Let $V(w)$ be the Jones polynomial of $D_w$, and let $\nabla(w)$ be its Conway polynomial. Then 
      \begin{enumerate}
          \item $\min \deg V(w) = \min \deg V(0) + w$
          \item $\max \deg V(w) = \max \deg V(0) + 4w$
          \item The second Jones coefficient of $V(w) = $ the second Jones coefficient of $V(0)$
      \end{enumerate}
      \end{claim}

      If our claims are true, then the second Jones coefficient of $D_w$ is $0$, and yet
      \begin{align*}
      \max \deg V(w) &= \max \deg V(0) + 4w\\
      & = 18 + 4w\\
      & \nleq 16 + 4w\\
      &= 4 \min \deg V(0)+ 4w \\
      & = 4 \min \deg V(w).
      \end{align*}
      
      so by Theorem \ref{main result for coeff 0}, $D_w$ cannot be a diagram of a positive knot. 

The proof of these is identical to the proofs in the next section.
    
    \subsection{An Infinite Family of Almost-Positive Knots with Second Jones Coefficient equal to $-1$}\label{coeff 1 family section}

    \input{Figures/Knot_and_Fam_coeff_1}

    In Figure \ref{fig:Knot coeff 1} we see a $15$-crossing almost-positive diagram (the negative crossing is circled). 

    Using the DT code for the pictured diagram $$[4, 10, 30, 20, 2, 24, 22, -26, -14, 8, 28, 12, -18, -16, 6]$$ 
    and inputting into Regina \cite{Regina}, we obtain the Jones polynomial and the HOMFLY polynomial, from which we can obtain the Conway polynomial under the substitution $P(a=1, z) = \nabla(z)$. 

    We have: 
    \begin{itemize}
        \item Jones polynomial: $V = t^3 - t^4 + 2t^5 - t^6 - t^7 + 5t^8 -8t^9 +11t^{10} -13t^{11} + 12 t^{12} - 10 t^{13} + 6t^{14} - 3t^{15} + t^{16}$
        \item HOMFLY polynomial: $P = a^{-6} z^6 + 5 a^{-6}  z^4 + 6 a^{-6}  z^2 + 2 a^{-6}  + a^{-8}  z^6 + 5 a^{-8} z^4 + 3 a^{-8}  z^2 + 2 a^{-10}  z^4 + a^{-10} z^2 + 2 a^{-12} z^4 + a^{-12}z^2 - 3 a^{-14} z^2 - 2 a^{-14} + a^{-16}$
        \item Conway polynomial: $\nabla = 1 + 8z^2 + 14z^4 + 2z^6$
      \end{itemize}

      Then $$ \max \deg V = 16 \nleq 14 = 4 \min \deg V + 2 \lead \coeff \nabla -2, $$ so by Theorem \ref{main result coeff 1 and 2}, this knot (identified by the KnotFinder tool as $15_n11331$ \cite{knotfinder}) is not positive, and is thus an almost-positive knot, since we do have an almost-positive diagram.

      We can use this knot to generate an infinite family of knots that can be shown not to be positive by using Theorem \ref{main result coeff 1 and 2} in the same way. We create a new diagram $D_w$ by adding $w$ copies of a three-crossing loop to the upper left corner of our original diagram $D_0$, as shown in Figure \ref{fig:Family coeff 1}, so that the crossing number of our new diagram is $c(D_w) = c(D_0) + 3w.$

      \begin{claim}\label{claims for family coeff 1}
       Let $V(w)$ be the Jones polynomial of $D_w$, and let $\nabla(w)$ be its Conway polynomial. Then
      \begin{enumerate}
          \item $\min \deg V(w) = \min \deg V(0) + w$ \label{claims for family coeff 1 part 1}
          \item $\max \deg V(w) = \max \deg V(0) + 4w$ \label{claims for family coeff 1 part 2}
          \item The second Jones coefficient of $V(w) = $ the second Jones coefficient of $V(0)$ \label{claims for family coeff 1 part 3}
          \item $\lead \coeff \nabla(w) = \lead \coeff \nabla(0)$. \label{claims for family coeff 1 part 4}
      \end{enumerate}
      \end{claim}

      If our claims are true, then the second Jones coefficient of $D_w$ is $-1$, and yet

      \begin{align*}
      \max \deg V(w) &= \max \deg V(0) + 4w\\
      & = 16 + 4w\\
      & \nleq 14 + 4w\\
      &= 4 \min \deg V(0) + 2\lead \coeff \nabla(0) - 2 + 4w \\
      & = 4 \min \deg V(w) + 2\lead \coeff \nabla (w) -2.
      \end{align*}
      
      so by Theorem \ref{main result coeff 1 and 2}, $D_w$ cannot be a diagram of a positive knot. 

      \newpage

      \subsubsection{Proof of Claim \ref{claims for family coeff 1}, part (\ref{claims for family coeff 1 part 1})}

      \begin{proof}

    Since there is no other crossing in $D_w$ connecting the same pair of Seifert circles as the negative crossing, we have by Stoimenow's Lemma \ref{stoimenow's lemma} that $2\min \deg V(w) = 1-\chi(D_w)$, and thus 
          \begin{align*}
              2\min \deg V(w) &= c(D_w) - s(D_w) + 1 \\
              &= (c(D_0) + 3w) - (s(D_0) + w) + 1\\
              &= c(D_0) - s(D_0) + 1 + 2w \\
              &= 2\min \deg V(0) + 2w,
          \end{align*}
          So indeed $\min \deg V(w) = \min \deg V(0) + w$. 

           \end{proof}

         \subsubsection{Proof of Claim \ref{claims for family coeff 1}, part \ref{claims for family coeff 1 part 2}}

         \begin{proof}

         Observe that each $D_w$, including $D_0$, is $B$-adequate, so a result of \ref{B-adequate} says that 
          \begin{align*}
             \max \deg V(w) &= c(D_w) + \frac{B_{w} - 1}{2} - \frac{3}{2}\\
             & = (c(D_0) + 3w) + \frac{(B_0 + 2w) - 1}{2} - \frac{3}{2}\\
             & = c(D_0) + \frac{B_0 - 1}{2} - \frac{3}{2} + 4w\\
             & = \max \deg V(0) + 4w.
          \end{align*}

           \end{proof}
         
       \subsubsection{Proof of Claim \ref{claims for family coeff 1}, parts (\ref{claims for family coeff 1 part 3}) and (\ref{claims for family coeff 1 part 4})}

         Parts $(3)$ and $(4)$ are about the second Jones coefficient and the leading Conway coefficient, and we can get information about both from the HOMFLY(PT) polynomial. Since we need both, we first look at the skein relation of the HOMFLY(PT) polynomial, and then we will specialize to each case. 

         \begin{proof}

  \input{Figures/Family_skein}

Let $D_{00}$ be the diagram obtained by smoothing the negative crossing in $D_0$, and let $D_{0+}$ be the diagram obtained by making that crossing positive. Then by the HOMFLY(PT) skein relation, we have 
\begin{equation}\label{neg crossing skein}
    \alpha P(D_{0+}) - \alpha^{-1}P(D_0) = z P(D_{00}).
\end{equation} 

In Figure \ref{fig:Family skein}, we look at two \say{layers} of skein relations. The first layer is $D_w$ with distinguished positive crossing $p$ indicated by the blue arrow. Making $p$ negative will give us a diagram equivalent to $D_{w-1}$, and smoothing $p$ gives us a diagram equivalent to $D_{wS}$. For the next layer, we choose one of the two crossings in the top part of $D_{wS}$ shown in Figure \ref{fig:Family skein} to be the distinguished crossing, and we see that making that crossing negative gives us a diagram equivalent to $D_{00} \# 3_1 \# \dots \# 3_1$ (the connected sum of $D_{00}$ and $w-1$ copies of the trefoil), and smoothing that crossing gives us $D_{w-1}$ back again.
                    
        So, the first layer of the skein relation gives us
          \begin{equation} \label{HOMFLY skein 1}
              \alpha^{}P(w) - \alpha^{-1} P(w-1) = z P(D_{wS}), 
          \end{equation} 
       and the second layer gives us 
          \begin{equation}\label{HOMFLY skein 2 v1}
              \alpha^{}P(D_{wS}) - \alpha^{-1} P(D_{00} \# 3_1 \# \dots \# 3_1) = z P(w-1).
          \end{equation} 
Rearranging, multiplying by $z$, and using the fact that the HOMFLY(PT) polynomial is multiplicative over connected sums, we get 
            \begin{equation}\label{HOMFLY skein 2 v2}
              \alpha^{}z P(D_{wS}) = \alpha^{-1}z P(D_{00}) P(3_1)^{w-1} + z^2 P(w-1).
          \end{equation} 

          We can then express $\alpha^2 P(w)$ as: 
          \begin{align}
          \notag \alpha^{2} P(w) &= \alpha \Big ( \alpha^{-1} P (w-1) + z P(wS)\Big)  \notag &\text{ (by \ref{HOMFLY skein 1})} \\
        \notag &= P(w-1) + \alpha^{-1} zP(D_{00}) P(3_1)^{w-1} + z^2 P(w-1)  \notag &\text{ (by \ref{HOMFLY skein 2 v2})} \\
       \notag   & = (1+z^2) P(w-1) + \alpha^{-1} \Big(\alpha P(D_{0+}) - \alpha^{-1} P(D_0)\Big ) P(3_1)^{w-1} &\text{ (by \ref{neg crossing skein})}  \\
        & = (1+z^2) P(w-1) + \Big(P(D_{0+}) - \alpha^{-2} P(0)\Big ) P(3_1)^{w-1} 
         \label{HOMFLY skein 3}
          \end{align}

          Now, we specialize to each part. 

        When we specialize to the Jones polynomial via the substitution $P(\alpha = t^{-1}, z=t^{1/2} - t^{-1/2})$,  Equation \ref{HOMFLY skein 3} becomes 

          \begin{equation}\label{Jones skein for family}
             t^{-2} V(w) = (t^{-1} - 1 + t)V(w-1) + \Big( V(D_{0+}) - t^2 V(0)\Big)V(3_1)^{w-1}.
          \end{equation}

          In particular, for the base case of $w=1$ we have 

          \begin{equation}\label{Jones skein for family case w=1}
              t^{-2} V(1) = (t^{-1} - 1 + t)V(0) + \Big( V(D_{0+}) - t^2 V(0)\Big) = (t^{-1} - 1 +t - t^2)V(0) + V(D_{0+}).
          \end{equation}

          Since there are no (other) crossings in the almost-positive diagram $D_0$ that connect the same pair of Seifert circles as the negative crossing, it follows from Lemma \ref{stoimenow's lemma} that $\min \deg V_{D_0} = \min \deg V_{D_{0+}}$. Therefore, the first term in $V(1)$ is contributed entirely by $V(0)$, and the second coefficient of $V(1)$ is exactly the first coefficient of $V(0)$ subtracted from the sum of the second coefficient of $V(0)$ and the first coefficient of $V(D_{0+})$. But we know that the first coefficients of both $V(0)$ and $V(D_{0+})$ are $1$, since $D_{0}$ is almost-positive and $D_{0+}$ is positive. Hence the second coefficient of $V(1)$ is exactly the second coefficient of $V(0)$.

          Now assume for induction that the second coefficient of $V(w)$ is equal to the second coefficient of $V(0)$. We recall that $V(3_1)=t+t^3 - t^4$, and by part $(1)$ we know $\min \deg V(w)= \min \deg V(0) + w = 3+w.$ Then by \ref{Jones skein for family}, we can write 
          
          \begin{equation}\label{Jones skein for family v2}
             t^{-2} V(w+1) = \underbrace{(t^{-1} - 1 + t)V(w)}_{\min \deg = -1 + \min\deg V(0) + w} + \overbrace{\Big( V(D_{0+}) - t^2 V(0)\Big)V(3_1)^{w}}^{\min \deg = \min\deg V(0) +w}.
          \end{equation}

          Therefore, the first coefficient of $V(w+1)$ is exactly the first coefficient of $V(w)$, and the second coefficient of $V(w+1)$ is exactly the first coefficient of $V(w)$ subtracted from the sum of the second coefficient of $V(w)$ and the first coefficient of $V(D_{0+})$. But again, we know that the first coefficients of both $V(w)$ and $V(D_{0+})$ are $1$, since $D_{w}$ is almost-positive and $D_{0+}$ is positive. Hence the second coefficient of $V(w+1)$ is exactly the second coefficient of $V(w)$, which is equal to the second coefficient of $V(0)$ by our induction hypothesis. 

          Thus we have that each $D_w$ has the same second Jones coefficient as $D_0$, which was our third claim.

          To prove our last claim, we specialize \ref{HOMFLY skein 3} to the Conway polynomial, via the substitution $P(\alpha = 1,z)$: 

          \begin{equation}\label{Conway skein for family}
              \nabla(w) = (1+z^2) \nabla(w-1) + \Big( \nabla(D_{0+}) - \nabla(0)\Big) \nabla(3_1)^{w-1} = (1+z^2) \nabla(w-1) + \nabla(D_{00}) \nabla(3_1)^{w-1}. 
          \end{equation}

We observe that $D_{00}$ is a positive diagram with $14$ crossings and $10$ $A$-circles, so then $\deg \nabla(D_{00})= 14-10+1=5$. And so for the base case of $w=1$ we have 
          \begin{align*}
         \nabla(1)  &= \underbrace{(1+z^2) \nabla(0)}_{\deg = 2 + \deg \nabla(0)} + \underbrace{z \nabla(D_{00})}_{\deg = 1 + 5}. 
          \end{align*}
So for $w=1$, since our $\deg \nabla(0)=6$, we then have that $\deg \nabla(w) = 2w+ \deg \nabla(0)$ and $\lead \coeff \nabla(w) = \lead \coeff \nabla(0)$.

We claim that this is true for all $w\geq 1$, and prove it by induction. Suppose our claim is true for $w$. Then, since $\nabla(3_1) = 1+z^2$, 

  \begin{align*}
         \nabla(w+1)  &= \underbrace{(1+z^2) \nabla(w)}_{\deg = 2 + 2w + \deg \nabla(0) } + \overbrace{z \nabla(D_{00}) \big(\nabla(3_1)\big)^{w}}^{\deg = 1 + 5 + 2w}.
          \end{align*}

Thus $\deg \nabla(w+1)= 2(w+1) + \deg \nabla(0)$, and we must have that $\lead \coeff \nabla(w+1) = \lead \coeff \nabla(w)$. And this is equal to $\lead \coeff \nabla(0)$ by our induction hypothesis, so indeed $\lead\coeff\nabla(w+1)=\lead\coeff\nabla(0)$. 

And we have now proved all four of our claims. 
          
      \end{proof}

\subsection{An Infinite Family of Almost-Positive Knots with Second Jones Coefficient equal to $-2$}\label{coeff 2 family section}

\input{Figures/Knot_and_Fam_coeff_2}

    Figure \ref{fig:Knot coeff 2} is a $15$-crossing almost-positive knot diagram with DT code $$[4, 8, 22, 2, 20, 26, 24, -28, -14, 10, 6, 30, 12, -18, -16].$$ 
    We use the obstruction provided by Theorem \ref{main result coeff 1 and 2} to prove that this knot cannot be positive. 

    Using Regina \cite{Regina}, we find its:
    \begin{itemize}
        \item  Jones polynomial:
    $V = t^3 - 2t^4 + 5t^5 - 6t^6 + 7 t^7 - 6t^8 + 3t^9 + t^{10} -4t^{11} + 6t^{12} -7t^{13} + 5t^{14} -3t^{15} + t^{16}.$
    \item HOMFLY polynomial: $a^{-6} z^6 + 4 a^{-6}  z^4 + 4 a^{-6}  z^2 + a^{-6}  + 2 a^{-8} z^6 + 9a^{-8} z^4 + 10 a^{-8}  z^2 + 3 a^{-8}  - a^{-10} z^4 - 6 a^{-10}z^2 - 4 a^{-10}+ 2 a^{-12} z^4 + 4 a^{-12} z^2 + 3 a^{-12} - 3 a^{-14} z^2 - 3 a^{-14} + a^{-16}$
    \item Conway polynomial: $\nabla = 1 + 9z^2 + 14z^4 + 3z^6$
    \end{itemize}

    As in the previous example, a quick check tells us that this diagram (identified by KnotFinder as $15_n11445$ \cite{knotfinder}) cannot represent a positive knot: $$\max \deg V = 16 \nleq 15 = 4\min \deg V + \lead \coeff \nabla,$$ so this fails the test of Theorem \ref{main result coeff 1 and 2}.

    We can use this knot to generate an infinite family of knots that can be shown not to be positive by using Theorem \ref{main result coeff 1 and 2} in the same way. We create a new diagram $D_w$ by adding $w$ copies of a three-crossing loop to the upper left corner of our original diagram $D_0$, as shown in Figure \ref{fig:Family coeff 1}, so that $c(D_w) = c(D_0) + 3w.$

      \begin{claim}\label{claims for family coeff 2}
       Let $V(w)$ be the Jones polynomial of $D_w$, and let $\nabla(w)$ be its Conway polynomial. 
      \begin{enumerate}
          \item $\min \deg V(w) = \min \deg V(0) + w$
          \item $\max \deg V(w) = \max \deg V(0) + 4w$
        \item The second Jones coefficient of $V(w) = $ the second Jones coefficient of $V(0)$
          \item $\lead \coeff \nabla(w) = \lead \coeff \nabla(0)$. 
      \end{enumerate}
      \end{claim}

    If our claims are true, then the second Jones coefficient of $D_w$ is $-2$, and yet
      \begin{align*}
      \max \deg V(w) &= \max \deg V(0) + 4w\\
      & = 16 + 4w\\
      & \nleq 15 + 4w\\
      &= 4 \min \deg V(0) + \lead \coeff \nabla(0) + 4w \\
      & = 4 \min \deg V(w)+ \lead \coeff \nabla(w).
      \end{align*}

      so by Theorem \ref{main result coeff 1 and 2}, $D_w$ cannot be a diagram of a positive knot. 

      \begin{proof}
         The proof is the same as for Claim \ref{claims for family coeff 1}. 
      \end{proof}

\section{Appendix}\label{Appendix}

This discussion of the Kauffman state-sum model is of course based on the work in Kauffman's paper \cite{Kauffman}. It is also heavily influenced by (and borrows notation from) similar discussions written by Livingston (\cite{livingston_1993}); Futer, Kalfagianni, and Purcell (\cite{futer2012guts}); and Stoimenow (\cite{Stoimenow}).

In the Kauffman state-sum model of the Jones polynomial, we begin with a link diagram and at every crossing, perform either an $A$-smoothing or a $B$-smoothing \cite{Kauffman}. 
If the diagram had $c$ crossings, then this results in $2^c$ possible arrangements of circles. Each such arrangement is called a \textit{state}. Of particular interest to us are the $A$-state (where every crossing is given an $A$-smoothing) and the $B$-state (where every crossing is given a $B$-smoothing). On each state, we compute the  \textit{bracket}. Let $\sigma$ denote a state. Then the bracket is defined by: 
$$\langle D | \sigma \rangle : = A^{\#A-\#B},$$ where $\#A$ is the number of $A$-smoothings in the state and $\#B$ is the number of $B$-smoothings in the state \cite{Lickorish1988}. The Kauffman bracket of a diagram $D$ is the sum over all states $\sigma$ of: 
$$\langle D \rangle := \sum_{\sigma} \langle D | \sigma \rangle (-A^2-A^{-2})^{|\sigma|-1},$$ where $|\sigma|$ is the number of circles in the state $\sigma$.  Then, the Kauffman bracket polynomial is defined 
\begin{align*}
    F(D) &:= (-A)^{-3w(D)} \langle D \rangle \\
    &= (-A)^{-3w(D)} \sum_\sigma \langle D | \sigma \rangle\\
    &= (-A)^{-3w(D)} \sum_{\sigma} A^{\#A-\#B}(-A^2-A^{-2})^{|\sigma|-1}
\end{align*}

The Kauffman bracket $\langle D \rangle $ is not an link invariant; the value changes depending the choice of diagram. (It is invariant under Reidemeister moves $II$ and $III$, but fails to be invariant under move $I$ \cite{Kauffman}.)  The Kauffman bracket\textit{ polynomial}, however, is a link invariant. And in fact, with just a simple substitution, we obtain the Jones polynomial: 

$$F(t^{-1/4}) = V.$$

This state-sum model allows us to easily observe how changing a single crossing has a ripple effect through the Jones polynomial.

We can observe that the all-$A$-state will contribute a term of highest possible degree to the Kauffman bracket polynomial $F$, which means that, via the substitution $t^{-1/4}$ for $A$, this corresponds to the lowest possible degree of the Jones polynomial. That is, 

\begin{equation*}
    \max \deg {F} \leq -3w(D) + c(D) + 2 (A_D - 1), 
\end{equation*}
and so with the replacement of $t^{-1/4}$ for $A$, we get a natural upper bound on the minimum degree of the Jones polynomial:

\begin{equation*}
    \min \deg V \geq \frac{-3w(D) + c(D) + 2(A_D - 1)}{-4} = \frac{3w(D) - c(D) - 2(A_D - 1)}{4}.
\end{equation*}

It can be helpful to rewrite this in terms of the number of negative crossings in the diagram. Let $q$ denote the number of negative crossings in $D$. Then $w(D) = c(D) - 2q$, and we can write: 
\begin{align*}
      \min \deg V &\geq \frac{3(c(D) - 2 q) - c(D) - 2(A_D - 1)}{4} \\
      &= \frac{c(D) - A_D + 1}{2} - \frac{3q}{2}.
\end{align*}

By the same logic, the all $B$-state contributes a term of lowest possible degree in the Kauffman bracket polynomial, and thus of highest possible degree in the Jones polynomial. Hence
\begin{equation}
    \max \deg V_D \leq \frac{-3w(D) - c(D) - 2(B_D - 1)}{-4} = \frac{3w(D) + c(D) + 2(B_D - 1)}{4} 
\end{equation}

Incorporating information about negative crossings:
\begin{equation}
    \max \deg V_D \leq \frac{3(c(D) - 2q) + c(D) + 2(B_D - 1)}{4}  = c(D) + \frac{B_D - 1}{2} - \frac{3q}{2}.
\end{equation}

If $D$ is $A$-adequate, then 
\begin{equation*}
    \min \deg V =  \frac{c(D) - A_D + 1 }{2}-\frac{3q}{2}.
\end{equation*}

 If $D$ is $B$-adequate, then 
\begin{equation}\label{B-adequate}
    \max \deg V = c(D) + \frac{B_D -1}{2}- \frac{3q}{2}.
\end{equation}

Further, we see that if $D$ is positive, then 
$$\min \deg V = \frac{c(D) - A_D + 1}{2} \hspace{20pt} \text{ and } \hspace{20pt} \max \deg V \leq c(D) + \frac{B_D -1}{2}.$$

\newpage 

\bibliographystyle{amsplain}

\bibliography{bib}

\end{document}

%% file: Figures/A_and_B_smoothings.tex
\begin{figure}[h]
    \centering
    \begin{tikzpicture}[every path/.style={thick}, every
node/.style={transform shape, knot crossing, inner sep=2.75pt}, scale=0.6]
    \node (tl) at (-1, 1) {};
    \node (tr) at (1, 1) {};
    \node (bl) at (-1, -1) {};
    \node (br) at (1, -1) {};
    \node (c) at (0,0) {};

    \draw (bl) -- (tr);
    \draw (br) -- (c);
    \draw (c) -- (tl);
    
    \begin{scope}[xshift=4cm]
    \node (tl) at (-1, 1) {};
    \node (tr) at (1, 1) {};
    \node (bl) at (-1, -1) {};
    \node (br) at (1, -1) {};

    \draw (bl) .. controls (bl.8 north east) and (tl.8 south east) .. (tl);
    \draw (br) .. controls (br.8 north west) and (tr.8 south west) .. (tr);
    \end{scope}

    \begin{scope}[xshift=8cm]
    \node (tl) at (-1, 1) {};
    \node (tr) at (1, 1) {};
    \node (bl) at (-1, -1) {};
    \node (br) at (1, -1) {};

    \draw (bl) .. controls (bl.8 north east) and (br.8 north west) .. (br);
    \draw (tl) .. controls (tl.8 south east) and (tr.8 south west) .. (tr);
    \end{scope}

    \end{tikzpicture}
    \caption{ A crossing (left), its $A$-smoothing (middle), and $B$-smoothing (right)}
    \label{A and B smoothing}
\end{figure}
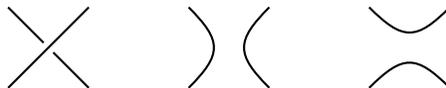

%% file: Figures/reduced_A-state_graph_example.tex

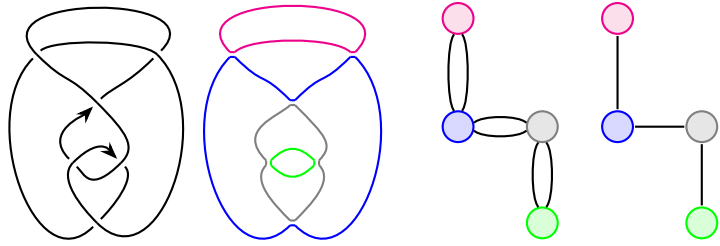
\begin{figure}[h]
    \centering
    \begin{tikzpicture}[every path/.style={thick}, every
node/.style={transform shape, knot crossing, inner sep=2pt}]

\begin{scope}[scale=0.8]
 
    \begin{scope}[xshift=0cm, scale=0.8]
    \node (tlo) at (-6.25,2.25) {};
    \node (tro) at (-3.75,2.25) {};
    \node (so) at (-5,1.25) {};
    \node (mlo) at (-5.5,0) {};
    \node (mro) at (-4.5, 0) {};
    \node (bo) at (-5,-1.25) {};
    \draw (bo.center) .. controls (bo.4 north west) and (mlo.4 south west) ..
    (mlo.center);
    \draw  (bo) .. controls (bo.4 north east) and (mro.4 south east).. (mro);
    \draw  [-{Stealth}] (mlo) .. controls (mlo.8 north west) and (so.3 south west) ..
    (so);
    \draw [-{Stealth}] (mlo.center) .. controls (mlo.8 north east) and (mro.2 north
    west) .. (mro);
    \draw  (mro.center) .. controls (mro.4 north east) and (so.8 south east) ..
    (so.center);
    \draw (mro.center) .. controls (mro.8 south west) and (mlo.3 south east) ..
    (mlo);
    \draw (so) .. controls (so.4 north east) and (tro.8 south west) .. (tro);
    \draw  (so.center) .. controls (so.8 north west) and (tlo.8 south
    east) .. (tlo.center);
    \draw (tlo.center) .. controls (tlo.16 north west) and (tro.16 north
    east) .. (tro);
    \draw (bo.center) .. controls (bo.16 south east) and (tro.16 south east) ..
    (tro.center);
    \draw (bo) .. controls (bo.16 south west) and (tlo.16 south
    west) .. (tlo);
    \draw (tlo) .. controls (tlo.4 north east) and (tro.4 north
    west) .. (tro.center);
    \end{scope}
    
    \begin{scope}[xshift=-0.75cm, scale=0.8]
    \node (tl) at (-1.25,2.25) {};
    \node (tltl) at (-1.3,2.3) {};
    \node (tlbl) at (-1.3,2.2) {};
    \node (tltr) at (-1.2,2.3) {};
    \node (tlbr) at (-1.2,2.2) {};
    \node (tr) at (1.25,2.25) {};
    \node (trtl) at (1.2,2.3) {};
    \node (trbl) at (1.2,2.2) {};
    \node (trtr) at (1.3,2.3) {};
    \node (trbr) at (1.3,2.2) {};
    \node (stl) at (-0.05,1.3) {};
    \node (str) at (0.05,1.3) {};
    \node (sbl) at (-0.05,1.2) {};
    \node (sbr) at (0.05,1.2) {};
    \node (ml) at (-0.5,0) {};
    \node (mltl) at (-0.55,0.05) {};
    \node (mlbl) at (-0.55,-0.05) {};
    \node (mltr) at (-0.45,0.05) {};
    \node (mlbr) at (-0.45,-0.05) {};
    \node (mr) at (0.5, 0) {};
    \node (mrtl) at (0.45, 0.05) {};
    \node (mrbl) at (0.45, -0.05) {};
    \node (mrtr) at (0.55, 0.05) {};
    \node (mrbr) at (0.55, -0.05) {};
    \node (b) at (0,-1.25) {};
    \node (btl) at (-0.05, -1.2) {};
    \node (bbl) at (-0.05, -1.3) {};
    \node (btr) at (0.05, -1.2) {};
    \node (bbr) at (0.05, -1.3) {};
    
    \draw [gray] (btl.center) .. controls (btl.4 north west) and (mlbl.4 south west) ..
    (mlbl.center);
    \draw [gray] (btr.center) .. controls (btr.4 north east) and (mrbr.4 south east)
    .. (mrbr.center);
    \draw [gray] (mltl.center) .. controls (mltl.8 north west) and (sbl.3 south west) ..
    (sbl.center);
    \draw [green] (mltr.center) .. controls (mltr.4 north east) and (mrtl.4 north
    west) .. (mrtl.center);
    \draw [gray] (mrtr.center) .. controls (mrtr.4 north east) and (sbr.8 south east) ..
    (sbr.center);
    \draw [green] (mrbl.center) .. controls (mrbl.4 south west) and (mlbr.4 south east) ..
    (mlbr.center);
    \draw [blue] (str.center) .. controls (str.8 north east) and (trbl.8 south west) .. (trbl.center);
    \draw [blue] (stl.center) .. controls (stl.8 north west) and (tlbr.8 south
    east) .. (tlbr.center);
    \draw [magenta] (tltl.center) .. controls (tltl.16 north west) and (trtr.16 north
    east) .. (trtr.center);
    \draw [blue] (bbr.center) .. controls (bbr.16 south east) and (trbr.16 south east) ..
    (trbr.center);
    \draw [blue] (bbl.center) .. controls (bbl.16 south west) and (tlbl.16 south
    west) .. (tlbl.center);
    \draw [magenta] (tltr.center) .. controls (tltr.4 north east) and (trtl.4 north
    west) .. (trtl.center);
    
    \draw [blue] (stl.center) -- (str.center) {};
    \draw [gray] (sbl.center) -- (sbr.center) {};
    \draw [gray] (mltl.center) -- (mlbl.center) {};
    \draw [green] (mltr.center) -- (mlbr.center) {};
    \draw [green] (mrtl.center) -- (mrbl.center) {};
    \draw [gray] (mrtr.center) -- (mrbr.center) {};
    \draw [magenta] (tltl.center) -- (tltr.center) {};
    \draw [blue] (tlbl.center) -- (tlbr.center) {};
    \draw [magenta] (trtl.center) -- (trtr.center) {};
    \draw [blue] (trbl.center) -- (trbr.center) {};
    \draw [gray] (btl.center) -- (btr.center) {};
    \draw [blue] (bbl.center) -- (bbr.center) {};
    \end{scope}

    
    \begin{scope}[xshift=2cm, scale =0.8]
    \node (magenta) at (0, 2.75) {};
    \node (m) at (0, 3) {};
    \node (blue) at (0, 1) {};
    \node (c) at (0, .75) {};
    \node (cr) at (0.25, 0.75) {};
    \node (orange) at (1.5, 0.75) {};
    \node (o) at (1.75, 0.75) {};
    \node (ou) at (1.75, 0.5) {};
    \node (ol) at (1.5, 0.75) {};
    \node (green) at (1.75, -1) {};
    \node (g) at (1.75, -1.25) {};

     \filldraw [color=magenta, fill=magenta!15] (m) circle[radius=0.32cm];
    \filldraw [color=blue, fill=blue!15] (c) circle[radius=0.32cm];
    \filldraw [color=gray, fill=gray!20] (o) circle[radius=0.32cm];
    \filldraw [color=green, fill=green!15] (g) circle[radius=0.32cm];
    \draw (blue) .. controls (blue.3 north east) and (magenta.3 south east) .. (magenta);
    \draw (blue) .. controls (blue.3 north west) and (magenta.3 south west) .. (magenta);
    \draw (orange) .. controls (orange.3 north west) and (cr.3 north east) .. (cr);
    \draw (orange) .. controls (orange.3 south west) and (cr.3 south east) .. (cr);

    \draw (green) .. controls (green.3 north west) and (ou.3 south west) .. (ou);
    \draw (green) .. controls (green.3 north east) and (ou.3 south east) .. (ou);
    \end{scope}
    
    \begin{scope}[xshift=4.65cm, scale=0.8]
    \node (magenta) at (0, 2.75) {};
    \node (m) at (0, 3) {};
    \node (cyan) at (0, 1) {};
    \node (c) at (0, .75) {};
    \node (cr) at (0.25, 0.75) {};
    \node (orange) at (1.5, 0.75) {};
    \node (o) at (1.75, 0.75) {};
    \node (ou) at (1.75, 0.5) {};
    \node (green) at (1.75, -1) {};
    \node (g) at (1.75, -1.25) {};
    
   \filldraw [color=magenta, fill=magenta!15] (m) circle[radius=0.32cm];
    \filldraw [color=blue, fill=blue!15] (c) circle[radius=0.32cm];
    \filldraw [color=gray, fill=gray!20] (o) circle[radius=0.32cm];
    \filldraw [color=green, fill=green!15] (g) circle[radius=0.32cm];
    \draw (cyan) -- (magenta);
    \draw (orange) -- (cr);
    \draw (green) -- (ou);
    \end{scope}

\end{scope}
\end{tikzpicture} 
\caption{ (Left to right:) A positive link diagram, its $A$-circles, its $A$-state graph, and its reduced $A$-state graph }
\label{reduced A-state graph example}
\end{figure}

%% file: Figures/Balanced_type_0_example.tex
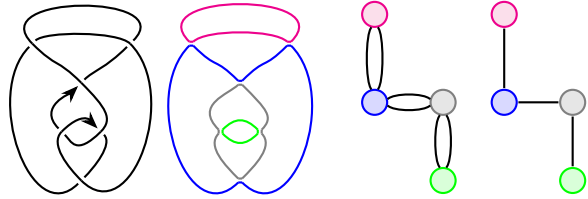
\begin{figure}[h]
    \centering
    \begin{tikzpicture}[every path/.style={thick}, every
node/.style={transform shape, knot crossing, inner sep=2pt}, scale=0.65]
    \begin{scope}[xshift=0cm, scale=0.8]
    \node (tlo) at (-6.25,2.25) {};
    \node (tro) at (-3.75,2.25) {};
    \node (so) at (-5,1.25) {};
    \node (mlo) at (-5.5,0) {};
    \node (mro) at (-4.5, 0) {};
    \node (bo) at (-5,-1.25) {};
    \draw (bo.center) .. controls (bo.4 north west) and (mlo.4 south west) ..
    (mlo.center);
    \draw  (bo) .. controls (bo.4 north east) and (mro.4 south east).. (mro);
    \draw  [-{Stealth}] (mlo) .. controls (mlo.8 north west) and (so.3 south west) ..
    (so);
    \draw [-{Stealth}] (mlo.center) .. controls (mlo.8 north east) and (mro.2 north
    west) .. (mro);
    \draw  (mro.center) .. controls (mro.4 north east) and (so.8 south east) ..
    (so.center);
    \draw (mro.center) .. controls (mro.8 south west) and (mlo.3 south east) ..
    (mlo);
    \draw (so) .. controls (so.4 north east) and (tro.8 south west) .. (tro);
    \draw  (so.center) .. controls (so.8 north west) and (tlo.8 south
    east) .. (tlo.center);
    \draw (tlo.center) .. controls (tlo.16 north west) and (tro.16 north
    east) .. (tro);
    \draw (bo.center) .. controls (bo.16 south east) and (tro.16 south east) ..
    (tro.center);
    \draw (bo) .. controls (bo.16 south west) and (tlo.16 south
    west) .. (tlo);
    \draw (tlo) .. controls (tlo.4 north east) and (tro.4 north
    west) .. (tro.center);
    \end{scope}
    
    \begin{scope}[xshift=-0.75cm, scale=0.8]
    \node (tl) at (-1.25,2.25) {};
    \node (tltl) at (-1.3,2.3) {};
    \node (tlbl) at (-1.3,2.2) {};
    \node (tltr) at (-1.2,2.3) {};
    \node (tlbr) at (-1.2,2.2) {};
    \node (tr) at (1.25,2.25) {};
    \node (trtl) at (1.2,2.3) {};
    \node (trbl) at (1.2,2.2) {};
    \node (trtr) at (1.3,2.3) {};
    \node (trbr) at (1.3,2.2) {};
    \node (stl) at (-0.05,1.3) {};
    \node (str) at (0.05,1.3) {};
    \node (sbl) at (-0.05,1.2) {};
    \node (sbr) at (0.05,1.2) {};
    \node (ml) at (-0.5,0) {};
    \node (mltl) at (-0.55,0.05) {};
    \node (mlbl) at (-0.55,-0.05) {};
    \node (mltr) at (-0.45,0.05) {};
    \node (mlbr) at (-0.45,-0.05) {};
    \node (mr) at (0.5, 0) {};
    \node (mrtl) at (0.45, 0.05) {};
    \node (mrbl) at (0.45, -0.05) {};
    \node (mrtr) at (0.55, 0.05) {};
    \node (mrbr) at (0.55, -0.05) {};
    \node (b) at (0,-1.25) {};
    \node (btl) at (-0.05, -1.2) {};
    \node (bbl) at (-0.05, -1.3) {};
    \node (btr) at (0.05, -1.2) {};
    \node (bbr) at (0.05, -1.3) {};
    
    \draw [gray] (btl.center) .. controls (btl.4 north west) and (mlbl.4 south west) ..
    (mlbl.center);
    \draw [gray] (btr.center) .. controls (btr.4 north east) and (mrbr.4 south east)
    .. (mrbr.center);
    \draw [gray] (mltl.center) .. controls (mltl.8 north west) and (sbl.3 south west) ..
    (sbl.center);
    \draw [green] (mltr.center) .. controls (mltr.4 north east) and (mrtl.4 north
    west) .. (mrtl.center);
    \draw [gray] (mrtr.center) .. controls (mrtr.4 north east) and (sbr.8 south east) ..
    (sbr.center);
    \draw [green] (mrbl.center) .. controls (mrbl.4 south west) and (mlbr.4 south east) ..
    (mlbr.center);
    \draw [blue] (str.center) .. controls (str.8 north east) and (trbl.8 south west) .. (trbl.center);
    \draw [blue] (stl.center) .. controls (stl.8 north west) and (tlbr.8 south
    east) .. (tlbr.center);
    \draw [magenta] (tltl.center) .. controls (tltl.16 north west) and (trtr.16 north
    east) .. (trtr.center);
    \draw [blue] (bbr.center) .. controls (bbr.16 south east) and (trbr.16 south east) ..
    (trbr.center);
    \draw [blue] (bbl.center) .. controls (bbl.16 south west) and (tlbl.16 south
    west) .. (tlbl.center);
    \draw [magenta] (tltr.center) .. controls (tltr.4 north east) and (trtl.4 north
    west) .. (trtl.center);
    
    \draw [blue] (stl.center) -- (str.center) {};
    \draw [gray] (sbl.center) -- (sbr.center) {};
    \draw [gray] (mltl.center) -- (mlbl.center) {};
    \draw [green] (mltr.center) -- (mlbr.center) {};
    \draw [green] (mrtl.center) -- (mrbl.center) {};
    \draw [gray] (mrtr.center) -- (mrbr.center) {};
    \draw [magenta] (tltl.center) -- (tltr.center) {};
    \draw [blue] (tlbl.center) -- (tlbr.center) {};
    \draw [magenta] (trtl.center) -- (trtr.center) {};
    \draw [blue] (trbl.center) -- (trbr.center) {};
    \draw [gray] (btl.center) -- (btr.center) {};
    \draw [blue] (bbl.center) -- (bbr.center) {};
    \end{scope}

    
    \begin{scope}[xshift=2cm, scale =0.8]
    \node (magenta) at (0, 2.75) {};
    \node (m) at (0, 3) {};
    \node (blue) at (0, 1) {};
    \node (c) at (0, .75) {};
    \node (cr) at (0.25, 0.75) {};
    \node (orange) at (1.5, 0.75) {};
    \node (o) at (1.75, 0.75) {};
    \node (ou) at (1.75, 0.5) {};
    \node (ol) at (1.5, 0.75) {};
    \node (green) at (1.75, -1) {};
    \node (g) at (1.75, -1.25) {};

    \filldraw [color=magenta, fill=magenta!15] (m) circle[radius=0.32cm];
    \filldraw [color=blue, fill=blue!15] (c) circle[radius=0.32cm];
    \filldraw [color=gray, fill=gray!20] (o) circle[radius=0.32cm];
    \filldraw [color=green, fill=green!15] (g) circle[radius=0.32cm];
    \draw (blue) .. controls (blue.3 north east) and (magenta.3 south east) .. (magenta);
    \draw (blue) .. controls (blue.3 north west) and (magenta.3 south west) .. (magenta);
    \draw (orange) .. controls (orange.3 north west) and (cr.3 north east) .. (cr);
    \draw (orange) .. controls (orange.3 south west) and (cr.3 south east) .. (cr);

    \draw (green) .. controls (green.3 north west) and (ou.3 south west) .. (ou);
    \draw (green) .. controls (green.3 north east) and (ou.3 south east) .. (ou);
    \end{scope}
    
    \begin{scope}[xshift=4.65cm, scale=0.8]
    \node (magenta) at (0, 2.75) {};
    \node (m) at (0, 3) {};
    \node (cyan) at (0, 1) {};
    \node (c) at (0, .75) {};
    \node (cr) at (0.25, 0.75) {};
    \node (orange) at (1.5, 0.75) {};
    \node (o) at (1.75, 0.75) {};
    \node (ou) at (1.75, 0.5) {};
    \node (green) at (1.75, -1) {};
    \node (g) at (1.75, -1.25) {};
    
    \filldraw [color=magenta, fill=magenta!15] (m) circle[radius=0.32cm];
    \filldraw [color=blue, fill=blue!15] (c) circle[radius=0.32cm];
    \filldraw [color=gray, fill=gray!20] (o) circle[radius=0.32cm];
    \filldraw [color=green, fill=green!15] (g) circle[radius=0.32cm];
    \draw (cyan) -- (magenta);
    \draw (orange) -- (cr);
    \draw (green) -- (ou);
    \end{scope}
\end{tikzpicture} 
\vspace{-7pt}
\caption{ (Left to right:) A Balanced diagram of type $0$, its $A$-circles, its $A$-state graph, and its reduced $A$-state graph }
\label{Balanced type 0 example}
\end{figure}

%% file: Figures/Balanced_type_1_example_2.tex
\begin{figure}[h]
    \centering
    \begin{tikzpicture}
      \begin{scope}[scale=0.6, >=Stealth]

      \begin{scope}[scale=0.5, xshift=5cm, yshift=2cm, rotate=45]
       \node (a) at (-3,0){};
       \node (b) at (-1,0){};
       \node (c) at (0,3) {};
       \node (d) at (0,1){};
       \node (e) at (1,0) {};
       \node (f) at (3,0){};
       \node (g) at (3,-4) {};
       \node (h) at (0, -4){};
       \node (i) at (-3,-4){};
       \node (j) at (-3, -3){};
       \node (k) at (-1, -3){};
       \node (l) at (0, -2){};
       

       \draw [out=100, in=100, relative, decoration={markings, mark=at position 0.5 with {\arrow{<}}}, postaction={decorate}] 
       (a.center) to (c);
       \draw [out=80, in=100, relative, decoration={markings, mark=at position 0.5 with {\arrow{>}}}, postaction={decorate}] 
       (c.center) to (f);
       \draw [out=45, in=135, relative, decoration={markings, mark=at position 0.5 with {\arrow{}}}, postaction={decorate}] 
       (f.center) to (g.center);
       \draw [out=45, in=135, relative, decoration={markings, mark=at position 0.5 with {\arrow{<}}}, postaction={decorate}] 
       (g.center) to (h);
       \draw [out=45, in=135, relative, decoration={markings, mark=at position 0.5 with {\arrow{>}}}, postaction={decorate}] 
       (h.center) to (i.center);
       \draw [out=45, in=135, relative, decoration={markings, mark=at position 0.5 with {\arrow{}}}, postaction={decorate}] 
       (i.center) to (j.center);
       \draw [out=45, in=135, relative, decoration={markings, mark=at position 0.5 with {\arrow{}}}, postaction={decorate}] 
       (j) to (a);

    \draw[]
       (a) to (c.center)
       (c) to (f.center);
       
       \draw [out=0, in=180, relative, decoration={markings, mark=at position 0.5 with {\arrow{}}}, postaction={decorate}] 
       (f) to (h.center);
       \draw [decoration={markings, mark=at position 0.5 with {\arrow{<}}}, postaction={decorate}] 
       (h) to (k);
       \draw [out=-60, in=-160, relative, decoration={markings, mark=at position 0.5 with {\arrow{<}}}, postaction={decorate}] 
       (k.center) to (a.center);


       \draw [out=45, in=135, relative, decoration={markings, mark=at position 0.75 with {\arrow{}}}, postaction={decorate}] 
       (j.center) to (k);
       \draw [out=45, in=135, relative, decoration={markings, mark=at position 0.5 with {\arrow{}}}, postaction={decorate}] 
       (k.center) to (j);
       \end{scope}


      \begin{scope}[scale=0.5, xshift=15cm, yshift=2cm, rotate=45]
       \node (a) at (-3,0){};
       \node (b) at (-1,0){};
       \node (c) at (0,3) {};
       \node (d) at (0,1){};
       \node (e) at (1,0) {};
       \node (f) at (3,0){};
       \node (g) at (3,-4) {};
       \node (h) at (0, -4){};
       \node (i) at (-3,-4){};
       \node (j) at (-3, -3){};
       \node (k) at (-1, -3){};
       \node (l) at (0, -2){};
       

       \draw [purp, thick, out=100, in=100, relative, decoration={markings, mark=at position 0.5 with {\arrow{}}}, postaction={decorate}] 
       (a.center) to (c);
       \draw [magenta, thick, out=80, in=100, relative, decoration={markings, mark=at position 0.5 with {\arrow{}}}, postaction={decorate}] 
       (c.center) to (f);
       \draw [Emerald, thick, out=45, in=135, relative, decoration={markings, mark=at position 0.5 with {\arrow{}}}, postaction={decorate}] 
       (f.center) to (g.center);
       \draw [Emerald, thick, out=45, in=135, relative, decoration={markings, mark=at position 0.5 with {\arrow{}}}, postaction={decorate}] 
       (g.center) to (h);
       \draw [blue, thick, out=45, in=135, relative, decoration={markings, mark=at position 0.5 with {\arrow{}}}, postaction={decorate}] 
       (h.center) to (i.center);
       \draw [blue, thick, out=45, in=135, relative, decoration={markings, mark=at position 0.5 with {\arrow{}}}, postaction={decorate}] 
       (i.center) to (j.center);
       \draw [blue, thick, out=45, in=135, relative, decoration={markings, mark=at position 0.5 with {\arrow{}}}, postaction={decorate}] 
       (j) to (a);

    \draw[purp, thick]
       (a) to (c.center);
    \draw[magenta, thick]
       (c) to (f.center);
       
       \draw [Emerald, thick, out=0, in=180, relative, decoration={markings, mark=at position 0.5 with {\arrow{}}}, postaction={decorate}] 
       (f) to (h.center);
       \draw [blue, thick, decoration={markings, mark=at position 0.5 with {\arrow{}}}, postaction={decorate}] 
       (h) to (k);
       \draw [blue, thick, out=-60, in=-160, relative, decoration={markings, mark=at position 0.5 with {\arrow{}}}, postaction={decorate}] 
       (k.center) to (a.center);


       \draw [red, thick, out=45, in=135, relative, decoration={markings, mark=at position 0.5 with {\arrow{}}}, postaction={decorate}] 
       (j.center) to (k);
       \draw [red, thick, out=45, in=135, relative, decoration={markings, mark=at position 0.5 with {\arrow{}}}, postaction={decorate}] 
       (k.center) to (j);
       \end{scope}


  \begin{scope}[scale = 0.85, xshift=15cm, yshift=1.5cm, rotate = 45]
  
       \node (a) at (-1,1){};
       \node (b) at (1,1){};
       \node (c) at (1,-1) {};
       \node (d) at (-1,-1){};
       \node (e) at (-2, -2) {};
        \draw[]
        (a.center) -- (b.center)
        (b.center) -- (c.center) 
        (c.center) -- (d.center)
        (d.center) -- (a.center);
         \draw[out=20, in=160, relative]
        (d) to (e)
        (e) to (d);

        \filldraw[color=purp, fill=purp!5] 
        (a) circle (1/3);
          \filldraw[color=magenta, fill=magenta!5] 
        (b) circle (1/3);
          \filldraw[color=Emerald, fill=Emerald!5] 
        (c) circle (1/3);
          \filldraw[color=blue, fill=blue!5] 
        (d) circle (1/3) ;
          \filldraw[color=red, fill=red!5] 
        (e) circle (1/3);
        
       \end{scope}

 \begin{scope}[scale = 0.85, xshift=20cm, yshift=1.5cm, rotate = 45]
  
       \node (a) at (-1,1){};
       \node (b) at (1,1){};
       \node (c) at (1,-1) {};
       \node (d) at (-1,-1){};
       \node (e) at (-2, -2) {};
        \draw[]
        (a.center) -- (b.center)
        (b.center) -- (c.center) 
        (c.center) -- (d.center)
        (d.center) -- (a.center)
        (d.center) -- (e.center);

          \filldraw[color=purp, fill=purp!5] 
        (a) circle (1/3);
          \filldraw[color=magenta, fill=magenta!5] 
        (b) circle (1/3);
          \filldraw[color=Emerald, fill=Emerald!5] 
        (c) circle (1/3);
          \filldraw[color=blue, fill=blue!5] 
        (d) circle (1/3) ;
          \filldraw[color=red, fill=red!5] 
        (e) circle (1/3);
        
       \end{scope}
          
       
       \end{scope}
       
  \end{tikzpicture}
    \vspace{-5pt}
    \caption{A Balanced diagram of type $1$, its $A$-circles, its $A$-state graph, and its reduced $A$-state graph}
    \label{fig:Balanced type 1 example 2}
\end{figure}
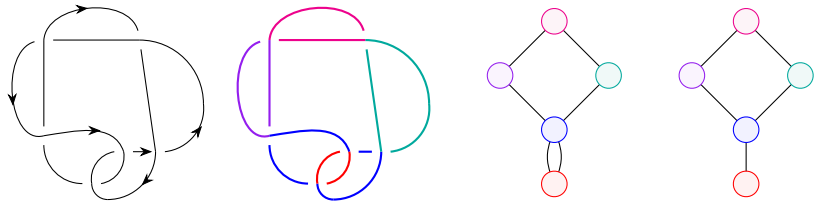

%% file: Figures/Balanced_type_2_example.tex
\begin{figure}[h]
    \centering
  \begin{tikzpicture}
      \begin{scope}[scale=0.5, >=Stealth]

      \begin{scope}

      \begin{scope}[rotate=-90]

       \node (a) at (-4,2){};
       \node (b) at (-4,3){};
       \node (c) at (-3,3) {};
       \node (d) at (-2,3){};
       \node (e) at (-1,3) {};
       \node (f) at (0,3){};
       \node (g) at (1,3) {};
       \node (h) at (2,3){};
       \node (i) at (2,2) {};
       \node (j) at (2,-1){};
       \node (k) at (1,-1) {};
       \node (l) at (0,-1){};
       \node (m) at (-1,-1) {};
       \node (n) at (-2,-1){};
       \node (o) at (-3,-1) {};
       \node (p) at (-4,-1){};

       \node (p1) at (-4, 0){};
       \node (p2) at (-4, 1){};
        \node (i1) at (2, 1){};
        \node (i2) at (2,0){};

        \node (mt) at (0, 2){};
        \node (mm) at (0, 1) {};
        \node (mb) at (0, 0) {};

        \node (km) at (1,1) {};
        \node (right) at (3, 1) {};
        \node (outright) at (4,1) {};


       \draw [out=45, in=135, relative, decoration={markings, mark=at position 0.15 with {\arrow{>}}}, postaction={decorate}] 
       (b.center) to (c);
       \draw [out=60, in=135, relative, decoration={markings, mark=at position 0.5 with {\arrow{}}}, postaction={decorate}] 
       (c.center) to (d);
       \draw [decoration={markings, mark=at position 0.5 with {\arrow{>}}}, postaction={decorate}] 
       (d.center) ..controls (2,5) and (5,3).. (outright.center);

       \draw [out=100, in=135, relative, decoration={markings, mark=at position 0.5 with {\arrow{>}}}, postaction={decorate}] 
       (outright) to (n);
       \draw [out=35, in=105, relative, decoration={markings, mark=at position 0.5 with {\arrow{}}}, postaction={decorate}] 
       (n.center) to (o);
       \draw [out=45, in=145, relative, decoration={markings, mark=at position 0.5 with {\arrow{}}}, postaction={decorate}] 
       (o.center) to (p.center);

       \draw [ out=45, in=135, relative, decoration={markings, mark=at position 0.5 with {\arrow{}}}, postaction={decorate}] 
       (p.center) to (p1);
       \draw [out=80, in=100, relative, decoration={markings, mark=at position 0.55 with {\arrow{<}}}, postaction={decorate}] 
       (p1.center) to (a);
       \draw [out=60, in=135, relative, decoration={markings, mark=at position 0.5 with {\arrow{}}}, postaction={decorate}] 
       (a.center) to (b.center);


        \draw [out = -30, in =-150, relative, decoration={markings, mark=at position 0.5 with {\arrow{}}}, postaction={decorate}]
       (o) to (p1.center);
       \draw [out = -45, in =-135, relative, decoration={markings, mark=at position 0.5 with {\arrow{>}}}, postaction={decorate}]
       (p1) to (a.center);
       \draw [out = -30, in =-150, relative, decoration={markings, mark=at position 0.5 with {\arrow{}}}, postaction={decorate}]
       (a) to (c.center);

       \draw [out = -45, in =-160, relative, decoration={markings, mark=at position 0.5 with {\arrow{}}}, postaction={decorate}]
       (c) to (d.center);
       \draw [out =0, in =180, relative, decoration={markings, mark=at position 0.5 with {\arrow{<}}}, postaction={decorate}]
       (d) to (mm);
        \draw [out = 35, in =135, relative, decoration={markings, mark=at position 0.5 with {\arrow{}}}, postaction={decorate}]
       (mm.center) to (km);
        \draw [out = 45, in =135, relative, decoration={markings, mark=at position 0.5 with {\arrow{<}}}, postaction={decorate}]
       (km.center) to (right);
       \draw [out = 45, in =135, relative, decoration={markings, mark=at position 0.5 with {\arrow{}}}, postaction={decorate}]
       (right.center) to (outright);

        \draw [out = 45, in =135, relative, decoration={markings, mark=at position 0.5 with {\arrow{}}}, postaction={decorate}]
       (outright.center) to (right);
        \draw [out = 45, in =135, relative, decoration={markings, mark=at position 0.5 with {\arrow{<}}}, postaction={decorate}]
       (right.center) to (km);
        \draw [out = 45, in =135, relative, decoration={markings, mark=at position 0.5 with {\arrow{}}}, postaction={decorate}]
       (km.center) to (mm);
       \draw [out = 0, in =180, relative, decoration={markings, mark=at position 0.5 with {\arrow{<}}}, postaction={decorate}]
       (mm.center) to (n.center);

       \draw [out = -45, in =-145, relative, decoration={markings, mark=at position 0.5 with {\arrow{}}}, postaction={decorate}]
       (n) to (o.center);

              
      \end{scope}

      \begin{scope}[xshift=7cm, rotate=-90]

       \node (a) at (-4,2){};
       \node (b) at (-4,3){};
       \node (c) at (-3,3) {};
       \node (d) at (-2,3){};
       \node (e) at (-1,3) {};
       \node (f) at (0,3){};
       \node (g) at (1,3) {};
       \node (h) at (2,3){};
       \node (i) at (2,2) {};
       \node (j) at (2,-1){};
       \node (k) at (1,-1) {};
       \node (l) at (0,-1){};
       \node (m) at (-1,-1) {};
       \node (n) at (-2,-1){};
       \node (o) at (-3,-1) {};
       \node (p) at (-4,-1){};

       \node (p1) at (-4, 0){};
       \node (p2) at (-4, 1){};
        \node (i1) at (2, 1){};
        \node (i2) at (2,0){};

        \node (mt) at (0, 2){};
        \node (mm) at (0, 1) {};
        \node (mb) at (0, 0) {};

        \node (km) at (1,1) {};
        \node (right) at (3, 1) {};
        \node (outright) at (4,1) {};


       \draw [blue, thick, out=45, in=155, relative, decoration={markings, mark=at position 0.15 with {\arrow{}}}, postaction={decorate}] 
       (b.center) to (c);
       \draw [magenta,thick, out=60, in=135, relative, decoration={markings, mark=at position 0.5 with {\arrow{}}}, postaction={decorate}] 
       (c.center) to (d);
       \draw [cyan, thick, decoration={markings, mark=at position 0.5 with {\arrow{}}}, postaction={decorate}] 
       (d.center) ..controls (2,5) and (5,3).. (outright.center);

       \draw [cyan,thick, out=100, in=135, relative, decoration={markings, mark=at position 0.5 with {\arrow{}}}, postaction={decorate}] 
       (outright) to (n);
       \draw [orange,thick, out=35, in=105, relative, decoration={markings, mark=at position 0.5 with {\arrow{}}}, postaction={decorate}] 
       (n.center) to (o);
       \draw [green,thick, out=45, in=145, relative, decoration={markings, mark=at position 0.5 with {\arrow{}}}, postaction={decorate}] 
       (o.center) to (p.center);

       \draw [green,thick, out=45, in=135, relative, decoration={markings, mark=at position 0.5 with {\arrow{}}}, postaction={decorate}] 
       (p.center) to (p1);
       \draw [red,thick, out=80, in=100, relative, decoration={markings, mark=at position 0.55 with {\arrow{}}}, postaction={decorate}] 
       (p1.center) to (a);
       \draw [blue,thick, out=60, in=135, relative, decoration={markings, mark=at position 0.5 with {\arrow{}}}, postaction={decorate}] 
       (a.center) to (b.center);


        \draw [green,thick, out = -30, in =-150, relative, decoration={markings, mark=at position 0.5 with {\arrow{}}}, postaction={decorate}]
       (o) to (p1.center);
       \draw [red, thick, out = -45, in =-135, relative, decoration={markings, mark=at position 0.5 with {\arrow{}}}, postaction={decorate}]
       (p1) to (a.center);
       \draw [blue,thick,  out = -30, in =-150, relative, decoration={markings, mark=at position 0.5 with {\arrow{}}}, postaction={decorate}]
       (a) to (c.center);

       \draw [magenta,thick,  out = -45, in =-160, relative, decoration={markings, mark=at position 0.5 with {\arrow{}}}, postaction={decorate}]
       (c) to (d.center);
       \draw [cyan, thick, out =0, in =180, relative, decoration={markings, mark=at position 0.5 with {\arrow{}}}, postaction={decorate}]
       (d) to (mm);
        \draw [purp, thick, out = 35, in =135, relative, decoration={markings, mark=at position 0.5 with {\arrow{}}}, postaction={decorate}]
       (mm.center) to (km);
        \draw [Emerald,thick, out = 45, in =135, relative, decoration={markings, mark=at position 0.5 with {\arrow{}}}, postaction={decorate}]
       (km.center) to (right);
       \draw [gray,thick, out = 45, in =135, relative, decoration={markings, mark=at position 0.5 with {\arrow{}}}, postaction={decorate}]
       (right.center) to (outright);

        \draw [gray, thick, out = 45, in =135, relative, decoration={markings, mark=at position 0.5 with {\arrow{}}}, postaction={decorate}]
       (outright.center) to (right);
        \draw [Emerald, thick, out = 45, in =135, relative, decoration={markings, mark=at position 0.5 with {\arrow{}}}, postaction={decorate}]
       (right.center) to (km);
        \draw [purp, thick, out = 45, in =135, relative, decoration={markings, mark=at position 0.5 with {\arrow{}}}, postaction={decorate}]
       (km.center) to (mm);
       \draw [cyan,thick, out = 0, in =180, relative, decoration={markings, mark=at position 0.5 with {\arrow{}}}, postaction={decorate}]
       (mm.center) to (n.center);

       \draw [orange, thick, out = -45, in =-145, relative, decoration={markings, mark=at position 0.5 with {\arrow{}}}, postaction={decorate}]
       (n) to (o.center);

              
      \end{scope}

  \begin{scope}[scale = 1, xshift=12cm, yshift= -0.5cm, rotate = -90]

        \node (a) at (-3,1.5){};
       \node (b) at (-3,3){};
       \node (d) at (-2,3){};
       \node (e) at (-1,3) {};
       \node (f) at (0,3){};
       \node (h) at (1,3){};
       \node (i) at (1,1.5) {};
       \node (j) at (1,0){};
       \node (l) at (0,0){};
       \node (m) at (-1,0) {};
       \node (n) at (-2,0){};
       \node (p) at (-3,0){};

       \node (p2) at (-3, 1.5){};
        \node (i1) at (1, 1.5){};

        \node (mt) at (0, 1.5){};
        \node (mm) at (0, 1) {};
        \node (mb) at (0, 0) {};

        \draw[]
        (b) -- (e)
        (e) -- (mt) 
        (mt) -- (h)
        (h) -- (j)
        (j) -- (mt)
        (mt) -- (m)
        (m) -- (p) 
        (p) -- (b)
        node[right, text centered, text width = 1cm, midway, xshift=0.25cm]{}  ;

        \filldraw[color=blue, fill=blue!8] 
        (b) circle (1/3);
\filldraw[color=magenta, fill=magenta!8]     
        (e) circle (1/3);
\filldraw[color=purp, fill=purp!8] 
        (h) circle (1/3);
\filldraw[color=Emerald, fill=Emerald!8] 
        (i1) circle (1/3);
        \filldraw[color=red, fill=red!8] 
        (p2) circle (1/3);
        \filldraw[color=gray, fill=gray!8] 
        (j) circle (1/3);
  \filldraw[color=orange, fill=orange!8] 
        (m) circle (1/3);
\filldraw[color=green, fill=green!8] 
        (p) circle (1/3);
\filldraw[color=cyan, fill=cyan!8] 
         (mt) circle (1/3);
       \end{scope}

  \begin{scope}[scale = 1, xshift=17cm, yshift= -0.5cm, rotate = -90]

        \node (a) at (-3,1.5){};
       \node (b) at (-3,3){};
       \node (d) at (-2,3){};
       \node (e) at (-1,3) {};
       \node (f) at (0,3){};
       \node (h) at (1,3){};
       \node (i) at (1,1.5) {};
       \node (j) at (1,0){};
       \node (l) at (0,0){};
       \node (m) at (-1,0) {};
       \node (n) at (-2,0){};
       \node (p) at (-3,0){};

       \node (p2) at (-3, 1.5){};
        \node (i1) at (1, 1.5){};

        \node (mt) at (0, 1.5){};
        \node (mm) at (0, 1) {};
        \node (mb) at (0, 0) {};

        \draw[]
        (b) -- (e)
        (e) -- (mt) 
        (mt) -- (h)
        (h) -- (j)
        (j) -- (mt)
        (mt) -- (m)
        (m) -- (p) 
        (p) -- (b)
        node[right, text centered, text width = 1cm, midway, xshift=0.25cm]{}  ;

        \filldraw[color=blue, fill=blue!8] 
        (b) circle (1/3);
\filldraw[color=magenta, fill=magenta!8]     
        (e) circle (1/3);
\filldraw[color=purp, fill=purp!8] 
        (h) circle (1/3);
\filldraw[color=Emerald, fill=Emerald!8] 
        (i1) circle (1/3);
        \filldraw[color=red, fill=red!8] 
        (p2) circle (1/3);
        \filldraw[color=gray, fill=gray!8] 
        (j) circle (1/3);
  \filldraw[color=orange, fill=orange!8] 
        (m) circle (1/3);
\filldraw[color=green, fill=green!8] 
        (p) circle (1/3);
\filldraw[color=cyan, fill=cyan!8] 
         (mt) circle (1/3);
       \end{scope}

          \end{scope}
       
       \end{scope}
  \end{tikzpicture}
  \vspace{-10pt}
    \caption{A Balanced diagram of type $2$, its $A$-circles, its $A$-state graph, and its reduced $A$-state graph}
    \label{fig:Balanced type 2 example}
\end{figure}
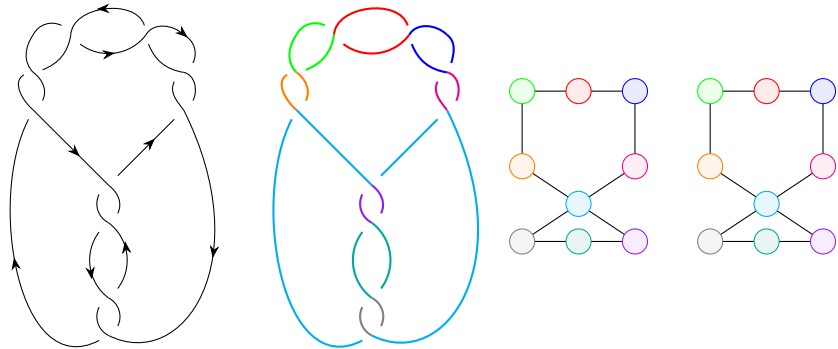

%% file: Figures/Oddly_Balanced_example.tex

\begin{figure}
    \centering
  \begin{tikzpicture}

  \begin{scope}[scale=0.7]
  
      \begin{scope}[scale=1, >=Stealth, xshift=-3.5cm, rotate around={-90:(-2.5,0)}]
       \node (a) at (-4,0){};
       \node (b) at (-4,1){};
       \node (c) at (-3,1) {};
       \node (d) at (-2,1){};
       \node (g) at (-2,1) {};
       \node (h) at (-1,1){};
       \node (i) at (-1,0) {};
       \node (j) at (-1,-1){};
       \node (k) at (-2,-1) {};
       \node (n) at (-2,-1){};
       \node (o) at (-3,-1) {};
       \node (p) at (-4,-1){};

       \node (middle) at (-2.5,0){};


       \draw [out=45, in=135, relative, decoration={markings, mark=at position 0.5 with {\arrow{}}}, postaction={decorate}] 
       (b.center) to (c);
       \draw [out=45, in=135, relative, decoration={markings, mark=at position 0.5 with {\arrow{}}}, postaction={decorate}] 
       (c.center) to (d);
       \draw [ out=60, in=135, relative, decoration={markings, mark=at position 0.5 with {\arrow{}}}, postaction={decorate}] 
       (g.center) to (h.center);

       \draw [out=45, in=135, relative, decoration={markings, mark=at position 0.5 with {\arrow{>}}}, postaction={decorate}] 
       (h.center) to (i);
       \draw [out=45, in=135, relative, decoration={markings, mark=at position 0.5 with {\arrow{}}}, postaction={decorate}] 
       (i.center) to (j.center);

       \draw [out=45, in=135, relative, decoration={markings, mark=at position 0.5 with {\arrow{<}}}, postaction={decorate}] 
       (j.center) to (k);
       \draw [out=45, in=135, relative, decoration={markings, mark=at position 0.5 with {\arrow{}}}, postaction={decorate}] 
       (n.center) to (o);
       \draw [out=60, in=135, relative, decoration={markings, mark=at position 0.5 with {\arrow{<}}}, postaction={decorate}] 
       (o.center) to (p.center);

       \draw [out=45, in=135, relative, decoration={markings, mark=at position 0.5 with {\arrow{}}}, postaction={decorate}] 
       (p.center) to (a);
       \draw [out=45, in=135, relative, decoration={markings, mark=at position 0.5 with {\arrow{>}}}, postaction={decorate}] 
       (a.center) to (b.center);


        \draw [decoration={markings, mark=at position 0.5 with {\arrow{>}}}, postaction={decorate}]
       (o) to (a.center);
       \draw [decoration={markings, mark=at position 0.5 with {\arrow{<}}}, postaction={decorate}]
       (a) to (c.center);


       \draw [decoration={markings, mark=at position 0.5 with {\arrow{<}}}, postaction={decorate}]
       (g) to (i.center);
       \draw [ decoration={markings, mark=at position 0.5 with {\arrow{>}}}, postaction={decorate}]
       (i) to (k.center);
       

       \draw [decoration={markings, mark=at position 0.5 with {\arrow{<}}}, postaction={decorate}]
       (d.center) to (middle);
       \draw []
       (middle) to (o.center);
       \draw []
       (c) to (middle.center);
       \draw [decoration={markings, mark=at position 0.5 with {\arrow{>}}}, postaction={decorate}]
       (middle.center) to (n);

       
       \end{scope}

      \begin{scope}[scale=1, >=Stealth, rotate around={-90:(-2.5,0)}]
       \node (a) at (-4,0){};
       \node (b) at (-4,1){};
       \node (c) at (-3,1) {};
       \node (d) at (-2,1){};
       \node (g) at (-2,1) {};
       \node (h) at (-1,1){};
       \node (i) at (-1,0) {};
       \node (j) at (-1,-1){};
       \node (k) at (-2,-1) {};
       \node (n) at (-2,-1){};
       \node (o) at (-3,-1) {};
       \node (p) at (-4,-1){};

       \node (middle) at (-2.5,0){};


       \draw [magenta, thick, out=45, in=135, relative, decoration={markings, mark=at position 0.5 with {\arrow{}}}, postaction={decorate}] 
       (b.center) to (c);
       \draw [cyan,thick, out=45, in=135, relative, decoration={markings, mark=at position 0.5 with {\arrow{}}}, postaction={decorate}] 
       (c.center) to (d);
       \draw [Periwinkle, thick, out=60, in=135, relative, decoration={markings, mark=at position 0.5 with {\arrow{}}}, postaction={decorate}] 
       (g.center) to (h.center);

       \draw [Periwinkle, thick, out=45, in=135, relative, decoration={markings, mark=at position 0.5 with {\arrow{}}}, postaction={decorate}] 
       (h.center) to (i);
       \draw [Emerald, thick, out=45, in=135, relative, decoration={markings, mark=at position 0.5 with {\arrow{}}}, postaction={decorate}] 
       (i.center) to (j.center);

       \draw [Emerald, thick, out=45, in=135, relative, decoration={markings, mark=at position 0.5 with {\arrow{}}}, postaction={decorate}] 
       (j.center) to (k);
       \draw [SpringGreen, thick, out=45, in=135, relative, decoration={markings, mark=at position 0.5 with {\arrow{}}}, postaction={decorate}] 
       (n.center) to (o);
       \draw [RoyalBlue, thick, out=60, in=135, relative, decoration={markings, mark=at position 0.5 with {\arrow{}}}, postaction={decorate}] 
       (o.center) to (p.center);

       \draw [RoyalBlue, thick, out=45, in=135, relative, decoration={markings, mark=at position 0.5 with {\arrow{}}}, postaction={decorate}] 
       (p.center) to (a);
       \draw [magenta, thick, out=45, in=135, relative, decoration={markings, mark=at position 0.5 with {\arrow{}}}, postaction={decorate}] 
       (a.center) to (b.center);


        \draw [RoyalBlue, thick, decoration={markings, mark=at position 0.5 with {\arrow{}}}, postaction={decorate}]
       (o) to (a.center);
       \draw [magenta, thick, decoration={markings, mark=at position 0.5 with {\arrow{}}}, postaction={decorate}]
       (a) to (c.center);


       \draw [Periwinkle, thick, decoration={markings, mark=at position 0.5 with {\arrow{}}}, postaction={decorate}]
       (g) to (i.center);
       \draw [Emerald, thick, decoration={markings, mark=at position 0.5 with {\arrow{}}}, postaction={decorate}]
       (i) to (k.center);
       

       \draw [cyan, thick, decoration={markings, mark=at position 0.5 with {\arrow{}}}, postaction={decorate}]
       (d.center) to (middle);
       \draw [SpringGreen, thick]
       (middle) to (o.center);
       \draw [cyan, thick]
       (c) to (middle.center);
       \draw [SpringGreen, thick, decoration={markings, mark=at position 0.5 with {\arrow{}}}, postaction={decorate}]
       (middle.center) to (n);

       
       \end{scope}
 \begin{scope}[xshift=0cm, thick, scale = 0.7, rotate around={-90:(0,0)}]

\node (a) at (-2,2){};
\node (b) at (0,2){};
\node (c) at (2,2){};
\node (d) at (2,0){};
\node (e) at (0,0){};
\node (f) at (-2,0){};

\draw[]
(a) -- (c) -- (d) -- (f) -- (a)
(b) -- (e);

\filldraw[color=magenta, fill=magenta!5,   thick]
(a) circle (1/3);
\filldraw[color=Cerulean, fill=Cerulean!5,   thick]
(b) circle (1/3);
\filldraw[color=Periwinkle!90, fill=Periwinkle!5,   thick]
(c) circle (1/3);
\filldraw[color=Emerald, fill=Emerald!5,   thick]
(d) circle (1/3);
\filldraw[color=SpringGreen!90, fill=SpringGreen!5,   thick]
(e) circle (1/3);
\filldraw[color=RoyalBlue, fill=RoyalBlue!5,   thick]
(f) circle (1/3);

\end{scope}

    \begin{scope}[xshift=2.5cm, thick, scale = 0.7, rotate around={-90:(0,0)}]

\node (a) at (-2,2){};
\node (b) at (0,2){};
\node (c) at (2,2){};
\node (d) at (2,0){};
\node (e) at (0,0){};
\node (f) at (-2,0){};

\draw[]
(a) -- (c) -- (d) -- (f) -- (a)
(b) -- (e);

\filldraw[color=magenta, fill=magenta!5,   thick]
(a) circle (1/3);
\filldraw[color=Cerulean, fill=Cerulean!5,   thick]
(b) circle (1/3);
\filldraw[color=Periwinkle!90, fill=Periwinkle!5,   thick]
(c) circle (1/3);
\filldraw[color=Emerald, fill=Emerald!5,   thick]
(d) circle (1/3);
\filldraw[color=SpringGreen!90, fill=SpringGreen!5,   thick]
(e) circle (1/3);
\filldraw[color=RoyalBlue, fill=RoyalBlue!5,   thick]
(f) circle (1/3);

\end{scope}

\end{scope}

  \end{tikzpicture}
    \caption{An Oddly Balanced diagram  of type $2$, its $A$-circles, $A$-state graph, and its reduced $A$-state graph}
    \label{fig:Oddly Balanced type 2 example}
\end{figure}
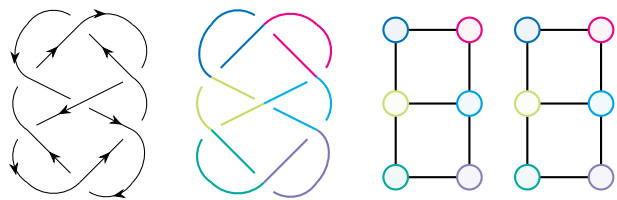


%% file: Figures/Burdened_type_0_example.tex
  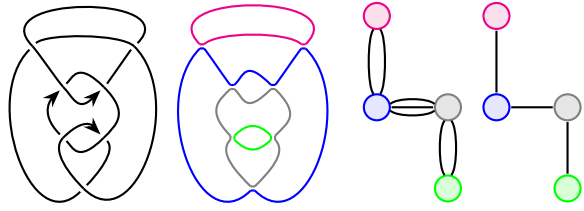
\begin{figure}[h]
    \centering
    \begin{tikzpicture}[every path/.style={thick}, every
node/.style={transform shape, knot crossing, inner sep=2pt},scale=0.6]
    \begin{scope}[scale=0.9]
    \node (tlo) at (-6.25,2.25) {};
    \node (tro) at (-3.75,2.25) {};
    \node (so) at (-5,1.25) {};
    \node (slo) at (-5.5,1.25) {};
    \node (sro) at (-4.5,1.25) {};
    \node (mlo) at (-5.5,0) {};
    \node (mro) at (-4.5, 0) {};
    \node (bo) at (-5,-1.25) {};
    \draw (bo.center) .. controls (bo.4 north west) and (mlo.4 south west) ..
    (mlo.center);
    \draw  (bo) .. controls (bo.4 north east) and (mro.4 south east).. (mro);
    \draw  [-{Stealth}] (mlo) .. controls (mlo.8 north west) and (slo.3 south west) ..(slo);
    \draw [-{Stealth}] (mlo.center) .. controls (mlo.8 north east) and (mro.2 north
    west) .. (mro);
    \draw (mro.center) .. controls (mro.4 north east) and (sro.8 south east) ..
    (sro.center);
    \draw (mro.center) .. controls (mro.8 south west) and (mlo.3 south east) ..
    (mlo);
    \draw (sro) .. controls (sro.north east) and (tro.south west) .. (tro);
    \draw (slo.center) .. controls (slo.north west) and (tlo.south
    east) .. (tlo.center);
    \draw [{Stealth}-] (sro) .. controls (sro.8 south west) and (slo.3 south east) ..
    (slo.center);
    \draw (sro.center) .. controls (sro.8 north west) and (slo.3 north east) ..
    (slo);
    \draw (tlo.center) .. controls (tlo.16 north west) and (tro.16 north
    east) .. (tro);
    \draw (bo.center) .. controls (bo.16 south east) and (tro.16 south east) ..
    (tro.center);
    \draw (bo) .. controls (bo.16 south west) and (tlo.16 south
    west) .. (tlo);
    \draw (tlo) .. controls (tlo.4 north east) and (tro.4 north
    west) .. (tro.center);
    \end{scope}
    
    \begin{scope}[xshift=-0.75cm, scale=0.9]
    \node (tl) at (-1.25,2.25) {};
    \node (tltl) at (-1.3,2.3) {};
    \node (tlbl) at (-1.3,2.2) {};
    \node (tltr) at (-1.2,2.3) {};
    \node (tlbr) at (-1.2,2.2) {};
    \node (tr) at (1.25,2.25) {};
    \node (trtl) at (1.2,2.3) {};
    \node (trbl) at (1.2,2.2) {};
    \node (trtr) at (1.3,2.3) {};
    \node (trbr) at (1.3,2.2) {};
    \node (stl) at (-0.05,1.3) {};
    \node (str) at (0.05,1.3) {};
    \node (sbl) at (-0.05,1.2) {};
    \node (sbr) at (0.05,1.2) {};
    \node (ml) at (-0.5,0) {};
    \node (mltl) at (-0.55,0.05) {};
    \node (mlbl) at (-0.55,-0.05) {};
    \node (mltr) at (-0.45,0.05) {};
    \node (mlbr) at (-0.45,-0.05) {};
    \node (mr) at (0.5, 0) {};
    \node (mrtl) at (0.45, 0.05) {};
    \node (mrbl) at (0.45, -0.05) {};
    \node (mrtr) at (0.55, 0.05) {};
    \node (mrbr) at (0.55, -0.05) {};
    \node (b) at (0,-1.25) {};
    \node (btl) at (-0.05, -1.2) {};
    \node (bbl) at (-0.05, -1.3) {};
    \node (btr) at (0.05, -1.2) {};
    \node (bbr) at (0.05, -1.3) {};
    
    \node (sl) at (-0.5,1.25) {};
    \node (sr) at (0.5,1.25) {};
    
    \node (sltl) at (-0.55,1.3) {};
    \node (sltr) at (-0.45,1.3) {};
    \node (slbl) at (-0.55,1.2) {};
    \node (slbr) at (-0.45,1.2) {};
    \node (srtl) at (0.45,1.3) {};
    \node (srtr) at (0.55,1.3) {};
    \node (srbl) at (0.45,1.2) {};
    \node (srbr) at (0.55,1.2) {};

    \draw  [gray] (mltl.center) .. controls (mltl.8 north west) and (slbl.3 south west) ..(slbl.center);
    \draw [gray] (mrtr.center) .. controls (mrtr.4 north east) and (srbr.8 south east) ..
    (srbr.center);
    \draw [blue] (srtr.center) .. controls (srtr.north east) and (trbl.south west) .. (trbl.center);
    \draw [blue] (sltl.center) .. controls (sltl.north west) and (tlbr.south
    east) .. (tlbr.center);
    \draw [gray] (srbl.center) .. controls (srbl.8 south west) and (slbr.3 south east) ..
    (slbr.center);
    \draw [blue] (srtl.center) .. controls (srtl.8 north west) and (sltr.3 north east) ..
    (sltr.center);
    
    
    
    
    \draw [gray] (btl.center) .. controls (btl.4 north west) and (mlbl.4 south west) ..
    (mlbl.center);
    \draw [gray] (btr.center) .. controls (btr.4 north east) and (mrbr.4 south east)
    .. (mrbr.center);
    \draw [green] (mltr.center) .. controls (mltr.4 north east) and (mrtl.4 north
    west) .. (mrtl.center);
    \draw [green] (mrbl.center) .. controls (mrbl.4 south west) and (mlbr.4 south east) ..
    (mlbr.center);
    \draw [magenta] (tltl.center) .. controls (tltl.16 north west) and (trtr.16 north
    east) .. (trtr.center);
    \draw [blue] (bbr.center) .. controls (bbr.16 south east) and (trbr.16 south east) ..
    (trbr.center);
    \draw [blue] (bbl.center) .. controls (bbl.16 south west) and (tlbl.16 south
    west) .. (tlbl.center);
    \draw [magenta] (tltr.center) .. controls (tltr.4 north east) and (trtl.4 north
    west) .. (trtl.center);
    
    \draw [gray] (mltl.center) -- (mlbl.center) {};
    \draw [green] (mltr.center) -- (mlbr.center) {};
    \draw [green] (mrtl.center) -- (mrbl.center) {};
    \draw [gray] (mrtr.center) -- (mrbr.center) {};
    \draw [magenta] (tltl.center) -- (tltr.center) {};
    \draw [blue] (tlbl.center) -- (tlbr.center) {};
    \draw [magenta] (trtl.center) -- (trtr.center) {};
    \draw [blue] (trbl.center) -- (trbr.center) {};
    \draw [gray] (btl.center) -- (btr.center) {};
    \draw [blue] (bbl.center) -- (bbr.center) {};
    
    \draw [blue] (sltl.center) -- (sltr.center) {};
    \draw [gray] (slbl.center) -- (slbr.center) {};
    \draw [blue] (srtl.center) -- (srtr.center) {};
    \draw [gray] (srbl.center) -- (srbr.center) {};
    
    \end{scope}
    
    
    \begin{scope}[xshift=2cm, scale =0.9]
    \node (magenta) at (0, 2.75) {};
    \node (m) at (0, 3) {};
    \node (blue) at (0, 1) {};
    \node (c) at (0, .75) {};
    \node (cr) at (0.25, 0.75) {};
    \node (orange) at (1.5, 0.75) {};
    \node (o) at (1.75, 0.75) {};
    \node (ou) at (1.75, 0.5) {};
    \node (ol) at (1.5, 0.75) {};
    \node (green) at (1.75, -1) {};
    \node (g) at (1.75, -1.25) {};

  \filldraw [color=magenta, fill=magenta!15] (m) circle[radius=0.32cm];
    \filldraw [color=blue, fill=blue!10] (c) circle[radius=0.32cm];
    \filldraw [color=gray, fill=gray!20] (o) circle[radius=0.32cm];
    \filldraw [color=green, fill=green!15] (g) circle[radius=0.32cm];
    \draw (cr) -- (ol);
    \draw (blue) .. controls (blue.3 north east) and (magenta.3 south east) .. (magenta);
    \draw (blue) .. controls (blue.3 north west) and (magenta.3 south west) .. (magenta);
    \draw (orange) .. controls (orange.3 north west) and (cr.3 north east) .. (cr);
    \draw (orange) .. controls (orange.3 south west) and (cr.3 south east) .. (cr);

    \draw (green) .. controls (green.3 north west) and (ou.3 south west) .. (ou);
    \draw (green) .. controls (green.3 north east) and (ou.3 south east) .. (ou);
    \end{scope}
    
    \begin{scope}[xshift=4.65cm, scale=0.9]
    \node (magenta) at (0, 2.75) {};
    \node (m) at (0, 3) {};
    \node (cyan) at (0, 1) {};
    \node (c) at (0, .75) {};
    \node (cr) at (0.25, 0.75) {};
    \node (orange) at (1.5, 0.75) {};
    \node (o) at (1.75, 0.75) {};
    \node (ou) at (1.75, 0.5) {};
    \node (green) at (1.75, -1) {};
    \node (g) at (1.75, -1.25) {};
    
     \filldraw [color=magenta, fill=magenta!15] (m) circle[radius=0.32cm];
    \filldraw [color=blue, fill=blue!10] (c) circle[radius=0.32cm];
    \filldraw [color=gray, fill=gray!20] (o) circle[radius=0.32cm];
    \filldraw [color=green, fill=green!15] (g) circle[radius=0.32cm];
    \draw (cyan) -- (magenta);
    \draw (orange) -- (cr);
    \draw (green) -- (ou);
    \end{scope}
\end{tikzpicture} 
\vspace{-10pt}
\caption{(Left to right:) A Burdened diagram of type $0$, its $A$-circles, its $A$-state graph, and its reduced $A$-state graph }
\label{Burdened type 0 example}
\end{figure}

%% file: Figures/Burdened_type_1_example.tex
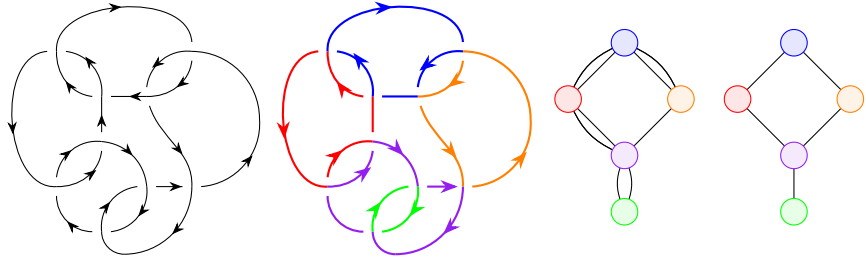
\begin{figure}[h]
    \centering
  \begin{tikzpicture}
      \begin{scope}[scale=0.53, >=Stealth]

      \begin{scope}[scale=0.8, xshift=-6cm, rotate=45]
       \node (a) at (-3,0){};
       \node (b) at (-1,0){};
       \node (c) at (0,3) {};
       \node (d) at (0,1){};
       \node (e) at (1,0) {};
       \node (f) at (3,0){};
       \node (g) at (3,-3) {};
       \node (h) at (0, -3){};
       \node (i) at (-3,-3){};
       \node (j) at (-3, -2){};
       \node (k) at (-1, -2){};
       

       \draw [out=90, in=90, relative, decoration={markings, mark=at position 0.5 with {\arrow{<}}}, postaction={decorate}] 
       (a.center) to (c);
       \draw [out=90, in=90, relative, decoration={markings, mark=at position 0.5 with {\arrow{>}}}, postaction={decorate}] 
       (c.center) to (f);
       \draw [out=45, in=135, relative, decoration={markings, mark=at position 0.5 with {\arrow{}}}, postaction={decorate}] 
       (f.center) to (g.center);
       \draw [out=45, in=135, relative, decoration={markings, mark=at position 0.5 with {\arrow{<}}}, postaction={decorate}] 
       (g.center) to (h);
       \draw [out=45, in=135, relative, decoration={markings, mark=at position 0.5 with {\arrow{>}}}, postaction={decorate}] 
       (h.center) to (i.center);
       \draw [out=45, in=135, relative, decoration={markings, mark=at position 0.5 with {\arrow{}}}, postaction={decorate}] 
       (i.center) to (j.center);
       \draw [out=45, in=135, relative, decoration={markings, mark=at position 0.5 with {\arrow{>}}}, postaction={decorate}] 
       (j) to (a);

       \draw [ decoration={markings, mark=at position 0.5 with {\arrow{>}}}, postaction={decorate}] 
       (b) to (d.center);
       \draw [ decoration={markings, mark=at position 0.5 with {\arrow{>}}}, postaction={decorate}] 
       (e.center) to (d);
       \draw [out=-30, in=170, relative, decoration={markings, mark=at position 0.5 with {\arrow{<}}}, postaction={decorate}] 
       (h.center) to (e);
       \draw [decoration={markings, mark=at position 0.5 with {\arrow{<}}}, postaction={decorate}] 
       (h) to (k);
       \draw [out=40, in=140, relative, decoration={markings, mark=at position 0.5 with {\arrow{>}}}, postaction={decorate}] 
       (b.center) to (k.center);


       \draw [out=45, in=135, relative, decoration={markings, mark=at position 0.5 with {\arrow{>}}}, postaction={decorate}] 
       (a) to (b.center);
       \draw [out=45, in=140, relative, decoration={markings, mark=at position 0.5 with {\arrow{<}}}, postaction={decorate}] 
       (b) to (a.center);
       \draw [out=45, in=140, relative, decoration={markings, mark=at position 0.5 with {\arrow{>}}}, postaction={decorate}] 
       (d) to (c.center);
       \draw [out=45, in=135, relative, decoration={markings, mark=at position 0.5 with {\arrow{<}}}, postaction={decorate}] 
       (c) to (d.center);
       \draw [out=45, in=140, relative, decoration={markings, mark=at position 0.5 with {\arrow{>}}}, postaction={decorate}] 
       (f) to (e.center);
       \draw [out=45, in=135, relative, decoration={markings, mark=at position 0.5 with {\arrow{<}}}, postaction={decorate}] 
       (e) to (f.center);
       \draw [out=45, in=135, relative, decoration={markings, mark=at position 0.5 with {\arrow{>}}}, postaction={decorate}] 
       (j.center) to (k);
       \draw [out=45, in=135, relative, decoration={markings, mark=at position 0.5 with {\arrow{>}}}, postaction={decorate}] 
       (k.center) to (j);
       \end{scope}

      \begin{scope}[scale=0.8, xshift=2.5cm, rotate=45]
       \node (a) at (-3,0){};
       \node (b) at (-1,0){};
       \node (c) at (0,3) {};
       \node (d) at (0,1){};
       \node (e) at (1,0) {};
       \node (f) at (3,0){};
       \node (g) at (3,-3) {};
       \node (h) at (0, -3){};
       \node (i) at (-3,-3){};
       \node (j) at (-3, -2){};
       \node (k) at (-1, -2){};
       

       \draw [red, thick, out=90, in=90, relative, decoration={markings, mark=at position 0.5 with {\arrow{<}}}, postaction={decorate}] 
       (a.center) to (c);
       \draw [blue, thick, out=90, in=90, relative, decoration={markings, mark=at position 0.5 with {\arrow{>}}}, postaction={decorate}] 
       (c.center) to (f);
       \draw [orange, thick, out=45, in=135, relative, decoration={markings, mark=at position 0.5 with {\arrow{}}}, postaction={decorate}] 
       (f.center) to (g.center);
       \draw [orange, thick, out=45, in=135, relative, decoration={markings, mark=at position 0.5 with {\arrow{<}}}, postaction={decorate}] 
       (g.center) to (h);
       \draw [purp, thick, out=45, in=135, relative, decoration={markings, mark=at position 0.5 with {\arrow{>}}}, postaction={decorate}] 
       (h.center) to (i.center);
       \draw [purp, thick, out=45, in=135, relative, decoration={markings, mark=at position 0.5 with {\arrow{}}}, postaction={decorate}] 
       (i.center) to (j.center);
       \draw [purp, thick, out=45, in=135, relative, decoration={markings, mark=at position 0.5 with {\arrow{}}}, postaction={decorate}] 
       (j) to (a);

       \draw [red, thick, decoration={markings, mark=at position 0.5 with {\arrow{}}}, postaction={decorate}] 
       (b) to (d.center);
       \draw [blue, thick, decoration={markings, mark=at position 0.5 with {\arrow{}}}, postaction={decorate}] 
       (e.center) to (d);
       \draw [orange, thick, out=-30, in=170, relative, decoration={markings, mark=at position 0.5 with {\arrow{<}}}, postaction={decorate}] 
       (h.center) to (e);
       \draw [purp, thick, decoration={markings, mark=at position 0.5 with {\arrow{<}}}, postaction={decorate}] 
       (h) to (k);
       \draw [purp, thick, out=40, in=140, relative, decoration={markings, mark=at position 0.5 with {\arrow{>}}}, postaction={decorate}] 
       (b.center) to (k.center);


       \draw [red, thick, out=45, in=135, relative, decoration={markings, mark=at position 0.5 with {\arrow{>}}}, postaction={decorate}] 
       (a) to (b.center);
       \draw [purp, thick, out=45, in=140, relative, decoration={markings, mark=at position 0.5 with {\arrow{<}}}, postaction={decorate}] 
       (b) to (a.center);
       \draw [red, thick, out=45, in=140, relative, decoration={markings, mark=at position 0.5 with {\arrow{>}}}, postaction={decorate}] 
       (d) to (c.center);
       \draw [blue, thick, out=45, in=135, relative, decoration={markings, mark=at position 0.5 with {\arrow{<}}}, postaction={decorate}] 
       (c) to (d.center);
       \draw [orange, thick, out=45, in=140, relative, decoration={markings, mark=at position 0.5 with {\arrow{>}}}, postaction={decorate}] 
       (f) to (e.center);
       \draw [blue, thick, out=45, in=135, relative, decoration={markings, mark=at position 0.5 with {\arrow{<}}}, postaction={decorate}] 
       (e) to (f.center);
       \draw [green, thick, out=45, in=135, relative, decoration={markings, mark=at position 0.5 with {\arrow{>}}}, postaction={decorate}] 
       (j.center) to (k);
       \draw [green, thick, out=45, in=135, relative, decoration={markings, mark=at position 0.5 with {\arrow{>}}}, postaction={decorate}] 
       (k.center) to (j);
       \end{scope}

  \begin{scope}[scale = 1, xshift=7.75cm, yshift=0.5cm, rotate = 45]
  
       \node (a) at (-1,1){};
       \node (b) at (1,1){};
       \node (c) at (1,-1) {};
       \node (d) at (-1,-1){};
       \node (e) at (-2, -2) {};
        \draw[]
        (a) -- (b)
        (b) -- (c) 
        (d) -- (a)
        (c) -- (d);
        \draw[out=20, in=160, relative]
        (a) to (b)
        (b) to (c)
        (d) to (a)
        (d) to (e)
        (e) to (d);
        \draw[out=20, in=160, relative]
        (a) to (b)
        (b) to (c)
        (d) to (a);

        \filldraw[color=red, fill=red!10] 
        (a) circle (1/3);
     \filldraw[color=blue, fill=blue!10] 
        (b) circle (1/3);
     \filldraw[color=orange, fill=orange!10] 
        (c) circle (1/3);
     \filldraw[color=purp, fill=purp!10] 
        (d) circle (1/3);
     \filldraw[color=green, fill=green!10] 
        (e) circle (1/3);
        
       \end{scope}

  \begin{scope}[scale = 1, xshift=12cm, yshift=0.5cm, rotate = 45]
  
       \node (a) at (-1,1){};
       \node (b) at (1,1){};
       \node (c) at (1,-1) {};
       \node (d) at (-1,-1){};
       \node (e) at (-2, -2) {};
        \draw[]
        (a) -- (b)
        (b) -- (c) 
        (d) -- (a)
        (c) -- (d)
        (d) -- (e);

        \filldraw[color=red, fill=red!10] 
        (a) circle (1/3);
     \filldraw[color=blue, fill=blue!10] 
        (b) circle (1/3);
     \filldraw[color=orange, fill=orange!10] 
        (c) circle (1/3);
     \filldraw[color=purp, fill=purp!10] 
        (d) circle (1/3);
     \filldraw[color=green, fill=green!10] 
        (e) circle (1/3);
        
       \end{scope}

       
       \end{scope}
       
  \end{tikzpicture}
  \vspace{-5pt}
     \caption{A Burdened diagram of type $1$, its $A$-circles, $A$-state graph, and reduced $A$-state graph}
    \label{fig:Burdened type 1 example}
\end{figure}

%% file: Figures/k-sequence_is_not_invariant.tex
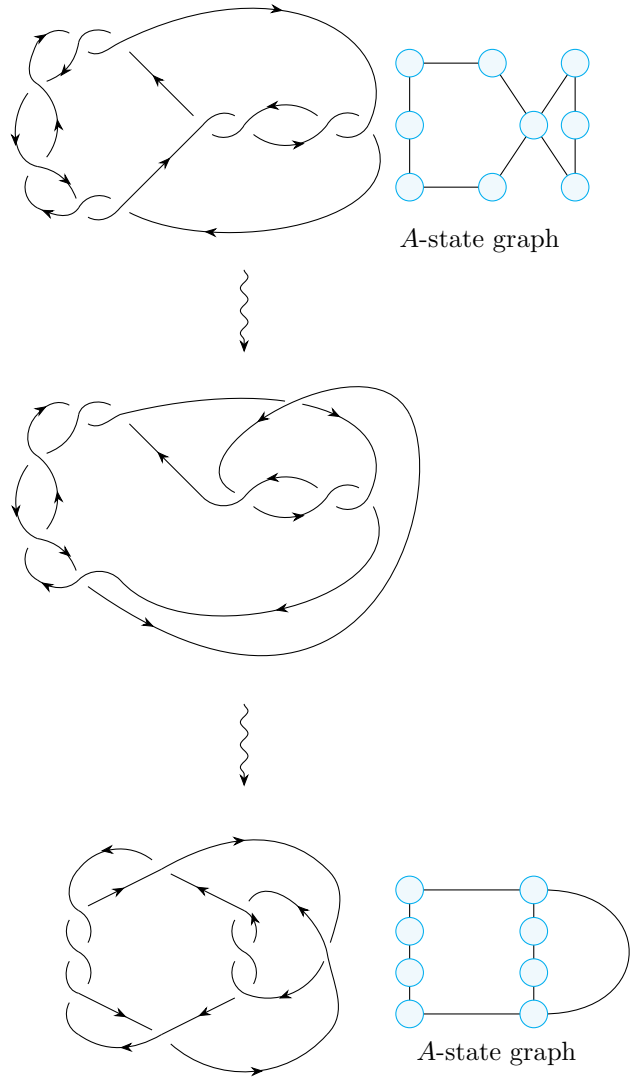
\begin{figure}[h]
    \centering
  \begin{tikzpicture}
      \begin{scope}[scale=0.55, >=Stealth]

      \begin{scope}
       \node (a) at (-4,2){};
       \node (b) at (-4,3){};
       \node (c) at (-3,3) {};
       \node (d) at (-2,3){};
       \node (e) at (-1,3) {};
       \node (f) at (0,3){};
       \node (g) at (1,3) {};
       \node (h) at (2,3){};
       \node (i) at (2,2) {};
       \node (j) at (2,-1){};
       \node (k) at (1,-1) {};
       \node (l) at (0,-1){};
       \node (m) at (-1,-1) {};
       \node (n) at (-2,-1){};
       \node (o) at (-3,-1) {};
       \node (p) at (-4,-1){};

       \node (p1) at (-4, 0){};
       \node (p2) at (-4, 1){};
        \node (i1) at (2, 1){};
        \node (i2) at (2,0){};

        \node (mt) at (0, 2){};
        \node (mm) at (0, 1) {};
        \node (mb) at (0, 0) {};

        \node (km) at (1,1) {};
        \node (right) at (3, 1) {};
        \node (outright) at (4,1) {};


       \draw [out=35, in=135, relative, decoration={markings, mark=at position 0.35 with {\arrow{>}}}, postaction={decorate}] 
       (b.center) to (c);
       \draw [out=60, in=135, relative, decoration={markings, mark=at position 0.5 with {\arrow{}}}, postaction={decorate}] 
       (c.center) to (d);
       \draw [ decoration={markings, mark=at position 0.5 with {\arrow{>}}}, postaction={decorate}] 
       (d.center) ..controls (2,5) and (5,3).. (outright.center);

       \draw [out=100, in=135, relative, decoration={markings, mark=at position 0.75 with {\arrow{>}}}, postaction={decorate}] 
       (outright) to (n);
       \draw [out=35, in=135, relative, decoration={markings, mark=at position 0.5 with {\arrow{}}}, postaction={decorate}] 
       (n.center) to (o);
       \draw [out=55, in=155, relative, decoration={markings, mark=at position 0.9 with {\arrow{>}}}, postaction={decorate}] 
       (o.center) to (p.center);

       \draw [out=45, in=135, relative, decoration={markings, mark=at position 0.5 with {\arrow{}}}, postaction={decorate}] 
       (p.center) to (p1);
       \draw [out=70, in=135, relative, decoration={markings, mark=at position 0.5 with {\arrow{<}}}, postaction={decorate}] 
       (p1.center) to (a);
       \draw [out=60, in=155, relative, decoration={markings, mark=at position 0.5 with {\arrow{}}}, postaction={decorate}] 
       (a.center) to (b.center);


        \draw [out = -30, in =-150, relative, decoration={markings, mark=at position 0.5 with {\arrow{<}}}, postaction={decorate}]
       (o) to (p1.center);
       \draw [out = -45, in =-135, relative, decoration={markings, mark=at position 0.5 with {\arrow{>}}}, postaction={decorate}]
       (p1) to (a.center);
       \draw [out = -30, in =-150, relative, decoration={markings, mark=at position 0.5 with {\arrow{<}}}, postaction={decorate}]
       (a) to (c.center);

       \draw [out = -45, in =-170, relative, decoration={markings, mark=at position 0.5 with {\arrow{}}}, postaction={decorate}]
       (c) to (d.center);
       \draw [out =0, in =180, relative, decoration={markings, mark=at position 0.5 with {\arrow{<}}}, postaction={decorate}]
       (d) to (mm);
        \draw [out = 35, in =135, relative, decoration={markings, mark=at position 0.7 with {\arrow{}}}, postaction={decorate}]
       (mm.center) to (km);
        \draw [out = 45, in =135, relative, decoration={markings, mark=at position 0.5 with {\arrow{<}}}, postaction={decorate}]
       (km.center) to (right);
       \draw [out = 45, in =135, relative, decoration={markings, mark=at position 0.5 with {\arrow{}}}, postaction={decorate}]
       (right.center) to (outright);

        \draw [out = 45, in =135, relative, decoration={markings, mark=at position 0.5 with {\arrow{}}}, postaction={decorate}]
       (outright.center) to (right);
        \draw [out = 45, in =135, relative, decoration={markings, mark=at position 0.5 with {\arrow{<}}}, postaction={decorate}]
       (right.center) to (km);
        \draw [out = 45, in =135, relative, decoration={markings, mark=at position 0.5 with {\arrow{}}}, postaction={decorate}]
       (km.center) to (mm);
       \draw [out = 0, in =180, relative, decoration={markings, mark=at position 0.5 with {\arrow{<}}}, postaction={decorate}]
       (mm.center) to (n.center);

       \draw [out = -45, in =-135, relative, decoration={markings, mark=at position 0.5 with {\arrow{}}}, postaction={decorate}]
       (n) to (o.center);



  \begin{scope}[scale = 1, xshift=8cm, yshift= -0.5cm]

        \node (a) at (-3,1.5){};
       \node (b) at (-3,3){};
       \node (d) at (-2,3){};
       \node (e) at (-1,3) {};
       \node (f) at (0,3){};
       \node (h) at (1,3){};
       \node (i) at (1,1.5) {};
       \node (j) at (1,0){};
       \node (l) at (0,0){};
       \node (m) at (-1,0) {};
       \node (n) at (-2,0){};
       \node (p) at (-3,0){};

       \node (p2) at (-3, 1.5){};
        \node (i1) at (1, 1.5){};

        \node (mt) at (0, 1.5){};
        \node (mm) at (0, 1) {};
        \node (mb) at (0, 0) {};

        \draw[]
        (b) -- (e)
        (e) -- (mt) 
        (mt) -- (h)
        (h) -- (j)
        node[right, text centered, text width = 1cm, midway, xshift=0.25cm]{}  
        node[below left, text centered, text width = 3cm, midway, xshift = 0.35cm, yshift=-1.25cm]{$A$-state graph}  
        (j) -- (mt)
        (mt) -- (m)
        (m) -- (p) 
        (p) -- (b)
        node[right, text centered, text width = 1cm, midway, xshift=0.25cm]{}  ;

        \filldraw[color=cyan, fill=cyan!5] 
        (b) circle (1/3)
        (e) circle (1/3)
        (h) circle (1/3)
        (i1) circle (1/3)
        (p2) circle (1/3)
        (j) circle (1/3)
        (m) circle (1/3)
        (p) circle (1/3)

         (mt) circle (1/3);
       \end{scope}
          \end{scope}

    \begin{scope}[xshift=1cm, yshift=-4.5cm]

       
       \draw [->,snake=snake, segment amplitude=.5mm, segment length=3mm, line after snake=1mm] 
       (0,2) -- (0,0)
        node [right,text width=3cm,text centered,midway]{};
    \end{scope}

    \begin{scope}[yshift=-9cm]

               \node (a) at (-4,2){};
       \node (b) at (-4,3){};
       \node (c) at (-3,3) {};
       \node (d) at (-2,3){};
       \node (e) at (-1,3) {};
       \node (f) at (0,3){};
       \node (g) at (1,3) {};
       \node (h) at (2,3){};
       \node (i) at (2,2) {};
       \node (j) at (2,-1){};
       \node (k) at (1,-1) {};
       \node (l) at (0,-1){};
       \node (m) at (-1,-1) {};
       \node (n) at (-2,-1){};
       \node (o) at (-3,-1) {};
       \node (p) at (-4,-1){};

       \node (p1) at (-4, 0){};
       \node (p2) at (-4, 1){};
        \node (i1) at (2, 1){};
        \node (i2) at (2,0){};

        \node (mt) at (0, 2){};
        \node (mm) at (0, 1) {};
        \node (mb) at (0, 0) {};

        \node (km) at (1,1) {};
        \node (right) at (3, 1) {};
        \node (outright) at (4,1) {};


       \draw [out=35, in=135, relative, decoration={markings, mark=at position 0.35 with {\arrow{>}}}, postaction={decorate}] 
       (b.center) to (c);
       \draw [out=60, in=135, relative, decoration={markings, mark=at position 0.5 with {\arrow{}}}, postaction={decorate}] 
       (c.center) to (d);
       \draw [ decoration={markings, mark=at position 0.7 with {\arrow{>}}}, postaction={decorate}] 
       (d.center) ..controls (2,4) and (5,3).. (outright.center);

       \draw [out=100, in=115, relative, decoration={markings, mark=at position 0.5 with {\arrow{>}}}, postaction={decorate}] 
       (outright) to (n.center);
       \draw [out=45, in=145, relative, decoration={markings, mark=at position 0.9 with {\arrow{>}}}, postaction={decorate}] 
       (o.center) to (p.center);

       \draw [out=45, in=135, relative, decoration={markings, mark=at position 0.5 with {\arrow{}}}, postaction={decorate}] 
       (p.center) to (p1);
       \draw [out=60, in=135, relative, decoration={markings, mark=at position 0.5 with {\arrow{<}}}, postaction={decorate}] 
       (p1.center) to (a);
       \draw [out=60, in=135, relative, decoration={markings, mark=at position 0.9 with {\arrow{}}}, postaction={decorate}] 
       (a.center) to (b.center);

    \filldraw[color=white]
    (2.2,3.3) circle (1/5);

    \draw[decoration={markings, mark=at position 0.65 with {\arrow{>}}}, postaction={decorate}]
       (5,3) .. controls (4,5) and (-1,2) .. (km);
    \draw[decoration={markings, mark=at position 0.15 with {\arrow{>}}}, postaction={decorate}]
       (o) .. controls (4,-6) and (6, 1) .. (5,3);


        \draw [out = -30, in =-150, relative, decoration={markings, mark=at position 0.5 with {\arrow{<}}}, postaction={decorate}]
       (o) to (p1.center);
       \draw [out = -45, in =-135, relative, decoration={markings, mark=at position 0.5 with {\arrow{>}}}, postaction={decorate}]
       (p1) to (a.center);
       \draw [out = -30, in =-150, relative, decoration={markings, mark=at position 0.5 with {\arrow{}}}, postaction={decorate}]
       (a) to (c.center);

       \draw [out = -45, in =-170, relative, decoration={markings, mark=at position 0.5 with {\arrow{}}}, postaction={decorate}]
       (c) to (d.center);
       \draw [out =0, in =180, relative, decoration={markings, mark=at position 0.5 with {\arrow{<}}}, postaction={decorate}]
       (d) to (mm.center);
        \draw [out = 45, in =135, relative, decoration={markings, mark=at position 0.5 with {\arrow{<}}}, postaction={decorate}]
       (km.center) to (right);
       \draw [out = 45, in =135, relative, decoration={markings, mark=at position 0.5 with {\arrow{}}}, postaction={decorate}]
       (right.center) to (outright);

        \draw [out = 45, in =135, relative, decoration={markings, mark=at position 0.5 with {\arrow{}}}, postaction={decorate}]
       (outright.center) to (right);
        \draw [out = 45, in =135, relative, decoration={markings, mark=at position 0.5 with {\arrow{<}}}, postaction={decorate}]
       (right.center) to (km);
        \draw [out = 45, in =135, relative, decoration={markings, mark=at position 0.5 with {\arrow{}}}, postaction={decorate}]
       (km.center) to (mm.center);

       \draw [out = -45, in =-135, relative, decoration={markings, mark=at position 0.5 with {\arrow{}}}, postaction={decorate}]
       (n.center) to (o.center);


    \end{scope}

    \begin{scope}[xshift=1cm, yshift=-15cm]

       
       \draw [->,snake=snake, segment amplitude=.5mm, segment length=3mm, line after snake=1mm] 
       (0,2) -- (0,0)
        node [right,text width=3cm,text centered,midway]{};
    \end{scope}

    \begin{scope}[xshift=1cm, yshift=-20cm]    
    
       \node (a) at (-4,2){};
       \node (b) at (-4,3){};
       \node (c) at (-3,3) {};
       \node (d) at (-2,3){};
       \node (e) at (-1,3) {};
       \node (f) at (0,3){};
       \node (g) at (1,3) {};
       \node (h) at (2,3){};
       \node (i) at (2,2) {};
       \node (j) at (2,-1){};
       \node (k) at (1,-1) {};
       \node (l) at (0,-1){};
       \node (m) at (-1,-1) {};
       \node (n) at (-2,-1){};
       \node (o) at (-3,-1) {};
       \node (p) at (-4,-1){};

       \node (p1) at (-4, 0){};
       \node (p2) at (-4, 1){};
        \node (i1) at (2, 1){};
        \node (i2) at (2,0){};

        \node (mt) at (0, 2){};
        \node (mm) at (0, 1) {};
        \node (mb) at (0, 0) {};


       \draw [out=45, in=135, relative, decoration={markings, mark=at position 0.5 with {\arrow{<}}}, postaction={decorate}] 
       (b.center) to (d);
       \draw [out=30, in=135, relative, decoration={markings, mark=at position 0.5 with {\arrow{>}}}, postaction={decorate}] 
       (d.center) to (h.center);

       \draw [out=45, in=160, relative, decoration={markings, mark=at position 0.5 with {\arrow{}}}, postaction={decorate}] 
       (h.center) to (i1);
       \draw [out=10, in=135, relative, decoration={markings, mark=at position 0.5 with {\arrow{}}}, postaction={decorate}] 
       (i1.center) to (j.center);

       \draw [out=45, in=135, relative, decoration={markings, mark=at position 0.5 with {\arrow{<}}}, postaction={decorate}] 
       (j.center) to (n);
       \draw [out=30, in=145, relative, decoration={markings, mark=at position 0.5 with {\arrow{>}}}, postaction={decorate}] 
       (n.center) to  (p.center);

       \draw [out=45, in=135, relative, decoration={markings, mark=at position 0.5 with {\arrow{}}}, postaction={decorate}] 
       (p.center) to (p1);
       \draw [out=60, in=135, relative, decoration={markings, mark=at position 0.6 with {\arrow{}}}, postaction={decorate}] 
       (p1.center) to (p2);
       \draw [out=45, in=135, relative, decoration={markings, mark=at position 0.5 with {\arrow{}}}, postaction={decorate}] 
       (p2.center) to (a);
       \draw [out=60, in=135, relative, decoration={markings, mark=at position 0.5 with {\arrow{}}}, postaction={decorate}] 
       (a.center) to (b.center);


        \draw [ decoration={markings, mark=at position 0.5 with {\arrow{<}}}, postaction={decorate}]
       (n) to (p1.center);
       \draw [out = -45, in =-135, relative, decoration={markings, mark=at position 0.6 with {\arrow{}}}, postaction={decorate}]
       (p1) to (p2.center);
       \draw [out = -45, in =-135, relative, decoration={markings, mark=at position 0.7 with {\arrow{}}}, postaction={decorate}]
       (p2) to (a.center);
       \draw [ decoration={markings, mark=at position 0.5 with {\arrow{>}}}, postaction={decorate}]
       (a) to (d.center);

        \draw [decoration={markings, mark=at position 0.5 with {\arrow{<}}}, postaction={decorate}]
       (d) to (mt.center);
        \draw [out = -45, in =-135, relative, decoration={markings, mark=at position 0.5 with {\arrow{}}}, postaction={decorate}]
       (mt) to (mm.center);
        \draw [out = -45, in =-135, relative, decoration={markings, mark=at position 0.5 with {\arrow{}}}, postaction={decorate}]
       (mm) to (mb.center);
       \draw [decoration={markings, mark=at position 0.5 with {\arrow{>}}}, postaction={decorate}]
       (mb) to (n.center);

        \draw [in=-90, out=-45, relative, decoration={markings, mark=at position 0.5 with {\arrow{>}}}, postaction={decorate}]
       (i1.center) to (mt);
        \draw [out = 45, in =135, relative, decoration={markings, mark=at position 0.5 with {\arrow{<}}}, postaction={decorate}]
       (mt.center) to (mm);
        \draw [out = 45, in =135, relative, decoration={markings, mark=at position 0.5 with {\arrow{}}}, postaction={decorate}]
       (mm.center) to (mb);
       \draw [in=-135, out =-45, relative, decoration={markings, mark=at position 0.5 with {\arrow{<}}}, postaction={decorate}]
       (mb.center) to (i1);





       \begin{scope}[scale = 1, xshift=7cm, yshift= -0.5cm]

        \node (a) at (-3,2){};
       \node (b) at (-3,3){};
       \node (d) at (-2,3){};
       \node (e) at (0,3) {};
       \node (f) at (0,3){};
       \node (h) at (1,3){};
       \node (i) at (1,2) {};
       \node (j) at (1,0){};
       \node (l) at (0,0){};
       \node (m) at (0,0) {};
       \node (n) at (-2,0){};
       \node (p) at (-3,0){};

       \node (p2) at (-3, 1){};
        \node (i1) at (1, 1){};

        \node (mt) at (0, 2){};
        \node (mm) at (0, 1) {};
        \node (mb) at (0, 0) {};

        \draw[]
        (e) .. controls (3,3) and (3,0).. (m) 
        (b) -- (e) 
        node[below, text centered, text width = 1.5cm, midway, yshift=-0.5cm]{}  
        (m) -- (p)
        (p) -- (b)
        (e) -- (m)
        node[below, text centered, text width = 3cm, xshift=-0.5cm, yshift=-0.25cm]{$A$-state graph}
        node[right, text centered, midway, xshift=0.2cm]{};

        \filldraw[color=cyan, fill=cyan!5] 
        (a) circle (1/3)
        (b) circle (1/3)
        (e) circle (1/3)
        (p2) circle (1/3)
        (m) circle (1/3)
        (p) circle (1/3)

        (mt) circle (1/3)
        (mm) circle (1/3);
       \end{scope}
       
\end{scope}
       
       \end{scope}
  \end{tikzpicture}
    \caption{A $(6,4)$-Balanced diagram (top) that is equivalent to an $(8,4)$-Oddly Balanced diagram (bottom)}
    \label{fig:k-sequence is not invariant}
\end{figure}

%% file: Figures/clasp_example_1.tex
\begin{figure}[h]
    \centering
  \begin{tikzpicture}
      \begin{scope}[scale=0.35, >=Stealth]

\begin{scope}[scale=0.9]
    
        \begin{scope}[scale=0.9, xshift=-6cm]
          \node(a) at (-5,0) {};

         \node(b) at (-2,3){};
         \node(c) at (-1,0){};
         \node (d) at (-2, -3) {};
         \node (e) at (-5.75,-3.5){};
         \node (f) at (-8, -3){};
         \node (g) at (-9, 0){};
         \node (h) at (-8,3){};
         \node (i) at (-5.75, 3.5){};

          \node(aul) at (-6, 2){};
          \node (aur) at (-4,2){};
          \node(adr) at (-4, -2){};
          \node (adl) at (-6,-2){};
          
          \node (al) at (-5.75,0){};
          \node (ar) at (-4.25,0){};

          \node (corner1) at (-12, 7){};
         \node (corner2) at (2, -7){};

         \node (circlecenter) at (-5,0){};


          \draw [decoration={markings, mark=at position 0.4 with {\arrow{<}}}, postaction={decorate}]
          (aur.center) .. controls (ar) and (al).. (aul.center);
          \draw [ decoration={markings, mark=at position 0.55 with {\arrow{<}}}, postaction={decorate}]
          (adl.center) .. controls (al) and (ar).. (adr.center);


          \draw[out=45, in=135, relative, decoration={markings, mark=at position 0.5 with {\arrow{}}}, postaction={decorate}]
          (aur.center) to (b.center);
          \draw[out=-45, in=-135, relative, decoration={markings, mark=at position 0.7 with {\arrow{<}}}, postaction={decorate}]
          (aul.center) to (h);
          \draw[out=45, in=135, relative, decoration={markings, mark=at position 0.5 with {\arrow{>}}}, postaction={decorate}]
          (adl.center) to (f.center);
          \draw[out=-45, in=-135, relative, decoration={markings, mark=at position 0.7 with {\arrow{<}}}, postaction={decorate}]
          (adr.center) to (d);

          \draw[out=60, in =120,relative, decoration={markings, mark=at position 0.7 with {\arrow{<}}}, postaction={decorate}]
          (h.center) to (b);
          \draw[out=60, in =120,relative, decoration={markings, mark=at position 0.7 with {\arrow{<}}}, postaction={decorate}]
          (d.center) to (f);
          \draw[out=20, in=160,relative, decoration={markings, mark=at position 0.7 with {\arrow{<}}}, postaction={decorate}]
          (b) to (d.center);
          \draw[out=60, in =120, relative, decoration={markings, mark=at position 0.3 with {\arrow{>}}}, postaction={decorate}]
          (b.center) to (d);
          \draw[out=20, in=160,relative, decoration={markings, mark=at position 0.7 with {\arrow{<}}}, postaction={decorate}]
          (f) to (h.center);
          \draw[out=60, in =120, relative, decoration={markings, mark=at position 0.3 with {\arrow{>}}}, postaction={decorate}]
          (f.center) to (h);


          

      \end{scope}

        \begin{scope}[scale=0.9, xshift=6cm]
          \node(a) at (-5,0) {};

         \node(b) at (-2,3){};
         \node(c) at (-1,0){};
         \node (d) at (-2, -3) {};
         \node (e) at (-5.75,-3.5){};
         \node (f) at (-8, -3){};
         \node (g) at (-9, 0){};
         \node (h) at (-8,3){};
         \node (i) at (-5.75, 3.5){};

          \node(aul) at (-6, 2){};
          \node (aur) at (-4,2){};
          \node(adr) at (-4, -2){};
          \node (adl) at (-6,-2){};
          
          \node (al) at (-5.75,0){};
          \node (ar) at (-4.25,0){};

          \node (corner1) at (-12, 7){};
         \node (corner2) at (2, -7){};

         \node (circlecenter) at (-5,0){};


          \draw [very thick, magenta]
          (aur.center) .. controls (ar) and (al).. (aul.center);
          \draw [very thick, blue]
          (adl.center) .. controls (al) and (ar).. (adr.center);


          \draw[out=45, in=135, relative, very thick, magenta]
          (aur.center) to (b.center);
          \draw[out=-45, in=-135, relative, very thick, magenta]
          (aul.center) to (h);
          \draw[out=45, in=135, relative, very thick, blue]
          (adl.center) to (f.center);
          \draw[out=-45, in=-135, relative, very thick, blue]
          (adr.center) to (d);

          \draw[out=60, in =120,relative, very thick, magenta]
          (h.center) to (b);
          \draw[out=60, in =120,relative, very thick, blue]
          (d.center) to (f);
          \draw[out=20, in=160,relative, very thick, Emerald]
          (b) to (d.center);
          \draw[out=60, in =120, relative, very thick, Emerald]
          (b.center) to (d);
          \draw[out=20, in=160,relative, very thick, yellow]
          (f) to (h.center);
          \draw[out=60, in =120, relative, very thick, yellow]
          (f.center) to (h);


          

      \end{scope}

      \begin{scope}[scale=1.3, xshift=8cm]
    \node (a) at (-2, 0){};
    \node (b) at (-2, 2){};
    \node (c) at (0, 2){};
    \node (d) at (2, 2){};
    \node (e) at (2, 0) {};
    \node (f) at (2,-2) {};
    \node (g) at (0, -2){};
    \node (h) at (-2, -2){};

    \draw
    (a) -- (c) -- (e) -- (g) -- (a);

    \filldraw[color=magenta, fill=magenta!5]
    (c) circle (1/3) ;
    \filldraw[color=Emerald, fill=Emerald!5]
    (e) circle (1/3);
    \filldraw[color=blue, fill=blue!5]
    (g) circle (1/3);
    \filldraw[color=yellow, fill=yellow!5]
    (a) circle (1/3)  ;

    \end{scope}

\end{scope}

    
    \begin{scope}[scale=0.9, xshift=0cm, yshift=-6cm]
    
    \begin{scope}[scale=0.9, xshift=-5cm]
    \draw [-{Stealth}, double] 
    (-6,1) -- (-6,-2);
     \end{scope}
    \begin{scope}[scale=0.9, xshift=7cm]
    \draw [-{Stealth}, double] 
    (-6,1) -- (-6,-2);
     \end{scope}
    \begin{scope}[scale=0.9, xshift=17.5cm]
    \draw [-{Stealth}, double] 
    (-6,1) -- (-6,-2);
    \end{scope}
    
    \end{scope}
    

\begin{scope}[scale=0.9, yshift = -13cm]
    
            \begin{scope}[scale=0.9, xshift=-6cm]
          \node(a) at (-5,0) {};

         \node(b) at (-2,3){};
         \node(c) at (-1,0){};
         \node (d) at (-2, -3) {};
         \node (e) at (-5.75,-3.5){};
         \node (f) at (-8, -3){};
         \node (g) at (-9, 0){};
         \node (h) at (-8,3){};
         \node (i) at (-5.75, 3.5){};

          \node(aul) at (-6, 2){};
          \node (aur) at (-4,2){};
          \node(adr) at (-4, -2){};
          \node (adl) at (-6,-2){};
          
          \node (al) at (-5.75,0){};
          \node (ar) at (-4.25,0){};

          \node (corner1) at (-12, 7){};
         \node (corner2) at (2, -7){};

         \node (circlecenter) at (-5,0){};



          \draw [out=0, in=-160, relative, decoration={markings, mark=at position 0.5 with {\arrow{}}}, postaction={decorate}]
          (adr.center) to (ar);
          \draw [out=-60, in=-120, relative]
          (ar) to (al.center);
          \draw [out=-15, in=-180, relative, decoration={markings, mark=at position 0.5 with {\arrow{>}}}, postaction={decorate}]
          (al.center) to (adl.center);
          \draw [out=-15, in=180, relative, decoration={markings, mark=at position 0.5 with {\arrow{}}}, postaction={decorate}]
          (aul.center) to (al);
          \draw [out=-60, in=-120, relative]
          (al) to (ar.center);
          \draw [out=-15, in=180, relative, decoration={markings, mark=at position 0.5 with {\arrow{>}}}, postaction={decorate}]
          (ar.center) to (aur.center);

          \draw[out=45, in=135, relative, decoration={markings, mark=at position 0.5 with {\arrow{}}}, postaction={decorate}]
          (aur.center) to (b.center);
          \draw[out=-45, in=-135, relative, decoration={markings, mark=at position 0.7 with {\arrow{<}}}, postaction={decorate}]
          (aul.center) to (h);
          \draw[out=45, in=135, relative]
          (adl.center) to (f.center);
          \draw[out=-45, in=-135, relative, decoration={markings, mark=at position 0.7 with {\arrow{<}}}, postaction={decorate}]
          (adr.center) to (d);

          \draw[out=60, in =120,relative, decoration={markings, mark=at position 0.7 with {\arrow{<}}}, postaction={decorate}]
          (h.center) to (b);
          \draw[out=60, in =120,relative, decoration={markings, mark=at position 0.7 with {\arrow{<}}}, postaction={decorate}]
          (d.center) to (f);
          \draw[out=20, in=160,relative, decoration={markings, mark=at position 0.7 with {\arrow{<}}}, postaction={decorate}]
          (b) to (d.center);
          \draw[out=60, in =120, relative, decoration={markings, mark=at position 0.3 with {\arrow{>}}}, postaction={decorate}]
          (b.center) to (d);
          \draw[out=20, in=160,relative, decoration={markings, mark=at position 0.7 with {\arrow{<}}}, postaction={decorate}]
          (f) to (h.center);
          \draw[out=60, in =120, relative, decoration={markings, mark=at position 0.3 with {\arrow{>}}}, postaction={decorate}]
          (f.center) to (h);


          

      \end{scope}

            \begin{scope}[scale=0.9, xshift=6cm]
          \node(a) at (-5,0) {};

         \node(b) at (-2,3){};
         \node(c) at (-1,0){};
         \node (d) at (-2, -3) {};
         \node (e) at (-5.75,-3.5){};
         \node (f) at (-8, -3){};
         \node (g) at (-9, 0){};
         \node (h) at (-8,3){};
         \node (i) at (-5.75, 3.5){};

          \node(aul) at (-6, 2){};
          \node (aur) at (-4,2){};
          \node(adr) at (-4, -2){};
          \node (adl) at (-6,-2){};
          
          \node (al) at (-5.75,0){};
          \node (ar) at (-4.25,0){};

          \node (corner1) at (-12, 7){};
         \node (corner2) at (2, -7){};

         \node (circlecenter) at (-5,0){};



          \draw [out=0, in=-160, relative, very thick, magenta]
          (adr.center) to (ar);
          \draw [out=-60, in=-120, relative, very thick, blue]
          (ar) to (al.center);
          \draw [out=-15, in=-180, relative, very thick, magenta]
          (al.center) to (adl.center);
          \draw [out=-15, in=180, relative, very thick, magenta]
          (aul.center) to (al);
          \draw [out=-60, in=-120, relative, very thick, blue]
          (al) to (ar.center);
          \draw [out=-15, in=180, relative, very thick, magenta]
          (ar.center) to (aur.center);

          \draw[out=45, in=135, relative, very thick, magenta]
          (aur.center) to (b.center);
          \draw[out=-45, in=-135, relative, very thick, magenta]
          (aul.center) to (h);
          \draw[out=45, in=135, relative, very thick, magenta]
          (adl.center) to (f.center);
          \draw[out=-45, in=-135, relative, very thick, magenta]
          (adr.center) to (d);

          \draw[out=60, in =120,relative, very thick, magenta]
          (h.center) to (b);
          \draw[out=60, in =120,relative, very thick, magenta]
          (d.center) to (f);
          \draw[out=20, in=160,relative, very thick, Emerald]
          (b) to (d.center);
          \draw[out=60, in =120, relative, very thick, Emerald]
          (b.center) to (d);
          \draw[out=20, in=160,relative, very thick, yellow]
          (f) to (h.center);
          \draw[out=60, in =120, relative, very thick, yellow]
          (f.center) to (h);


          

      \end{scope}

      \begin{scope}[scale=1.3, xshift=8cm]
    \node (a) at (-2, 0){};
    \node (b) at (-2, 2){};
    \node (c) at (0, 2){};
    \node (d) at (2, 2){};
    \node (e) at (2, 0) {};
    \node (f) at (2,-2) {};
    \node (g) at (0, -2){};
    \node (h) at (-2, -2){};

    \draw
    (a) -- (c) -- (e);
     \draw[out = 15, in = 165, relative]
      (g) to (c)
      (c) to (g);
    \draw[out = 20, in = 160, relative]
        (a) to (c)
        (c) to (e);

    \filldraw[color=magenta, fill=magenta!5]
    (c) circle (1/3) ;
    \filldraw[color=Emerald, fill=Emerald!5]
    (e) circle (1/3);
    \filldraw[color=blue, fill=blue!5]
    (g) circle (1/3);
    \filldraw[color=yellow, fill=yellow!5]
    (a) circle (1/3)  ;

    \end{scope}

\end{scope}
       
       \end{scope}
  \end{tikzpicture}
    \caption{Clasping a $4$-Balanced diagram of type $1$ to obtain a Balanced diagram of type $0$}
    \label{fig:clasp example 1}
\end{figure}
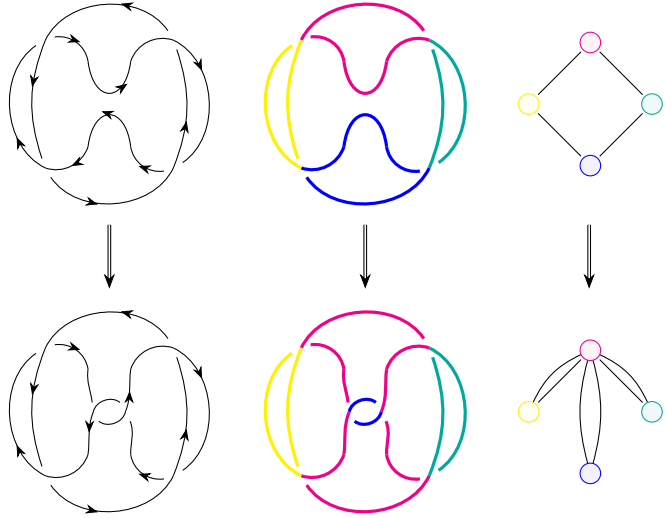

%% file: Figures/clasp_move_0.tex
\begin{figure}[h]
    \centering
  \begin{tikzpicture}
      \begin{scope}[scale=0.5, >=Stealth]

\begin{scope}[scale=0.9]
    
            \begin{scope}[scale=0.9, xshift=-3cm]

          \node(a) at (-5,0) {};

          \node at (-5.4,3.5) {Link diagram};

          \filldraw[color=green, fill=green!5]
          (a) ellipse (1.85 and 2.2)
          (-7.5,-3) rectangle (-2.5,3);

          
          \node(aul) at (-6, 2){};
          \node (aur) at (-4,2){};
          \node(adr) at (-4, -2){};
          \node (adl) at (-6,-2){};
          
          \node (al) at (-5.75,0){};
          \node (ar) at (-4.25,0){};

          \draw [out=-90, in=-90, relative, decoration={markings, mark=at position 0.5 with {\arrow{>}}}, postaction={decorate}]
          (aul) to (aur);
          \draw [out=90, in=90, relative, decoration={markings, mark=at position 0.5 with {\arrow{<}}}, postaction={decorate}]
          (adl) to (adr);

      \end{scope}

        \begin{scope}[scale=0.9, xshift=3cm]
          \node(a) at (-5,0) {};

          \node at (-5,3.5) {Different $A$-circles};

          \filldraw[color=green, fill=green!5](a) ellipse (1.85 and 2.2)
          (-7.5,-3) rectangle (-2.5,3);

          
          \node(aul) at (-6, 2){};
          \node (aur) at (-4,2){};
          \node(adr) at (-4, -2){};
          \node (adl) at (-6,-2){};
          
          \node (al) at (-5.75,0){};
          \node (ar) at (-4.25,0){};

          \draw [   thick, magenta, out=-90, in=-90, relative, decoration={markings, mark=at position 0.5 with {\arrow{}}}, postaction={decorate}]
          (aul) to (aur);
          \draw [   thick, blue, out=90, in=90, relative, decoration={markings, mark=at position 0.5 with {\arrow{}}}, postaction={decorate}]
          (adl) to (adr);


          \draw [   thick, magenta, out=90, in=90, relative, decoration={markings, mark=at position 0.5 with {\arrow{}}}, postaction={decorate}]
          (aul.center) to (aur.center);
          \draw [   thick, blue, out=-90, in=-90, relative, decoration={markings, mark=at position 0.5 with {\arrow{}}}, postaction={decorate}]
          (adl.center) to (adr.center);
      \end{scope}
      

\end{scope}

    
    \begin{scope}[scale=0.9, xshift=0cm, yshift=-3.5cm]
    
    \begin{scope}[scale=0.9, xshift=-2cm]
    \draw [-{Stealth}, double] 
    (-6,0) -- (-6,-1);
    \draw [-{Stealth}, double] 
    (0,0) -- (0,-1);
    \end{scope}
    
    \end{scope}
    

\begin{scope}[scale=0.9, yshift = -8cm]
    
           \begin{scope}[scale=0.9, xshift=-3cm]
          \node(a) at (-5,0) {};

          \filldraw[color=green, fill=green!5]
          (a) ellipse (1.85 and 2.2)
          (-7.5,-3) rectangle (-2.5,3);

          
          \node(aul) at (-6, 2){};
          \node (aur) at (-4,2){};
          \node(adr) at (-4, -2){};
          \node (adl) at (-6,-2){};
          
          \node (al) at (-5.75,0){};
          \node (ar) at (-4.25,0){};


          \draw [out=0, in=-160, relative, decoration={markings, mark=at position 0.5 with {\arrow{>}}}, postaction={decorate}]
          (adr) to (ar);
          \draw [out=-60, in=-120, relative]
          (ar) to (al.center);
          \draw [out=-15, in=-180, relative, decoration={markings, mark=at position 0.5 with {\arrow{>}}}, postaction={decorate}]
          (al.center) to (adl);
          \draw [out=-15, in=180, relative, decoration={markings, mark=at position 0.5 with {\arrow{>}}}, postaction={decorate}]
          (aul) to (al);
          \draw [out=-60, in=-120, relative]
          (al) to (ar.center);
          \draw [out=-15, in=180, relative, decoration={markings, mark=at position 0.5 with {\arrow{>}}}, postaction={decorate}]
          (ar.center) to (aur);
      \end{scope}

        \begin{scope}[scale=0.9, xshift=3cm]
          \node(a) at (-5,0) {};

          \filldraw[color=green, fill=green!5]
          (a) ellipse (1.85 and 2.2)
          (-7.5,-3) rectangle (-2.5,3);

          
          \node(aul) at (-6, 2){};
          \node (aur) at (-4,2){};
          \node(adr) at (-4, -2){};
          \node (adl) at (-6,-2){};
          
          \node (al) at (-5.75,0){};
          \node (ar) at (-4.25,0){};
          
          \node [below left] (lbl) at (al) {};
          \node [below right] (lbr) at (al) {};
          \node [above left] (lal) at (al) {};
          \node [above right] (lar) at (al) {};

          \node [below left] (rbl) at (ar) {};
          \node [below right] (rbr) at (ar) {};
          \node [above left] (ral) at (ar) {};
          \node [above right] (rar) at (ar) {};


          \draw [out=0, in=-160, relative,    thick, magenta, cap=round]
          (adr) to (rbr.center);
          \draw [   thick, blue, cap=round, out=60, in=120, relative, decoration={markings, mark=at position 0.5 with {\arrow{}}}, postaction={decorate}]
          (rbl.center) to (lbr.center);
          \draw [   thick, magenta, out=-15, in=-180, relative, decoration={markings, mark=at position 0.5 with {\arrow{}}}, postaction={decorate}]
          (lbl.center) to (adl);
          \draw [   thick, magenta, out=-15, in=180, relative, decoration={markings, mark=at position 0.5 with {\arrow{}}}, postaction={decorate}]
          (aul) to (lal.center);
          \draw [   thick, blue, out=60, in=120, relative, decoration={markings, mark=at position 0.5 with {\arrow{}}}, postaction={decorate}]
          (lar.center) to (ral.center);
          \draw [   thick, magenta, out=-15, in=180, relative, decoration={markings, mark=at position 0.5 with {\arrow{}}}, postaction={decorate}]
          (rar.center) to (aur);
          
          \draw [   thick, magenta, out=90, in=90, relative, decoration={markings, mark=at position 0.5 with {\arrow{}}}, postaction={decorate}]
          (aul.center) to (aur.center);
          \draw [   thick, magenta, out=-90, in=-90, relative, decoration={markings, mark=at position 0.5 with {\arrow{}}}, postaction={decorate}]
          (adl.center) to (adr.center);
          
          \draw[   thick, magenta, cap=round]
          (lbl.center) -- (lal.center)
          (rbr.center) -- (rar.center);
          \draw[   thick, blue, cap=round]
          (lar.center) -- (lbr.center)
          (ral.center) -- (rbl.center);
          
      \end{scope}
      

\end{scope}
       
       \end{scope}
  \end{tikzpicture}
    \caption{Adding a clasp with positive crossings}
    \label{fig:clasp move 0}
\end{figure}
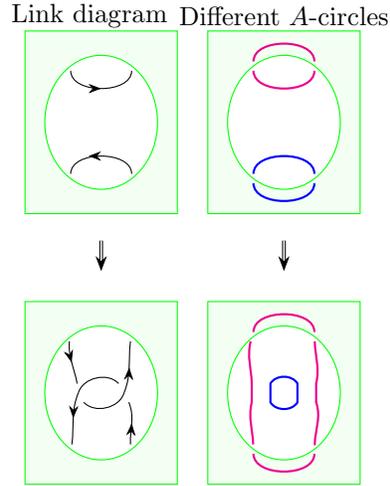

%% file: Figures/clasp_move_1.tex
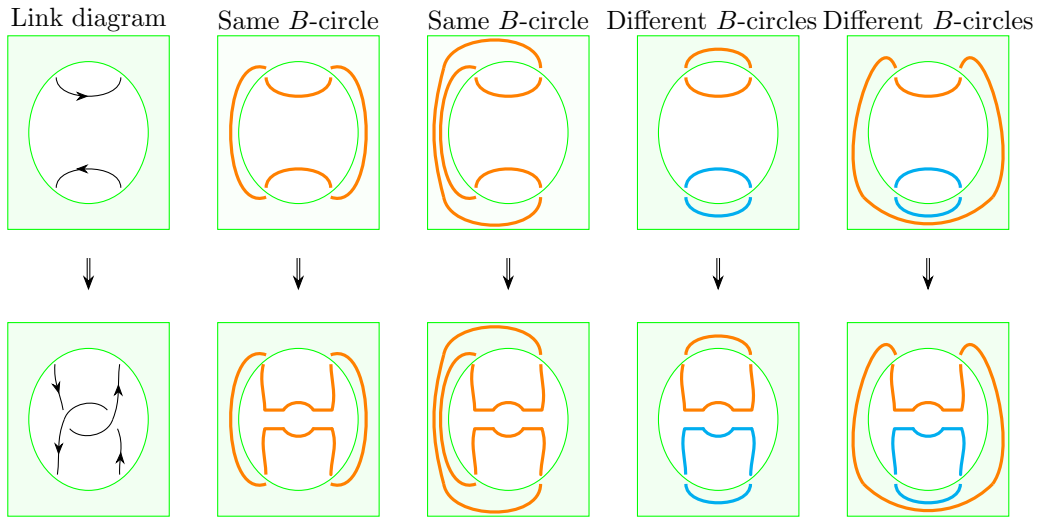
\begin{figure}[h]
    \centering
  \begin{tikzpicture}
      \begin{scope}[scale=0.53, >=Stealth]

\begin{scope}[scale=0.9]
    
            \begin{scope}[scale=0.9, xshift=-13cm]

          \node(a) at (-5,0) {};

          \node at (-5,3.5) {Link diagram};

          \filldraw[color=green, fill=green!5]
          (a) ellipse (1.85 and 2.2)
          (-7.5,-3) rectangle (-2.5,3);

          
          \node(aul) at (-6, 2){};
          \node (aur) at (-4,2){};
          \node(adr) at (-4, -2){};
          \node (adl) at (-6,-2){};
          
          \node (al) at (-5.75,0){};
          \node (ar) at (-4.25,0){};

          \draw [out=-90, in=-90, relative, decoration={markings, mark=at position 0.5 with {\arrow{>}}}, postaction={decorate}]
          (aul) to (aur);
          \draw [out=90, in=90, relative, decoration={markings, mark=at position 0.5 with {\arrow{<}}}, postaction={decorate}]
          (adl) to (adr);

      \end{scope}
      
        \begin{scope}[scale=0.9, xshift=-6.5cm]
          \node(a) at (-5,0) {};

          \node at (-5,3.5) {Same $B$-circle};

          \filldraw[color=green, fill=green!2]
          (a) ellipse (1.85 and 2.2)
          (-7.5,-3) rectangle (-2.5,3);

          
          \node(aul) at (-6, 2){};
          \node (aur) at (-4,2){};
          \node(adr) at (-4, -2){};
          \node (adl) at (-6,-2){};
          
          \node (al) at (-5.75,0){};
          \node (ar) at (-4.25,0){};

          \draw [very thick, orange, out=-90, in=-90, relative, decoration={markings, mark=at position 0.5 with {\arrow{}}}, postaction={decorate}]
          (aul) to (aur);
          \draw [very thick, orange, out=90, in=90, relative, decoration={markings, mark=at position 0.5 with {\arrow{}}}, postaction={decorate}]
          (adl) to (adr);


          \draw [very thick, orange, out=110, in=70, relative, decoration={markings, mark=at position 0.5 with {\arrow{}}}, postaction={decorate}]
          (aur.center) to (adr.center);
          \draw [very thick, orange, out=110, in=70, relative, decoration={markings, mark=at position 0.5 with {\arrow{}}}, postaction={decorate}]
          (adl.center) to (aul.center);
      \end{scope}
      

\begin{scope}[scale=0.9, xshift=0cm]
          \node(a) at (-5,0) {};

          \node at (-5,3.5) {Same $B$-circle};

          \filldraw[color=green, fill=green!2]
          (a) ellipse (1.85 and 2.2)
          (-7.5,-3) rectangle (-2.5,3);

          
          \node(aul) at (-6, 2){};
          \node (aur) at (-4,2){};
          \node(adr) at (-4, -2){};
          \node (adl) at (-6,-2){};
          
          \node (al) at (-5.75,0){};
          \node (ar) at (-4.25,0){};

          \node (h) at (-7, 2){};
          \node (f) at (-7, -2){};

          \draw [very thick, orange, out=-90, in=-90, relative, decoration={markings, mark=at position 0.5 with {\arrow{}}}, postaction={decorate}]
          (aul) to (aur);
          \draw [very thick, orange, out=90, in=90, relative, decoration={markings, mark=at position 0.5 with {\arrow{}}}, postaction={decorate}]
          (adl) to (adr);


          \draw [very thick, orange, out=-90, in=-100, relative, decoration={markings, mark=at position 0.5 with {\arrow{}}}, postaction={decorate}]
          (aur.center) to (h.center);
          \draw [very thick, orange, out=-15, in=-165, relative, decoration={markings, mark=at position 0.5 with {\arrow{}}}, postaction={decorate}]
          (h.center) to (f.center);
          \draw [very thick, orange, out=-80, in=-90, relative, decoration={markings, mark=at position 0.5 with {\arrow{}}}, postaction={decorate}]
          (f.center) to (adr.center);
          \draw [very thick, orange, out=110, in=70, relative, decoration={markings, mark=at position 0.5 with {\arrow{}}}, postaction={decorate}]
          (adl.center) to (aul.center);
      \end{scope}

        \begin{scope}[scale=0.9, xshift=6.5cm]
          \node(a) at (-5,0) {};

          \node at (-5.2,3.5) {Different $B$-circles};

          \filldraw[color=green, fill=green!5](a) ellipse (1.85 and 2.2)
          (-7.5,-3) rectangle (-2.5,3);

          
          \node(aul) at (-6, 2){};
          \node (aur) at (-4,2){};
          \node(adr) at (-4, -2){};
          \node (adl) at (-6,-2){};
          
          \node (al) at (-5.75,0){};
          \node (ar) at (-4.25,0){};

          \draw [very thick, orange, out=-90, in=-90, relative, decoration={markings, mark=at position 0.5 with {\arrow{}}}, postaction={decorate}]
          (aul) to (aur);
          \draw [very thick, cyan, out=90, in=90, relative, decoration={markings, mark=at position 0.5 with {\arrow{}}}, postaction={decorate}]
          (adl) to (adr);


          \draw [very thick, orange, out=90, in=90, relative, decoration={markings, mark=at position 0.5 with {\arrow{}}}, postaction={decorate}]
          (aul.center) to (aur.center);
          \draw [very thick, cyan, out=-90, in=-90, relative, decoration={markings, mark=at position 0.5 with {\arrow{}}}, postaction={decorate}]
          (adl.center) to (adr.center);
      \end{scope}

     \begin{scope}[scale=0.9, xshift=13cm]
          \node(a) at (-5,0) {};

          \node at (-5,3.5) {Different $B$-circles};

          \filldraw[color=green, fill=green!5](a) ellipse (1.85 and 2.2)
          (-7.5,-3) rectangle (-2.5,3);

          
          \node(aul) at (-6, 2){};
          \node (aur) at (-4,2){};
          \node(adr) at (-4, -2){};
          \node (adl) at (-6,-2){};
          
          \node (al) at (-5.75,0){};
          \node (ar) at (-4.25,0){};

          \node (h) at (-7, 2){};
          \node (f) at (-7, -2) {};
          \node (b) at (-3, 2) {};
          \node (d) at (-3, -2){};

          \draw [very thick, orange, out=-90, in=-90, relative, decoration={markings, mark=at position 0.5 with {\arrow{}}}, postaction={decorate}]
          (aul) to (aur);
          \draw [very thick, cyan, out=90, in=90, relative, decoration={markings, mark=at position 0.5 with {\arrow{}}}, postaction={decorate}]
          (adl) to (adr);


          \draw [very thick, orange, out=-145, in=-130, relative, decoration={markings, mark=at position 0.5 with {\arrow{}}}, postaction={decorate}]
          (aul.center) to (f.center);
          \draw [very thick, orange, out=-45, in=-135, relative, decoration={markings, mark=at position 0.5 with {\arrow{}}}, postaction={decorate}]
          (f.center) to (d.center);
          \draw [very thick, orange, out=-50, in=-35, relative, decoration={markings, mark=at position 0.5 with {\arrow{}}}, postaction={decorate}]
          (d.center) to (aur.center);
          \draw [very thick, cyan, out=-90, in=-90, relative, decoration={markings, mark=at position 0.5 with {\arrow{}}}, postaction={decorate}]
          (adl.center) to (adr.center);
      \end{scope}

\end{scope}

    
    \begin{scope}[scale=0.9, xshift=0cm, yshift=-3.5cm]
    
    \begin{scope}[scale=0.9, xshift=-5cm]
    \draw [-{Stealth}, double] 
    (-13,0) -- (-13,-1);
    \draw [-{Stealth}, double] 
    (-6.5,0) -- (-6.5,-1);
    \draw [-{Stealth}, double] 
    (0,0) -- (0,-1);
    \draw [-{Stealth}, double] 
    (6.5,0) -- (6.5,-1);
    \draw [-{Stealth}, double] 
    (13,0) -- (13,-1);
    \end{scope}
    
    \end{scope}
    

\begin{scope}[scale=0.9, yshift = -8cm]
    
           \begin{scope}[scale=0.9, xshift=-13cm]
          \node(a) at (-5,0) {};

          \filldraw[color=green, fill=green!5]
          (a) ellipse (1.85 and 2.2)
          (-7.5,-3) rectangle (-2.5,3);

          
          \node(aul) at (-6, 2){};
          \node (aur) at (-4,2){};
          \node(adr) at (-4, -2){};
          \node (adl) at (-6,-2){};
          
          \node (al) at (-5.75,0){};
          \node (ar) at (-4.25,0){};


          \draw [out=0, in=-160, relative, decoration={markings, mark=at position 0.5 with {\arrow{>}}}, postaction={decorate}]
          (adr) to (ar);
          \draw [out=-60, in=-120, relative, decoration={markings, mark=at position 0.5 with {\arrow{}}}, postaction={decorate}]
          (ar) to (al.center);
          \draw [out=-15, in=-180, relative, decoration={markings, mark=at position 0.65 with {\arrow{>}}}, postaction={decorate}]
          (al.center) to (adl);
          \draw [out=-15, in=180, relative, decoration={markings, mark=at position 0.65 with {\arrow{>}}}, postaction={decorate}]
          (aul) to (al);
          \draw [out=-60, in=-120, relative, decoration={markings, mark=at position 0.5 with {\arrow{}}}, postaction={decorate}]
          (al) to (ar.center);
          \draw [out=-15, in=180, relative, decoration={markings, mark=at position 0.65 with {\arrow{>}}}, postaction={decorate}]
          (ar.center) to (aur);
      \end{scope}
      
        \begin{scope}[scale=0.9,xshift = -6.5cm]
          \node(a) at (-5,0) {};

          \filldraw[color=green, fill=green!5]
          (a) ellipse (1.85 and 2.2)
          (-7.5,-3) rectangle (-2.5,3);

          
          \node(aul) at (-6, 2){};
          \node (aur) at (-4,2){};
          \node(adr) at (-4, -2){};
          \node (adl) at (-6,-2){};
          
          \node (al) at (-5.75,0){};
          \node (ar) at (-4.25,0){};
          
          \node [below left] (lbl) at (al) {};
          \node [below right] (lbr) at (al) {};
          \node [above left] (lal) at (al) {};
          \node [above right] (lar) at (al) {};

          \node [below left] (rbl) at (ar) {};
          \node [below right] (rbr) at (ar) {};
          \node [above left] (ral) at (ar) {};
          \node [above right] (rar) at (ar) {};


          \draw [out=0, in=-160, relative, very thick, orange, cap=round]
          (adr) to (rbr.center);
          \draw [very thick, orange, cap=round, out=60, in=120, relative, decoration={markings, mark=at position 0.5 with {\arrow{}}}, postaction={decorate}]
          (rbl.center) to (lbr.center);
          \draw [very thick, orange, out=-15, in=-180, relative, decoration={markings, mark=at position 0.5 with {\arrow{}}}, postaction={decorate}]
          (lbl.center) to (adl);
          \draw [very thick, orange, out=-15, in=180, relative, decoration={markings, mark=at position 0.5 with {\arrow{}}}, postaction={decorate}]
          (aul) to (lal.center);
          \draw [very thick, orange, out=60, in=120, relative, decoration={markings, mark=at position 0.5 with {\arrow{}}}, postaction={decorate}]
          (lar.center) to (ral.center);
          \draw [very thick, orange, out=-15, in=180, relative, decoration={markings, mark=at position 0.5 with {\arrow{}}}, postaction={decorate}]
          (rar.center) to (aur);

          \draw [very thick, orange, out=110, in=70, relative, decoration={markings, mark=at position 0.5 with {\arrow{}}}, postaction={decorate}]
          (aur.center) to (adr.center);
          \draw [very thick, orange, out=110, in=70, relative, decoration={markings, mark=at position 0.5 with {\arrow{}}}, postaction={decorate}]
          (adl.center) to (aul.center);

          \draw[very thick, orange, cap=round]
          (lbl.center) -- (lbr.center)
          (lal.center) -- (lar.center)
          (rbl.center) -- (rbr.center)
          (ral.center) -- (rar.center);
          
      \end{scope}
      

        \begin{scope}[scale=0.9,xshift = 0cm]
          \node(a) at (-5,0) {};

          \filldraw[color=green, fill=green!5]
          (a) ellipse (1.85 and 2.2)
          (-7.5,-3) rectangle (-2.5,3);

          
          \node(aul) at (-6, 2){};
          \node (aur) at (-4,2){};
          \node(adr) at (-4, -2){};
          \node (adl) at (-6,-2){};
          
          \node (al) at (-5.75,0){};
          \node (ar) at (-4.25,0){};

          \node (h) at (-7, 2){};
          \node (f) at (-7, -2){};
          
          \node [below left] (lbl) at (al) {};
          \node [below right] (lbr) at (al) {};
          \node [above left] (lal) at (al) {};
          \node [above right] (lar) at (al) {};

          \node [below left] (rbl) at (ar) {};
          \node [below right] (rbr) at (ar) {};
          \node [above left] (ral) at (ar) {};
          \node [above right] (rar) at (ar) {};


          \draw [out=0, in=-160, relative, very thick, orange, cap=round]
          (adr) to (rbr.center);
          \draw [very thick, orange, cap=round, out=60, in=120, relative, decoration={markings, mark=at position 0.5 with {\arrow{}}}, postaction={decorate}]
          (rbl.center) to (lbr.center);
          \draw [very thick, orange, out=-15, in=-180, relative, decoration={markings, mark=at position 0.5 with {\arrow{}}}, postaction={decorate}]
          (lbl.center) to (adl);
          \draw [very thick, orange, out=-15, in=180, relative, decoration={markings, mark=at position 0.5 with {\arrow{}}}, postaction={decorate}]
          (aul) to (lal.center);
          \draw [very thick, orange, out=60, in=120, relative, decoration={markings, mark=at position 0.5 with {\arrow{}}}, postaction={decorate}]
          (lar.center) to (ral.center);
          \draw [very thick, orange, out=-15, in=180, relative, decoration={markings, mark=at position 0.5 with {\arrow{}}}, postaction={decorate}]
          (rar.center) to (aur);

          \draw [very thick, orange, out=-90, in=-100, relative, decoration={markings, mark=at position 0.5 with {\arrow{}}}, postaction={decorate}]
          (aur.center) to (h.center);
          \draw [very thick, orange, out=-15, in=-165, relative, decoration={markings, mark=at position 0.5 with {\arrow{}}}, postaction={decorate}]
          (h.center) to (f.center);
          \draw [very thick, orange, out=-80, in=-90, relative, decoration={markings, mark=at position 0.5 with {\arrow{}}}, postaction={decorate}]
          (f.center) to (adr.center);
          \draw [very thick, orange, out=110, in=70, relative, decoration={markings, mark=at position 0.5 with {\arrow{}}}, postaction={decorate}]
          (adl.center) to (aul.center);

          \draw[very thick, orange, cap=round]
          (lbl.center) -- (lbr.center)
          (lal.center) -- (lar.center)
          (rbl.center) -- (rbr.center)
          (ral.center) -- (rar.center);
          
      \end{scope}

        \begin{scope}[scale=0.9, xshift=6.5cm]
          \node(a) at (-5,0) {};

          \filldraw[color=green, fill=green!5]
          (a) ellipse (1.85 and 2.2)
          (-7.5,-3) rectangle (-2.5,3);

          
          \node(aul) at (-6, 2){};
          \node (aur) at (-4,2){};
          \node(adr) at (-4, -2){};
          \node (adl) at (-6,-2){};
          
          \node (al) at (-5.75,0){};
          \node (ar) at (-4.25,0){};
          
          \node [below left] (lbl) at (al) {};
          \node [below right] (lbr) at (al) {};
          \node [above left] (lal) at (al) {};
          \node [above right] (lar) at (al) {};

          \node [below left] (rbl) at (ar) {};
          \node [below right] (rbr) at (ar) {};
          \node [above left] (ral) at (ar) {};
          \node [above right] (rar) at (ar) {};


          \draw [out=0, in=-160, relative, very thick, cyan, cap=round]
          (adr) to (rbr.center);
          \draw [very thick, cyan, cap=round, out=60, in=120, relative, decoration={markings, mark=at position 0.5 with {\arrow{}}}, postaction={decorate}]
          (rbl.center) to (lbr.center);
          \draw [very thick, cyan, out=-15, in=-180, relative, decoration={markings, mark=at position 0.5 with {\arrow{}}}, postaction={decorate}]
          (lbl.center) to (adl);
          \draw [very thick, orange, out=-15, in=180, relative, decoration={markings, mark=at position 0.5 with {\arrow{}}}, postaction={decorate}]
          (aul) to (lal.center);
          \draw [very thick, orange, out=60, in=120, relative, decoration={markings, mark=at position 0.5 with {\arrow{}}}, postaction={decorate}]
          (lar.center) to (ral.center);
          \draw [very thick, orange, out=-15, in=180, relative, decoration={markings, mark=at position 0.5 with {\arrow{}}}, postaction={decorate}]
          (rar.center) to (aur);
          
          \draw [very thick, orange, out=90, in=90, relative, decoration={markings, mark=at position 0.5 with {\arrow{}}}, postaction={decorate}]
          (aul.center) to (aur.center);
          \draw [very thick, cyan, out=-90, in=-90, relative, decoration={markings, mark=at position 0.5 with {\arrow{}}}, postaction={decorate}]
          (adl.center) to (adr.center);
          
          \draw[very thick, cyan, cap=round]
          (lbl.center) -- (lbr.center)
          (rbl.center) -- (rbr.center);
          \draw[very thick, orange, cap=round]
          (lal.center) -- (lar.center)
          (ral.center) -- (rar.center);
          
      \end{scope}
      

 \begin{scope}[scale=0.9, xshift=13cm]
          \node(a) at (-5,0) {};

          \filldraw[color=green, fill=green!5]
          (a) ellipse (1.85 and 2.2)
          (-7.5,-3) rectangle (-2.5,3);

          
          \node(aul) at (-6, 2){};
          \node (aur) at (-4,2){};
          \node(adr) at (-4, -2){};
          \node (adl) at (-6,-2){};
          
          \node (al) at (-5.75,0){};
          \node (ar) at (-4.25,0){};

          \node (h) at (-7, 2){};
          \node (f) at (-7, -2) {};
          \node (b) at (-3, 2) {};
          \node (d) at (-3, -2){};
          
          \node [below left] (lbl) at (al) {};
          \node [below right] (lbr) at (al) {};
          \node [above left] (lal) at (al) {};
          \node [above right] (lar) at (al) {};

          \node [below left] (rbl) at (ar) {};
          \node [below right] (rbr) at (ar) {};
          \node [above left] (ral) at (ar) {};
          \node [above right] (rar) at (ar) {};


          \draw [out=0, in=-160, relative, very thick, cyan, cap=round]
          (adr) to (rbr.center);
          \draw [very thick, cyan, cap=round, out=60, in=120, relative, decoration={markings, mark=at position 0.5 with {\arrow{}}}, postaction={decorate}]
          (rbl.center) to (lbr.center);
          \draw [very thick, cyan, out=-15, in=-180, relative, decoration={markings, mark=at position 0.5 with {\arrow{}}}, postaction={decorate}]
          (lbl.center) to (adl);
          \draw [very thick, orange, out=-15, in=180, relative, decoration={markings, mark=at position 0.5 with {\arrow{}}}, postaction={decorate}]
          (aul) to (lal.center);
          \draw [very thick, orange, out=60, in=120, relative, decoration={markings, mark=at position 0.5 with {\arrow{}}}, postaction={decorate}]
          (lar.center) to (ral.center);
          \draw [very thick, orange, out=-15, in=180, relative, decoration={markings, mark=at position 0.5 with {\arrow{}}}, postaction={decorate}]
          (rar.center) to (aur);
          
          \draw [very thick, orange, out=-145, in=-130, relative, decoration={markings, mark=at position 0.5 with {\arrow{}}}, postaction={decorate}]
          (aul.center) to (f.center);
          \draw [very thick, orange, out=-45, in=-135, relative, decoration={markings, mark=at position 0.5 with {\arrow{}}}, postaction={decorate}]
          (f.center) to (d.center);
          \draw [very thick, orange, out=-50, in=-35, relative, decoration={markings, mark=at position 0.5 with {\arrow{}}}, postaction={decorate}]
          (d.center) to (aur.center);
          \draw [very thick, cyan, out=-90, in=-90, relative, decoration={markings, mark=at position 0.5 with {\arrow{}}}, postaction={decorate}]
          (adl.center) to (adr.center);
          
          \draw[very thick, cyan, cap=round]
          (lbl.center) -- (lbr.center)
          (rbl.center) -- (rbr.center);
          \draw[very thick, orange, cap=round]
          (lal.center) -- (lar.center)
          (ral.center) -- (rar.center);
          
      \end{scope}

\end{scope}
       
       \end{scope}
  \end{tikzpicture}
    \caption{Adding a clasp with positive crossings does not change the number of $B$-circles.}
    \label{fig:clasp move 1}
\end{figure}

%% file: Figures/v2_cycle_degree_2.tex
\begin{figure}[h]
    \centering
  \begin{tikzpicture}
      \begin{scope}[scale=0.37, >=Stealth]

\begin{scope}[scale=0.9]
    
      \begin{scope}[scale=1.3, xshift=-5cm]
    \node (a) at (-2, 3){};
    \node (b) at (1, 0){};
    \node (c) at (-2, -3){};
    
    \draw
     (a) -- (b) -- (c);

    \filldraw[color=cyan, fill=green!3]
    (a) circle (1/3) 
    (b) circle (1/3)
    (c) circle (1/3);

    \end{scope}

    \begin{scope}[scale=0.9, xshift=-4cm]
    
        \node (corner1) at (-9.7, 6){};
         \node (corner2) at (2.7, -6){};

         \node (circlecenter) at (-3.5,0){};
         
          \filldraw[color=green, fill=green!5, even odd rule]
          (corner1) rectangle (corner2)
          (circlecenter) ellipse (4.5 and 5.5);
        
    \end{scope}

          \begin{scope}[scale=0.9, xshift=12cm]
          \node(a) at (-5,0) {};



         \node(b) at (-2,3){};
         \node(c) at (-1,0){};
         \node (d) at (-2, -3) {};
         \node (e) at (-5.75,-4.25){};
         \node (f) at (-8, -3){};
         \node (g) at (-9, 0){};
         \node (h) at (-8,3){};
         \node (i) at (-5.75, 4.25){};
          
          \node(aul) at (-6, 2){};
          \node (aur) at (-4,2){};
          \node(adr) at (-4, -2){};
          \node (adl) at (-6,-2){};
          
          \node (al) at (-5.75,0){};
          \node (ar) at (-4.25,0){};

         \node (corner1) at (-9.7, 6){};
         \node (corner2) at (2.7, -6){};

         \node (circlecenter) at (-3.5,0){};


          \draw [out=-90, in=-90, relative, decoration={markings, mark=at position 0.5 with {\arrow{>}}}, postaction={decorate}]
          (aul.center) to (aur.center);
          \draw [out=90, in=90, relative, decoration={markings, mark=at position 0.5 with {\arrow{<}}}, postaction={decorate}]
          (adl.center) to (adr.center);


          \draw[out=45, in=135, relative, decoration={markings, mark=at position 0.5 with {\arrow{>}}}, postaction={decorate}]
          (aur.center) to (b.center);
          \draw[out=-45, in=-135, relative, decoration={markings, mark=at position 0.85 with {\arrow{<}}}, postaction={decorate}]
          (aul.center) to (h);
          \draw[out=45, in=135, relative, decoration={markings, mark=at position 0.5 with {\arrow{>}}}, postaction={decorate}]
          (adl.center) to (f.center);
          \draw[out=-45, in=-135, relative, decoration={markings, mark=at position 0.85 with {\arrow{<}}}, postaction={decorate}]
          (adr.center) to (d);

          \draw[out=60, in =120,relative, decoration={markings, mark=at position 0.95 with {\arrow{<}}}, postaction={decorate}]
          (h.center) to (b);
          \draw[out=60, in =120,relative, decoration={markings, mark=at position 0.9 with {\arrow{<}}}, postaction={decorate}]
          (d.center) to (f);
          \draw[out=20, in=160,relative, decoration={markings, mark=at position 0.9 with {\arrow{<}}}, postaction={decorate}]
          (b) to (d.center);
          \draw[out=60, in =120, relative, decoration={markings, mark=at position 0.1 with {\arrow{>}}}, postaction={decorate}]
          (b.center) to (d);
          \draw[out=20, in=160,relative, decoration={markings, mark=at position 0.9 with {\arrow{<}}}, postaction={decorate}]
          (f) to (h.center);
          \draw[out=60, in =120, relative, decoration={markings, mark=at position 0.1 with {\arrow{>}}}, postaction={decorate}]
          (f.center) to (h);
          
          \filldraw[color=green, fill=green!5,thick]
          (c) ellipse (1.25 and 1.7)
          (e) ellipse (2.8 and 1.7)
          (i) ellipse (2.8 and 1.7);
          \filldraw[color=green, fill=green!5, even odd rule]
          (corner1) rectangle (corner2)
          (circlecenter) ellipse (4.5 and 5.5);
          \fill[color=green!5]
          (c) ellipse (1.25 and 1.7)
          (e) ellipse (2.8 and 1.7)
          (i) ellipse (2.8 and 1.7);

      \end{scope}

\end{scope}


 \begin{scope}[scale=0.9, xshift=6cm, yshift=0cm]
    
    \begin{scope}[scale=0.9, xshift=-0.25cm]
    \draw [-{Stealth}] 
    (-7,0) -- (-5,0);
     \draw [-{Stealth}] 
    (-5,0) -- (-7,0);
    \end{scope}
    
    \end{scope}
    
    \begin{scope}[scale=0.9, xshift=14cm, yshift=-6.5cm]
    
    \begin{scope}[scale=0.9, xshift=0cm]
    \draw [-{Stealth}, double] 
    (-7,0.5) -- (-7,-1.5);
    \end{scope}
    
    \end{scope}
    

\begin{scope}[scale=0.9, yshift = -14cm]
    
           \begin{scope}[scale=0.9, xshift=12cm]
          \node(a) at (-5,0) {};



         \node(b) at (-2,3){};
         \node(c) at (-1,0){};
         \node (d) at (-2, -3) {};
         \node (e) at (-5.75,-4.25){};
         \node (f) at (-8, -3){};
         \node (g) at (-9, 0){};
         \node (h) at (-8,3){};
         \node (i) at (-5.75, 4.25){};
          
          \node(aul) at (-6, 2){};
          \node (aur) at (-4,2){};
          \node(adr) at (-4, -2){};
          \node (adl) at (-6,-2){};
          
          \node (al) at (-5.75,0){};
          \node (ar) at (-4.25,0){};

         \node (corner1) at (-9.7, 6){};
         \node (corner2) at (2.7, -6){};

         \node (circlecenter) at (-3.5,0){};



          \draw [out=0, in=-160, relative, decoration={markings, mark=at position 0.5 with {\arrow{>}}}, postaction={decorate}]
          (adr.center) to (ar);
          \draw [out=-60, in=-120, relative, decoration={markings, mark=at position 0.5 with {\arrow{>}}}, postaction={decorate}]
          (ar) to (al.center);
          \draw [out=-15, in=-180, relative, decoration={markings, mark=at position 0.5 with {\arrow{>}}}, postaction={decorate}]
          (al.center) to (adl.center);
          \draw [out=-15, in=180, relative, decoration={markings, mark=at position 0.5 with {\arrow{>}}}, postaction={decorate}]
          (aul.center) to (al);
          \draw [out=-60, in=-120, relative, decoration={markings, mark=at position 0.5 with {\arrow{>}}}, postaction={decorate}]
          (al) to (ar.center);
          \draw [out=-15, in=180, relative, decoration={markings, mark=at position 0.5 with {\arrow{>}}}, postaction={decorate}]
          (ar.center) to (aur.center);

          \draw[out=45, in=135, relative, decoration={markings, mark=at position 0.5 with {\arrow{>}}}, postaction={decorate}]
          (aur.center) to (b.center);
          \draw[out=-45, in=-135, relative, decoration={markings, mark=at position 0.85 with {\arrow{<}}}, postaction={decorate}]
          (aul.center) to (h);
          \draw[out=45, in=135, relative, decoration={markings, mark=at position 0.5 with {\arrow{>}}}, postaction={decorate}]
          (adl.center) to (f.center);
          \draw[out=-45, in=-135, relative, decoration={markings, mark=at position 0.85 with {\arrow{<}}}, postaction={decorate}]
          (adr.center) to (d);

          \draw[out=60, in =120,relative, decoration={markings, mark=at position 0.95 with {\arrow{<}}}, postaction={decorate}]
          (h.center) to (b);
          \draw[out=60, in =120,relative, decoration={markings, mark=at position 0.9 with {\arrow{<}}}, postaction={decorate}]
          (d.center) to (f);
          \draw[out=20, in=160,relative, decoration={markings, mark=at position 0.9 with {\arrow{<}}}, postaction={decorate}]
          (b) to (d.center);
          \draw[out=60, in =120, relative, decoration={markings, mark=at position 0.1 with {\arrow{>}}}, postaction={decorate}]
          (b.center) to (d);
          \draw[out=20, in=160,relative, decoration={markings, mark=at position 0.9 with {\arrow{<}}}, postaction={decorate}]
          (f) to (h.center);
          \draw[out=60, in =120, relative, decoration={markings, mark=at position 0.1 with {\arrow{>}}}, postaction={decorate}]
          (f.center) to (h);
          
          \filldraw[color=green, fill=green!5,thick]
          (c) ellipse (1.25 and 1.7)
          (e) ellipse (2.8 and 1.7)
          (i) ellipse (2.8 and 1.7);
          \filldraw[color=green, fill=green!5, even odd rule]
          (corner1) rectangle (corner2)
          (circlecenter) ellipse (4.5 and 5.5);
          \fill[color=green!5]
          (c) ellipse (1.25 and 1.7)
          (e) ellipse (2.8 and 1.7)
          (i) ellipse (2.8 and 1.7);

      \end{scope}

\end{scope}
       
       \end{scope}
  \end{tikzpicture}
  \vspace{-7pt}
    \caption{A (morally) degree-$2$ vertex lets us know that we can clasp arcs in the diagram corresponding to the two adjacent vertices in the graph}
    \label{fig:v2 cycle degree 2}
\end{figure}
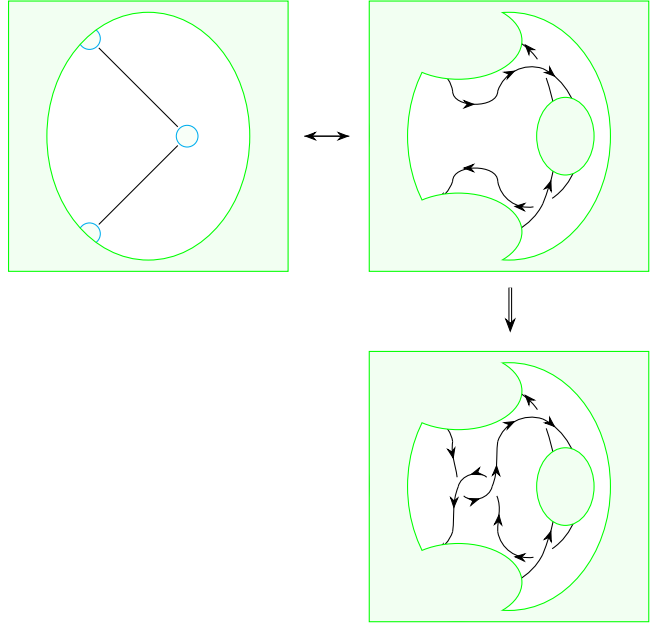

%% file: Figures/clasping_k-balanced.tex
\begin{figure}[h]
    \centering
  \begin{tikzpicture}
      \begin{scope}[scale=0.9, >=Stealth]
       
  \begin{scope}[scale = 0.6, xshift=-8cm]
  
       \node (a) at (-3,1){};
       \node (b) at (-1,1){};
       \node (c) at (1,1) {};
       \node (d) at (3,1){};
       \node (e) at (3, -1) {};
       \node (f) at (1, -1) {};
       \node (g) at (-1, -1) {};
       \node (h) at (-3, -1) {};
        \draw[]
        (a) -- (b)
        (b) -- (c) 
        (c) -- (d)
        (d) -- (e)
        (e) -- (f)
        (f) -- (g)
        (g) -- (h);
        \draw[dashed]
        (h) -- (a);
        
        \filldraw[color=cyan, fill=cyan!5] 
        (a) circle (1/3)
        (b) circle (1/3)
        (c) circle (1/3)
        (d) circle (1/3)       
        (e) circle (1/3)
        (f) circle (1/3)
        (g) circle (1/3)
        (h) circle (1/3); 
        
       \end{scope}

\begin{scope}[scale = 0.6, xshift=0cm]
  
       \node (a) at (-3,1){};
       \node (b) at (-1,1){};
       \node (c) at (1,1) {};
       \node (d) at (3,1){};
       \node (e) at (3, -1) {};
       \node (f) at (1, -1) {};
       \node (g) at (-1, -1) {};
       \node (h) at (-3, -1) {};
        \draw[]
        (c) -- (f)
        (a) -- (b)
        (b) -- (c) 
        (c) -- (d)
        (f) -- (g)
        (g) -- (h);
        \draw[dashed]
        (h) -- (a);
        \draw[out=20, in=160, relative]
        (c) to (d)
        (e) to (c)
        (c) to (e);
        
        \filldraw[color=cyan, fill=cyan!5] 
        (a) circle (1/3)
        (b) circle (1/3)
        (c) circle (1/3)
        (d) circle (1/3)       
        (e) circle (1/3)
        (f) circle (1/3)
        (g) circle (1/3)
        (h) circle (1/3); 
        
       \end{scope}

\begin{scope}[scale = 0.6, xshift=8cm]
  
       \node (a) at (-3,1){};
       \node (b) at (-1,1){};
       \node (c) at (1,1) {};
       \node (d) at (3,1){};
       \node (e) at (3, -1) {};
       \node (f) at (1, -1) {};
       \node (g) at (-1, -1) {};
       \node (h) at (-3, -1) {};
        \draw[]
        (b) -- (g)
        (a) -- (b)
        (b) -- (c) 
        (c) -- (d)
        (g) -- (h);
        \draw[dashed]
        (h) -- (a);
        \draw[out=20, in=160, relative]
        (b) to (c)
        (c) to (d)
        (e) to (c)
        (c) to (e)
        (f) to (b)
        (b) to (f);
        
        \filldraw[color=cyan, fill=cyan!5] 
        (a) circle (1/3)
        (b) circle (1/3)
        (c) circle (1/3)
        (d) circle (1/3)       
        (e) circle (1/3)
        (f) circle (1/3)
        (g) circle (1/3)
        (h) circle (1/3); 
        
       \end{scope}
       
       \end{scope}
       
  \end{tikzpicture}
  \vspace{-8pt}
    \caption{Performing successive clasp moves on a $k$-Balanced diagram}
    \label{fig:clasping k-balanced}
\end{figure}
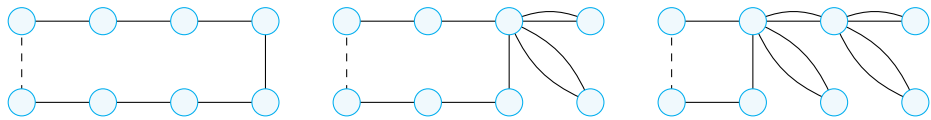

%% file: Figures/clasping_Bal_type_2_x=2.tex
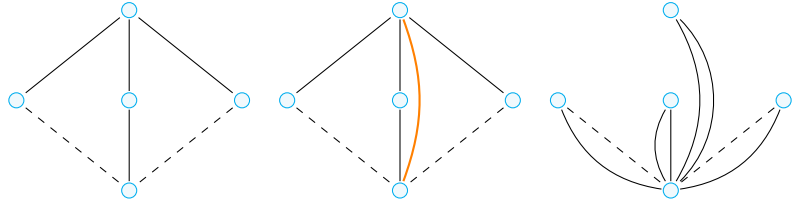
\begin{figure}[h]
    \centering
  \begin{tikzpicture}
      \begin{scope}[scale=0.5, >=Stealth]
       
  \begin{scope}[scale = 0.6, xshift=-12cm]
  
       \node (a) at (-5,0){};
       \node (b) at (0,4){};
       \node (c) at (0,2) {};
       \node (d) at (0,0){};
       \node (e) at (0, -2) {};
       \node (f) at (0, -4) {};
       \node (g) at (5, 0) {};
       \draw[]
        (b) -- (g)
        (a) -- (b)
        (b) -- (d) 
        (d) -- (f);
        \draw[dashed]
        (a) -- (f)
        (g) -- (f);
        
        \filldraw[color=cyan, fill=cyan!5] 
        (a) circle (1/3)
        (b) circle (1/3)
        (d) circle (1/3)       
        (f) circle (1/3)
        (g) circle (1/3); 
        
       \end{scope}

 \begin{scope}[scale = 0.6, xshift=0cm]
  
       \node (a) at (-5,0){};
       \node (b) at (0,4){};
       \node (c) at (0,2) {};
       \node (d) at (0,0){};
       \node (e) at (0, -2) {};
       \node (f) at (0, -4) {};
       \node (g) at (5, 0) {};
        \draw[]
        (b) -- (g)
        (a) -- (b)
        (b) -- (d) 
        (d) -- (f);
        \draw[dashed]
        (a) -- (f)
        (g) -- (f);
        \draw[out=20, in = 160, relative, thick, orange]{}
        (b) to (f);
        
        \filldraw[color=cyan, fill=cyan!5] 
        (a) circle (1/3)
        (b) circle (1/3)
        (d) circle (1/3)       
        (f) circle (1/3)
        (g) circle (1/3); 
        
       \end{scope}

 \begin{scope}[scale = 0.6, xshift=12cm]
  
       \node (a) at (-5,0){};
       \node (b) at (0,4){};
       \node (c) at (0,2) {};
       \node (d) at (0,0){};
       \node (e) at (0, -2) {};
       \node (f) at (0, -4) {};
       \node (g) at (5, 0) {};
    \draw[]
        (d) -- (f);
        \draw[dashed]
        (a) -- (f)
        (g) -- (f);
        \draw[out=30, in = 150, relative]{}
        (b) to (f)
        (f) to (a)
        (f) to (d)
        (g) to (f);
        \draw[out = 45, in =135, relative]{}
        (b) to (f);
        
        \filldraw[color=cyan, fill=cyan!5] 
        (a) circle (1/3)
        (b) circle (1/3)
        (d) circle (1/3)       
        (f) circle (1/3)
        (g) circle (1/3); 
        
       \end{scope}
       
       \end{scope}
       
  \end{tikzpicture}
    \caption{$x=2$: Left is a generalized $A$-state graph of a Balanced type $2$ diagram with $x=2$. The orange in the middle shows the vertices that will be clasped together to produce a diagram whose $A$-state graph looks like the figure on the right}
    \label{fig:clasping Bal type 2 x=2}
\end{figure}

%% file: Figures/clasping_Bal_type_2_x_geq_4.tex
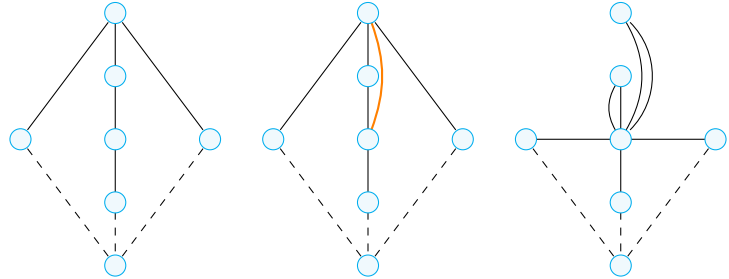
\begin{figure}[h]
    \centering
  \begin{tikzpicture}
      \begin{scope}[scale=0.7, >=Stealth]
       
  \begin{scope}[scale = 0.6, xshift=-8cm]
  
       \node (a) at (-3,0){};
       \node (b) at (0,4){};
       \node (c) at (0,2) {};
       \node (d) at (0,0){};
       \node (e) at (0, -2) {};
       \node (f) at (0, -4) {};
       \node (g) at (3, 0) {};
        \draw[]
        (b) -- (g)
        (a) -- (b)
        (b) -- (c) 
        (c) -- (d)
        (d) -- (e);
         \draw[dashed]
          (e) -- (f)
        (a) -- (f)
        (g) -- (f);
        
        \filldraw[color=cyan, fill=cyan!5] 
        (a) circle (1/3)
        (b) circle (1/3)
        (c) circle (1/3)
        (d) circle (1/3)       
        (e) circle (1/3)
        (f) circle (1/3)
        (g) circle (1/3); 
        
       \end{scope}

 \begin{scope}[scale = 0.6, xshift=0cm]
  
       \node (a) at (-3,0){};
       \node (b) at (0,4){};
       \node (c) at (0,2) {};
       \node (d) at (0,0){};
       \node (e) at (0, -2) {};
       \node (f) at (0, -4) {};
       \node (g) at (3, 0) {};
        \draw[]
        (b) -- (g)
        (a) -- (b)
        (b) -- (c) 
        (c) -- (d)
       (d) -- (e);
        \draw[dashed]
          (e) -- (f)
        (a) -- (f)
        (g) -- (f);
        \draw[out = 20, in = 160, relative, thick, orange]{}
        (b) to (d);
        
        \filldraw[color=cyan, fill=cyan!5] 
        (a) circle (1/3)
        (b) circle (1/3)
        (c) circle (1/3)
        (d) circle (1/3)       
        (e) circle (1/3)
        (f) circle (1/3)
        (g) circle (1/3); 
        
       \end{scope}

 \begin{scope}[scale = 0.6, xshift=8cm]
  
       \node (a) at (-3,0){};
       \node (b) at (0,4){};
       \node (c) at (0,2) {};
       \node (d) at (0,0){};
       \node (e) at (0, -2) {};
       \node (f) at (0, -4) {};
       \node (g) at (3, 0) {};
        \draw[]
        (a) --(d)
        (d)--(g)
        (c) -- (d)
        (d) -- (e);
        \draw[dashed]
          (e) -- (f)
        (a) -- (f)
        (g) -- (f);
        \draw[out=30, in=150, relative]
        (b) to (d)
        (d) to (c);
        \draw[out=45, in=135, relative]
        (b) to (d);
        
        \filldraw[color=cyan, fill=cyan!5] 
        (a) circle (1/3)
        (b) circle (1/3)
        (c) circle (1/3)
        (d) circle (1/3)       
        (e) circle (1/3)
        (f) circle (1/3)
        (g) circle (1/3); 
        
       \end{scope}
       
       \end{scope}
       
  \end{tikzpicture}
    \caption{$x\geq 4$: The orange in the middle shows the vertices that will be clasped together to produce a diagram whose $A$-state graph looks like the figure on the right}
    \label{fig:clasping Bal type 2 x geq 4}
\end{figure}

%% file: Figures/Oddly_Balanced_B_not_equal_to_n.tex

\begin{figure}
    \centering
  \begin{tikzpicture}

  \begin{scope}[scale=0.7, xshift=2.5cm]
  
      \begin{scope}[scale=1, >=Stealth, xshift=-4cm]
       \node (a) at (-4,0){};
       \node (b) at (-4,1){};
       \node (c) at (-3,1) {};
       \node (d) at (-2,1){};
       \node (g) at (-2,1) {};
       \node (h) at (-1,1){};
       \node (i) at (-1,0) {};
       \node (j) at (-1,-1){};
       \node (k) at (-2,-1) {};
       \node (n) at (-2,-1){};
       \node (o) at (-3,-1) {};
       \node (p) at (-4,-1){};

       \node (middle) at (-2.5,0){};


       \draw [out=45, in=135, relative] 
       (b.center) to (c);
       \draw [out=45, in=135, relative, decoration={markings, mark=at position 0.5 with {\arrow{<}}}, postaction={decorate}] 
       (c.center) to (d);
       \draw [out=60, in=135, relative] 
       (g.center) to (h.center);

       \draw [out=45, in=135, relative] 
       (h.center) to (i);
       \draw [out=45, in=135, relative, decoration={markings, mark=at position 0.5 with {\arrow{<}}}, postaction={decorate}] 
       (i.center) to (j.center);

       \draw [out=45, in=135, relative, decoration={markings, mark=at position 0.5 with {\arrow{<}}}, postaction={decorate}] 
       (j.center) to (k);
       \draw [out=45, in=135, relative, decoration={markings, mark=at position 0.5 with {\arrow{>}}}, postaction={decorate}] 
       (n.center) to (o);
       \draw [out=60, in=135, relative] 
       (o.center) to (p.center);

       \draw [out=45, in=135, relative, decoration={markings, mark=at position 0.5 with {\arrow{<}}}, postaction={decorate}] 
       (p.center) to (a);
       \draw [out=45, in=135, relative] 
       (a.center) to (b.center);


        \draw [decoration={markings, mark=at position 0.5 with {\arrow{>}}}, postaction={decorate}]
       (o) to (a.center);
       \draw [decoration={markings, mark=at position 0.5 with {\arrow{<}}}, postaction={decorate}]
       (a) to (c.center);


       \draw [decoration={markings, mark=at position 0.5 with {\arrow{<}}}, postaction={decorate}]
       (g) to (i.center);
       \draw [ decoration={markings, mark=at position 0.5 with {\arrow{>}}}, postaction={decorate}]
       (i) to (k.center);
       

       \draw [decoration={markings, mark=at position 0.5 with {\arrow{<}}}, postaction={decorate}]
       (d.center) to (middle);
       \draw []
       (middle) to (o.center);
       \draw []
       (c) to (middle.center);
       \draw [decoration={markings, mark=at position 0.5 with {\arrow{>}}}, postaction={decorate}]
       (middle.center) to (n);

       
       \end{scope}

  \begin{scope}[scale=1, >=Stealth]
       \node (a) at (-4,0){};
       \node (b) at (-4,1){};
       \node (c) at (-3,1) {};
       \node (d) at (-2,1){};
       \node (g) at (-2,1) {};
       \node (h) at (-1,1){};
       \node (i) at (-1,0) {};
       \node (j) at (-1,-1){};
       \node (k) at (-2,-1) {};
       \node (n) at (-2,-1){};
       \node (o) at (-3,-1) {};
       \node (p) at (-4,-1){};

       \node (middle) at (-2.5,0){};

\begin{scope}

       \draw [out=45, in=135, relative, draw=magenta, double=magenta,   very thick, cap=round] 
       (b.center) to (c);
       \draw [out=45, in=135, relative, draw=Cerulean!50, double=Cerulean!50,   very thick, cap=round] 
       (c) to (d);
       \draw [out=60, in=135, relative, draw=Periwinkle!60, double=Periwinkle!60,   very thick, cap=round] 
       (g) to (h.center);

       \draw [out=45, in=135, relative, draw=Periwinkle!60, double=Periwinkle!60,   very thick, cap=round] 
       (h.center) to (i);
       \draw [out=45, in=135, relative, draw=Emerald, double=Emerald,   very thick, cap=round] 
       (i) to (j.center);

       \draw [out=45, in=135, relative, draw=Emerald, double=Emerald,   very thick, cap=round] 
       (j.center) to (k);
       \draw [out=45, in=135, relative, draw=SpringGreen!90, double=SpringGreen!90,   very thick, cap=round] 
       (n) to (o);
       \draw [out=60, in=135, relative, draw=RoyalBlue, double=RoyalBlue,   very thick, cap=round] 
       (o) to (p.center);

       \draw [out=45, in=135, relative, draw=RoyalBlue, double=RoyalBlue,   very thick, cap=round] 
       (p.center) to (a);
       \draw [out=45, in=135, relative, draw=magenta, double=magenta,   very thick, cap=round] 
       (a) to (b.center);


        \draw [draw=magenta, double=magenta,   very thick, cap=round]
       (a) to (c);
        \draw [draw=RoyalBlue, double=RoyalBlue,   very thick, cap=round]
       (o) to (a);


       
       \draw [draw=Emerald, double=Emerald,   very thick, cap=round]
       (i) to (k);
       \draw [draw=Periwinkle!60, double=Periwinkle!60,   very thick, cap=round]
       (g) to (i);
       

       \draw [draw=Cerulean!50, double=Cerulean!50,   very thick, cap=round]
       (d) to (middle);
       \draw [draw=SpringGreen!90, double=SpringGreen!90,   very thick, cap=round]
       (middle) to (o);
       \draw [draw=Cerulean!50, double=Cerulean!50,   very thick, cap=round]
       (c) to (middle);
       \draw [draw=SpringGreen!90, double=SpringGreen!90,   very thick, cap=round]
       (middle) to (n);

    \end{scope}

       \begin{scope}

       \draw [out=45, in=135, relative, decoration={markings, mark=at position 0.5 with {\arrow{>}}}, postaction={decorate}] 
       (b.center) to (c);
       \draw [out=45, in=135, relative, decoration={markings, mark=at position 0.5 with {\arrow{<}}}, postaction={decorate}] 
       (c.center) to (d);
       \draw [out=60, in=135, relative, decoration={markings, mark=at position 0.5 with {\arrow{>}}}, postaction={decorate}] 
       (g.center) to (h.center);

       \draw [out=45, in=135, relative, decoration={markings, mark=at position 0.5 with {\arrow{>}}}, postaction={decorate}] 
       (h.center) to (i);
       \draw [out=45, in=135, relative, decoration={markings, mark=at position 0.5 with {\arrow{<}}}, postaction={decorate}] 
       (i.center) to (j.center);

       \draw [out=45, in=135, relative, decoration={markings, mark=at position 0.5 with {\arrow{<}}}, postaction={decorate}] 
       (j.center) to (k);
       \draw [out=45, in=135, relative, decoration={markings, mark=at position 0.5 with {\arrow{>}}}, postaction={decorate}] 
       (n.center) to (o);
       \draw [out=60, in=135, relative, decoration={markings, mark=at position 0.5 with {\arrow{<}}}, postaction={decorate}] 
       (o.center) to (p.center);

       \draw [out=45, in=135, relative, decoration={markings, mark=at position 0.5 with {\arrow{<}}}, postaction={decorate}] 
       (p.center) to (a);
       \draw [out=45, in=135, relative, decoration={markings, mark=at position 0.5 with {\arrow{>}}}, postaction={decorate}] 
       (a.center) to (b.center);


        \draw [decoration={markings, mark=at position 0.5 with {\arrow{>}}}, postaction={decorate}]
       (o) to (a.center);
       \draw [decoration={markings, mark=at position 0.5 with {\arrow{<}}}, postaction={decorate}]
       (a) to (c.center);


       \draw [decoration={markings, mark=at position 0.5 with {\arrow{<}}}, postaction={decorate}]
       (g) to (i.center);
       \draw [ decoration={markings, mark=at position 0.5 with {\arrow{>}}}, postaction={decorate}]
       (i) to (k.center);
       

       \draw [decoration={markings, mark=at position 0.5 with {\arrow{<}}}, postaction={decorate}]
       (d.center) to (middle);
       \draw []
       (middle) to (o.center);
       \draw []
       (c) to (middle.center);
       \draw [decoration={markings, mark=at position 0.5 with {\arrow{>}}}, postaction={decorate}]
       (middle.center) to (n);

       \end{scope}
         
       \end{scope}

  \begin{scope}[scale=1, >=Stealth, xshift=4cm]
       \node (a) at (-4,0){};
       \node (b) at (-4,1){};
       \node (c) at (-3,1) {};
       \node (d) at (-2,1){};
       \node (g) at (-2,1) {};
       \node (h) at (-1,1){};
       \node (i) at (-1,0) {};
       \node (j) at (-1,-1){};
       \node (k) at (-2,-1) {};
       \node (n) at (-2,-1){};
       \node (o) at (-3,-1) {};
       \node (p) at (-4,-1){};

       \node (middle) at (-2.5,0){};

\begin{scope}

       \draw [out=45, in=135, relative, draw=purp, double=purp,   very thick, cap=round] 
       (b.center) to (c);
       \draw [out=45, in=135, relative, draw=purp, double=purp,   very thick, cap=round] 
       (c) to (d);
       \draw [out=60, in=135, relative, draw=purp, double=purp,   very thick, cap=round] 
       (g) to (h.center);

       \draw [out=45, in=135, relative, draw=purp, double=purp,   very thick, cap=round] 
       (h.center) to (i);
       \draw [out=45, in=135, relative, draw=purp, double=purp,   very thick, cap=round] 
       (i) to (j.center);

       \draw [out=45, in=135, relative, draw=purp, double=purp,   very thick, cap=round] 
       (j.center) to (k);
       \draw [out=45, in=135, relative, draw=purp, double=purp,   very thick, cap=round] 
       (n) to (o);
       \draw [out=60, in=135, relative, draw=purp, double=purp,   very thick, cap=round] 
       (o) to (p.center);

       \draw [out=45, in=135, relative, draw=purp, double=purp,   very thick, cap=round] 
       (p.center) to (a);
       \draw [out=45, in=135, relative, draw=purp, double=purp,   very thick, cap=round] 
       (a) to (b.center);


        \draw [draw=orange, double=orange,   very thick, cap=round]
       (a) to (c);
        \draw [draw=orange, double=orange,   very thick, cap=round]
       (o) to (a);


       
       \draw [draw=cyan, double=cyan,   very thick, cap=round]
       (i) to (k);
       \draw [draw=cyan, double=cyan,   very thick, cap=round]
       (g) to (i);
       

       \draw [draw=cyan, double=cyan,   very thick, cap=round]
       (d) to (middle);
       \draw [draw=orange, double=orange,   very thick, cap=round]
       (middle) to (o);
       \draw [draw=orange, double=orange,   very thick, cap=round]
       (c) to (middle);
       \draw [draw=cyan, double=cyan,   very thick, cap=round]
       (middle) to (n);

    \end{scope}

       \begin{scope}

       \draw [out=45, in=135, relative, decoration={markings, mark=at position 0.5 with {\arrow{>}}}, postaction={decorate}] 
       (b.center) to (c);
       \draw [out=45, in=135, relative, decoration={markings, mark=at position 0.5 with {\arrow{<}}}, postaction={decorate}] 
       (c.center) to (d);
       \draw [out=60, in=135, relative, decoration={markings, mark=at position 0.5 with {\arrow{>}}}, postaction={decorate}] 
       (g.center) to (h.center);

       \draw [out=45, in=135, relative, decoration={markings, mark=at position 0.5 with {\arrow{>}}}, postaction={decorate}] 
       (h.center) to (i);
       \draw [out=45, in=135, relative, decoration={markings, mark=at position 0.5 with {\arrow{<}}}, postaction={decorate}] 
       (i.center) to (j.center);

       \draw [out=45, in=135, relative, decoration={markings, mark=at position 0.5 with {\arrow{<}}}, postaction={decorate}] 
       (j.center) to (k);
       \draw [out=45, in=135, relative, decoration={markings, mark=at position 0.5 with {\arrow{>}}}, postaction={decorate}] 
       (n.center) to (o);
       \draw [out=60, in=135, relative, decoration={markings, mark=at position 0.5 with {\arrow{<}}}, postaction={decorate}] 
       (o.center) to (p.center);

       \draw [out=45, in=135, relative, decoration={markings, mark=at position 0.5 with {\arrow{<}}}, postaction={decorate}] 
       (p.center) to (a);
       \draw [out=45, in=135, relative, decoration={markings, mark=at position 0.5 with {\arrow{>}}}, postaction={decorate}] 
       (a.center) to (b.center);


        \draw [decoration={markings, mark=at position 0.5 with {\arrow{>}}}, postaction={decorate}]
       (o) to (a.center);
       \draw [decoration={markings, mark=at position 0.5 with {\arrow{<}}}, postaction={decorate}]
       (a) to (c.center);


       \draw [decoration={markings, mark=at position 0.5 with {\arrow{<}}}, postaction={decorate}]
       (g) to (i.center);
       \draw [ decoration={markings, mark=at position 0.5 with {\arrow{>}}}, postaction={decorate}]
       (i) to (k.center);
       

       \draw [decoration={markings, mark=at position 0.5 with {\arrow{<}}}, postaction={decorate}]
       (d.center) to (middle);
       \draw []
       (middle) to (o.center);
       \draw []
       (c) to (middle.center);
       \draw [decoration={markings, mark=at position 0.5 with {\arrow{>}}}, postaction={decorate}]
       (middle.center) to (n);

       \end{scope}
         
       \end{scope}

\end{scope}

\begin{scope}[scale = 0.4, yshift=-5cm, thick]

\node (a) at (-2,2){};
\node (b) at (0,2){};
\node (c) at (2,2){};
\node (d) at (2,0){};
\node (e) at (0,0){};
\node (f) at (-2,0){};

\draw[]
(a) -- (c) -- (d) -- (f) -- (a)
(b) -- (e);

\filldraw[color=magenta, fill=magenta!5, very thick]
(a) circle (1/3);
\filldraw[color=Cerulean, fill=Cerulean!5, very thick]
(b) circle (1/3);
\filldraw[color=Periwinkle!90, fill=Periwinkle!5, very thick]
(c) circle (1/3);
\filldraw[color=Emerald, fill=Emerald!5, very thick]
(d) circle (1/3);
\filldraw[color=SpringGreen!90, fill=SpringGreen!5, very thick]
(e) circle (1/3);
\filldraw[color=RoyalBlue, fill=RoyalBlue!5, very thick]
(f) circle (1/3);

\end{scope}

  \end{tikzpicture}
    \caption{An Oddly Balanced knot diagram (left), its $A$-circles (top middle), $A$-state graph (bottom), and its $B$-circles (right)}
    \label{fig:Oddly Balanced B not equal to n}
\end{figure}
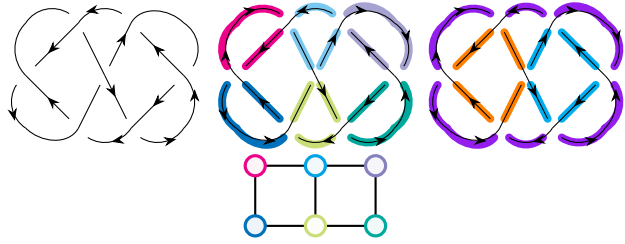


%% file: Figures/clasping_Oddly_Bal_type_2_special_case_x=1.tex
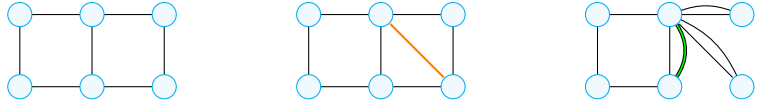
\begin{figure}[h]
    \centering
  \begin{tikzpicture}
      \begin{scope}[scale=0.8, >=Stealth]
       
 \begin{scope}[scale = 0.6, xshift=-8cm]
  
       \node (a) at (-3,1){};
       \node (b) at (-1,1){};
       \node (c) at (1,1) {};
       \node (d) at (1, -1) {};
       \node (e) at (-1, -1) {};
       \node (f) at (-3, -1) {};
        \draw[]
        (a) -- (b)
        (b) -- (c) 
        (b) -- (e)
        (c) -- (d)
        (d) -- (e)
         (e) -- (f)
         (f) -- (a);
        
        \filldraw[color=cyan, fill=cyan!5] 
        (a) circle (1/3)
        (b) circle (1/3)
        (c) circle (1/3)
        (d) circle (1/3)       
        (e) circle (1/3)
        (f) circle (1/3); 
        
       \end{scope}

\begin{scope}[scale = 0.6, xshift=0cm]
  
       \node (a) at (-3,1){};
       \node (b) at (-1,1){};
       \node (c) at (1,1) {};
       \node (d) at (1, -1) {};
       \node (e) at (-1, -1) {};
       \node (f) at (-3, -1) {};
        \draw[]
       (a) -- (b)
        (b) -- (c) 
        (b) -- (e)
        (c) -- (d)
        (d) -- (e)
         (e) -- (f)
         (f) -- (a);
         \draw[thick, orange]
         (b) to (d);
        
        \filldraw[color=cyan, fill=cyan!5] 
        (a) circle (1/3)
        (b) circle (1/3)
        (c) circle (1/3)
        (d) circle (1/3)       
        (e) circle (1/3)
        (f) circle (1/3); 
        
       \end{scope}

\begin{scope}[scale = 0.6, xshift=8cm]
  
       \node (a) at (-3,1){};
       \node (b) at (-1,1){};
       \node (c) at (1,1) {};
       \node (d) at (1, -1) {};
       \node (e) at (-1, -1) {};
       \node (f) at (-3, -1) {};
        \draw[]
        (a) -- (b)
        (b) -- (c) 
        (b) -- (e)
        (b) -- (d)
         (e) -- (f)
         (f) -- (a);
        \draw[out=20, in=160, relative]
        (b) to (c)
        (b) to (d);
        \draw[out=35, in =145, relative, double=green]
        (b) to (e);
        
        \filldraw[color=cyan, fill=cyan!5] 
        (a) circle (1/3)
        (b) circle (1/3)
        (c) circle (1/3)
        (d) circle (1/3)       
        (e) circle (1/3)
        (f) circle (1/3); 
        
       \end{scope}
       
       \end{scope}
       
  \end{tikzpicture}
    \caption{Special case $x=1$: We have only one (up to symmetry) choice of where to clasp, and this gives us a Burdened diagram of type $1$, with the edge preventing it from being Balanced overlined in green.}
    \label{fig:clasping Oddly Bal type 2 special case x=1}
\end{figure}

%% file: Figures/clasping_Oddly_Bal_type_2_special_case_x=3.tex
\begin{figure}[h]
    \centering
  \begin{tikzpicture}
      \begin{scope}[scale=0.65, >=Stealth]

      \begin{scope}[]

  \begin{scope}[scale = 0.6, xshift=-10cm]
  
       \node (a) at (0,3){};
       \node (b) at (0,1){};
       \node (c) at (0,-1) {};
       \node (d) at (0,-3){};
       \node (e) at (-3, 1) {};
       \node (f) at (-3, -1) {};
       \node (g) at (4, 1){};
       \node (h) at (4, -1){};
        \draw[]
        (a) -- (b)
        (b) -- (c) 
        (a) -- (g)
        (g) -- (h)
          (c) -- (d)
          (h) -- (d);
        \draw[out= 60, in = 120, relative]
           (d.center) to (a.center);
        
        \filldraw[color=cyan, fill=cyan!5] 
        (a) circle (1/3)
        (b) circle (1/3)
        (c) circle (1/3)
        (d) circle (1/3)       
        (g) circle (1/3)
        (h) circle (1/3);
        
       \end{scope}

 \begin{scope}[scale = 0.6, xshift=0cm]
  
        \node (a) at (0,3){};
       \node (b) at (0,1){};
       \node (c) at (0,-1) {};
       \node (d) at (0,-3){};
       \node (e) at (-3, 1){};
       \node (f) at (-3, -1) {};
       \node (g) at (4, 1){};
       \node (h) at (4, -1){};
         \draw[]
        (a) -- (b)
        (b) -- (c) 
        (a) -- (g)
        (g) -- (h)
          (c) -- (d)
          (h) -- (d);
        \draw[out= 60, in = 120, relative]
           (d.center) to (a.center);
        \draw[out = 35, in = 145, relative, thick, orange]
        (c.center) to (a.center);
        
        \filldraw[color=cyan, fill=cyan!5] 
        (a) circle (1/3)
        (b) circle (1/3)
        (c) circle (1/3)
        (d) circle (1/3)       
         (g) circle (1/3)
        (h) circle (1/3); 
        
       \end{scope}

 \begin{scope}[scale = 0.6, xshift=9cm]
  
        \node (a) at (0,3){};
       \node (b) at (0,1){};
       \node (c) at (0,-1) {};
       \node (d) at (0,-3){};
       \node (e) at (-3, 1) {};
       \node (f) at (-3, -1) {};
       \node (g) at (4, 1){};
       \node (h) at (4, -1){};
        \draw[]
        (b) -- (c) 
        (c) -- (g)
        (g) -- (h)
          (c) -- (d)
          (h) -- (d);
        \draw[out = 30, in = 150, relative]
        (b.center) to (c.center)
        (c.center) to (a.center);
        \draw[out = 45, in = 135, relative]
        (c.center) to (a.center);
        \draw[out = 60, in = 120, relative, double=green]
        (d.center) to (c.center);

        \filldraw[color=cyan, fill=cyan!5] 
        (a) circle (1/3)
        (b) circle (1/3)
        (c) circle (1/3)
        (d) circle (1/3)       
         (g) circle (1/3)
        (h) circle (1/3); 
        
       \end{scope}

      \end{scope}

\begin{scope}[yshift=-4.5cm]

 \begin{scope}[scale = 0.6, xshift=0cm]
  
        \node (a) at (0,3){};
       \node (b) at (0,1){};
       \node (c) at (0,-1) {};
       \node (d) at (0,-3){};
       \node (e) at (-3, 1){};
       \node (f) at (-3, -1) {};
       \node (g) at (4, 1){};
       \node (h) at (4, -1){};
         \draw[]
        (a) -- (b)
        (b) -- (c) 
        (a) -- (g)
        (g) -- (h)
          (c) -- (d)
          (h) -- (d);
        \draw[out= 60, in = 120, relative]
           (d) to (a);
        \draw[thick, orange]
        (a.center) to (h.center);
        
        \filldraw[color=cyan, fill=cyan!5] 
        (a) circle (1/3)
        (b) circle (1/3)
        (c) circle (1/3)
        (d) circle (1/3)       
         (g) circle (1/3)
        (h) circle (1/3); 
        
       \end{scope}

 \begin{scope}[scale = 0.6, xshift=9cm]
  
        \node (a) at (0,3){};
       \node (b) at (0,1){};
       \node (c) at (0,-1) {};
       \node (d) at (0,-3){};
       \node (e) at (-3, 1) {};
       \node (f) at (-3, -1) {};
       \node (g) at (4, 1){};
       \node (h) at (4, -1){};
        \draw[]
        (a.center) -- (h.center)
        (b.center) -- (c.center) 
        (b) -- (h)
        (g) -- (h)
          (c) -- (d)
          (h) -- (d);
        \draw[out = 30, in = 150, relative]
        (g.center) to (h.center);
        \draw[out = 15, in = 165, relative]
        (a.center) to (h.center);
        \draw[out = 20, in = 160, relative, double=green]
        (h.center) to (d.center);

        \filldraw[color=cyan, fill=cyan!5] 
        (a) circle (1/3)
        (b) circle (1/3)
        (c) circle (1/3)
        (d) circle (1/3)       
         (g) circle (1/3)
        (h) circle (1/3); 
        
       \end{scope}

\end{scope}

       \end{scope}
       
  \end{tikzpicture}
    \caption{Special Case $x=3$: The middle column shows our choice of clasping, and the right column shows the diagram that results from performing that clasp. The extra edge preventing the diagram from being Balanced is overlined in green.}
    \label{fig:clasping Oddly Bal type 2 special case x = 3}
\end{figure}
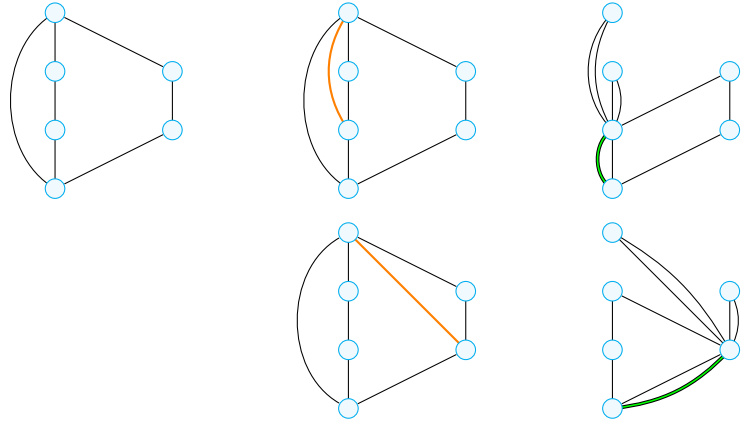

%% file: Figures/clasping_Oddly_Bal_type_2_x_=_1_can_do_it.tex
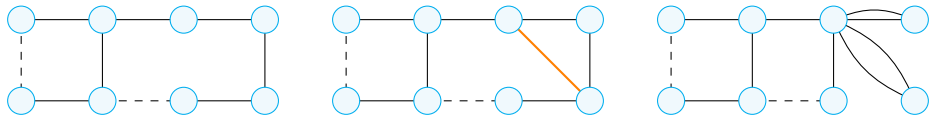
\begin{figure}[h]
    \centering
  \begin{tikzpicture}
      \begin{scope}[scale=0.9, >=Stealth]
       
  \begin{scope}[scale = 0.6, xshift=-8cm]
  
       \node (a) at (-3,1){};
       \node (b) at (-1,1){};
       \node (c) at (1,1) {};
       \node (d) at (3,1){};
       \node (e) at (3, -1) {};
       \node (f) at (1, -1) {};
       \node (g) at (-1, -1) {};
       \node (h) at (-3, -1) {};
        \draw[]
        (b) -- (g)
        (a) -- (b)
        (b) -- (c) 
        (c) -- (d)
        (d) -- (e)
        (e) -- (f)
        (g) -- (h);
        \draw[dashed]
        (h) -- (a)
        (f) -- (g);
        
        \filldraw[color=cyan, fill=cyan!5] 
        (a) circle (1/3)
        (b) circle (1/3)
        (c) circle (1/3)
        (d) circle (1/3)       
        (e) circle (1/3)
        (f) circle (1/3)
        (g) circle (1/3)
        (h) circle (1/3); 
        
       \end{scope}

  \begin{scope}[scale = 0.6, xshift=0cm]
  
       \node (a) at (-3,1){};
       \node (b) at (-1,1){};
       \node (c) at (1,1) {};
       \node (d) at (3,1){};
       \node (e) at (3, -1) {};
       \node (f) at (1, -1) {};
       \node (g) at (-1, -1) {};
       \node (h) at (-3, -1) {};
        \draw[]
        (b) -- (g)
        (a) -- (b)
        (b) -- (c) 
        (c) -- (d)
        (d) -- (e)
        (e) -- (f)
        (g) -- (h);
        \draw[dashed]
        (h) -- (a)
        (f) -- (g);
        \draw[thick, orange]
        (c) -- (e);
        
        \filldraw[color=cyan, fill=cyan!5] 
        (a) circle (1/3)
        (b) circle (1/3)
        (c) circle (1/3)
        (d) circle (1/3)       
        (e) circle (1/3)
        (f) circle (1/3)
        (g) circle (1/3)
        (h) circle (1/3); 
        
       \end{scope}

\begin{scope}[scale = 0.6, xshift=8cm]
  
       \node (a) at (-3,1){};
       \node (b) at (-1,1){};
       \node (c) at (1,1) {};
       \node (d) at (3,1){};
       \node (e) at (3, -1) {};
       \node (f) at (1, -1) {};
       \node (g) at (-1, -1) {};
       \node (h) at (-3, -1) {};
        \draw[]
        (c) -- (f)
         (b) -- (g)
        (a) -- (b)
        (b) -- (c) 
        (c) -- (d)
        (g) -- (h);
        \draw[dashed]
        (h) -- (a)
        (f) -- (g);
        \draw[out=20, in=160, relative]
        (c) to (d)
        (e) to (c)
        (c) to (e);
        
        \filldraw[color=cyan, fill=cyan!5] 
        (a) circle (1/3)
        (b) circle (1/3)
        (c) circle (1/3)
        (d) circle (1/3)       
        (e) circle (1/3)
        (f) circle (1/3)
        (g) circle (1/3)
        (h) circle (1/3); 
        
       \end{scope}
       
       \end{scope}
       
  \end{tikzpicture}
    \caption{$x=1$: When $k_2\geq 6$, we can clasp our $(k_1, k_2)$-Oddly Balanced diagram and obtain a $(k_1, k_2-2)$-Oddly Balanced diagram}
    \label{fig:clasping Oddly Bal type 2 x=1 can do it}
\end{figure}

%% file: Figures/clasping_Oddly_Bal_type_2_x=3.tex
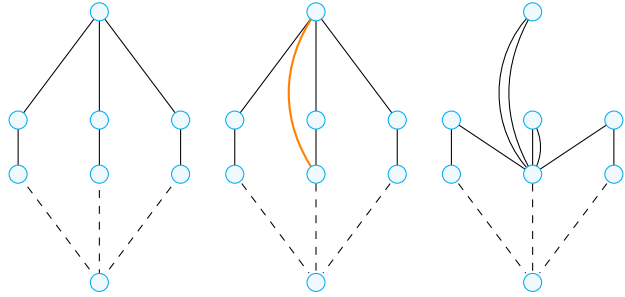
\begin{figure}[h]
    \centering
  \begin{tikzpicture}
      \begin{scope}[scale=0.6, >=Stealth]
       
  \begin{scope}[scale = 0.6, xshift=-8cm]
  
       \node (a) at (0,5){};
       \node (b) at (0,1){};
       \node (c) at (0,-1) {};
       \node (d) at (0,-5){};
       \node (e) at (-3, 1) {};
       \node (f) at (-3, -1) {};
       \node (g) at (3, 1){};
       \node (h) at (3, -1){};
        \draw[]
        (a.center) -- (b.center)
        (b.center) -- (c.center) 
        (a.center) -- (e.center)
        (e.center) -- (f.center)
        (a.center) -- (g.center)
        (g.center) -- (h.center);
         \draw[dashed]
          (f.center) -- (d.center)
          (c.center) -- (d.center)
          (h.center) -- (d.center);
        
        \filldraw[color=cyan, fill=cyan!5] 
        (a) circle (1/3)
        (b) circle (1/3)
        (c) circle (1/3)
        (d) circle (1/3)       
        (e) circle (1/3)
        (f) circle (1/3)
        (g) circle (1/3)
        (h) circle (1/3);
        
       \end{scope}

 \begin{scope}[scale = 0.6, xshift=0cm]
  
        \node (a) at (0,5){};
       \node (b) at (0,1){};
       \node (c) at (0,-1) {};
       \node (d) at (0,-5){};
       \node (e) at (-3, 1) {};
       \node (f) at (-3, -1) {};
       \node (g) at (3, 1){};
       \node (h) at (3, -1){};
        \draw[]
        (a.center) -- (b.center)
        (b.center) -- (c.center) 
        (a.center) -- (e.center)
        (e.center) -- (f.center)
        (a.center) -- (g.center)
        (g.center) -- (h.center);
         \draw[dashed]
          (f) -- (d)
          (c) -- (d)
          (h) -- (d);
        \draw[out = 35, in = 145, relative, thick, orange]
        (c.center) to (a.center);
        
        \filldraw[color=cyan, fill=cyan!5] 
        (a) circle (1/3)
        (b) circle (1/3)
        (c) circle (1/3)
        (d) circle (1/3)       
        (e) circle (1/3)
        (f) circle (1/3)
         (g) circle (1/3)
        (h) circle (1/3); 
        
       \end{scope}

 \begin{scope}[scale = 0.6, xshift=8cm]
  
        \node (a) at (0,5){};
       \node (b) at (0,1){};
       \node (c) at (0,-1) {};
       \node (d) at (0,-5){};
       \node (e) at (-3, 1) {};
       \node (f) at (-3, -1) {};
       \node (g) at (3, 1){};
       \node (h) at (3, -1){};
        \draw[]
        (b.center) -- (c.center) 
        (c.center) -- (e.center)
        (e.center) -- (f.center)
        (c.center) -- (g.center)
        (g.center) -- (h.center);
         \draw[dashed]
          (f) -- (d)
          (c) -- (d)
          (h) -- (d);
        \draw[out = 30, in = 150, relative]
        (b.center) to (c.center)
        (c.center) to (a.center);
        \draw[out = 45, in = 135, relative]
        (c.center) to (a.center);

        \filldraw[color=cyan, fill=cyan!5] 
        (a) circle (1/3)
        (b) circle (1/3)
        (c) circle (1/3)
        (d) circle (1/3)       
        (e) circle (1/3)
        (f) circle (1/3)
         (g) circle (1/3)
        (h) circle (1/3); 
        
       \end{scope}
       
       \end{scope}
       
  \end{tikzpicture}
    \caption{$x=3$: The orange in the middle shows the vertices that will be clasped together to produce a diagram whose $A$-state graph looks like the figure on the right}
    \label{fig:clasping Oddly Bal type 2 x = 3}
\end{figure}

%% file: Figures/clasping_Oddly_Bal_type_2_x_geq_5.tex
\begin{figure}[h]
    \centering
  \begin{tikzpicture}
      \begin{scope}[scale=0.6, >=Stealth]
       
  \begin{scope}[scale = 0.6, xshift=-8cm]
  
       \node (a) at (0,5){};
       \node (b) at (0,3){};
       \node (c) at (0,1) {};
       \node (d) at (0,-1){};
       \node (e) at (0, -3) {};
       \node (f) at (0, -5) {};
       \node (g) at (3, 3){};
       \node (h) at (3, -3){};
        \draw[]
        (a.center) -- (b.center)
        (b) -- (c) 
        (c) -- (d)
        (d) -- (e)
        (a.center) -- (g.center)
        (g.center) -- (h.center);
         \draw[dashed]
          (e) -- (f)
          (h) -- (f);
          \draw[out= 60, in = 120, relative, dashed]
          (f.center) to (a.center);
        
        \filldraw[color=cyan, fill=cyan!5] 
        (a) circle (1/3)
        (b) circle (1/3)
        (c) circle (1/3)
        (d) circle (1/3)       
        (e) circle (1/3)
        (f) circle (1/3)
        (g) circle (1/3)
        (h) circle (1/3);
        
       \end{scope}

 \begin{scope}[scale = 0.6, xshift=0cm]
  
       \node (a) at (0,5){};
       \node (b) at (0,3){};
       \node (c) at (0,1) {};
       \node (d) at (0,-1){};
       \node (e) at (0, -3) {};
       \node (f) at (0, -5) {};
        \node (g) at (3, 3){};
       \node (h) at (3, -3){};
        \draw[]
        (a) -- (b)
        (b) -- (c) 
        (c) -- (d)
        (d) -- (e)
        (a.center) -- (g.center)
        (g) -- (h);
         \draw[dashed]
          (e) -- (f)
          (h) -- (f);
          \draw[out= 60, in = 120, relative, dashed]
          (f) to (a);
        \draw[out = 35, in = 145, relative, thick, orange]
        (c.center) to (a.center);
        
        \filldraw[color=cyan, fill=cyan!5] 
        (a) circle (1/3)
        (b) circle (1/3)
        (c) circle (1/3)
        (d) circle (1/3)       
        (e) circle (1/3)
        (f) circle (1/3)
         (g) circle (1/3)
        (h) circle (1/3); 
        
       \end{scope}

 \begin{scope}[scale = 0.6, xshift=8cm]
  
       \node (a) at (0,5){};
       \node (b) at (0,3){};
       \node (c) at (0,1) {};
       \node (d) at (0,-1){};
       \node (e) at (0, -3) {};
       \node (f) at (0, -5) {};
        \node (g) at (3, 3){};
       \node (h) at (3, -3){};
        \draw[]
        (b) -- (c) 
        (c) -- (d)
       (d) -- (e)
        (g) -- (h)
        (g.center) -- (c.center);
         \draw[dashed]
          (e) -- (f)
          (h) -- (f);
          \draw[out= 60, in = 120, relative, dashed]
          (f.center) to (c.center);
        \draw[out = 30, in = 150, relative]
        (b.center) to (c.center)
        (c.center) to (a.center);
        \draw[out = 45, in = 135, relative]
        (c.center) to (a.center);

        \filldraw[color=cyan, fill=cyan!5] 
        (a) circle (1/3)
        (b) circle (1/3)
        (c) circle (1/3)
        (d) circle (1/3)       
        (e) circle (1/3)
        (f) circle (1/3)
         (g) circle (1/3)
        (h) circle (1/3); 
        
       \end{scope}
       
       \end{scope}
       
  \end{tikzpicture}
    \caption{$x\geq 5$: The orange in the middle shows the vertices that will be clasped together to produce a diagram whose $A$-state graph looks like the figure on the right}
    \label{fig:clasping Oddly Bal type 2 x geq 5}
\end{figure}
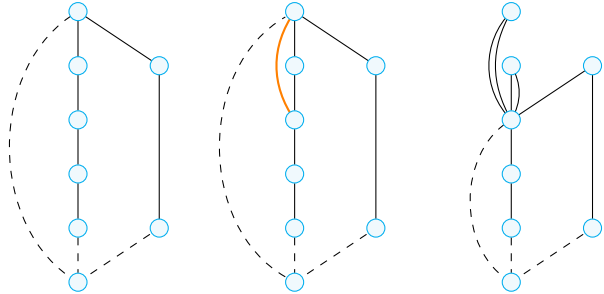

%% file: Figures/smoothing_number_vs_burdening_number.tex
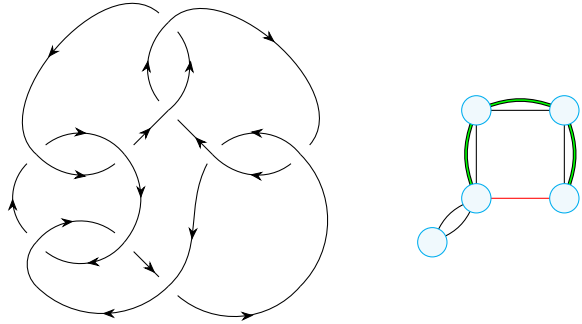
\begin{figure}[h]
    \centering
  \begin{tikzpicture}
      \begin{scope}[scale=0.65, >=Stealth]

      \begin{scope}[scale=0.9, xshift=-4cm]
       \node (a) at (-3,0){};
       \node (b) at (-1,0){};
       \node (c) at (0,3) {};
       \node (d) at (0,1){};
       \node (e) at (1,0) {};
       \node (f) at (3,0){};
       \node (g) at (3,-3) {};
       \node (h) at (0, -3){};
       \node (i) at (-3,-3){};
       \node (j) at (-3, -2){};
       \node (k) at (-1, -2){};
       

       \draw [out=90, in=90, relative, decoration={markings, mark=at position 0.5 with {\arrow{<}}}, postaction={decorate}] 
       (a.center) to (c);
       \draw [out=90, in=90, relative, decoration={markings, mark=at position 0.5 with {\arrow{>}}}, postaction={decorate}] 
       (c.center) to (f);
       \draw [out=45, in=135, relative, decoration={markings, mark=at position 0.5 with {\arrow{}}}, postaction={decorate}] 
       (f.center) to (g.center);
       \draw [out=45, in=135, relative, decoration={markings, mark=at position 0.5 with {\arrow{<}}}, postaction={decorate}] 
       (g.center) to (h);
       \draw [out=45, in=135, relative, decoration={markings, mark=at position 0.5 with {\arrow{>}}}, postaction={decorate}] 
       (h.center) to (i.center);
       \draw [out=45, in=135, relative, decoration={markings, mark=at position 0.5 with {\arrow{}}}, postaction={decorate}] 
       (i.center) to (j.center);
       \draw [out=45, in=135, relative, decoration={markings, mark=at position 0.5 with {\arrow{>}}}, postaction={decorate}] 
       (j) to (a);

       \draw [ decoration={markings, mark=at position 0.5 with {\arrow{>}}}, postaction={decorate}] 
       (b) to (d.center);
       \draw [ decoration={markings, mark=at position 0.5 with {\arrow{>}}}, postaction={decorate}] 
       (e.center) to (d);
       \draw [out=-30, in=170, relative, decoration={markings, mark=at position 0.5 with {\arrow{<}}}, postaction={decorate}] 
       (h.center) to (e);
       \draw [decoration={markings, mark=at position 0.5 with {\arrow{<}}}, postaction={decorate}] 
       (h) to (k);
       \draw [out=40, in=140, relative, decoration={markings, mark=at position 0.5 with {\arrow{>}}}, postaction={decorate}] 
       (b.center) to (k.center);


       \draw [out=45, in=135, relative, decoration={markings, mark=at position 0.5 with {\arrow{>}}}, postaction={decorate}] 
       (a) to (b.center);
       \draw [out=45, in=140, relative, decoration={markings, mark=at position 0.5 with {\arrow{<}}}, postaction={decorate}] 
       (b) to (a.center);
       \draw [out=45, in=140, relative, decoration={markings, mark=at position 0.5 with {\arrow{>}}}, postaction={decorate}] 
       (d) to (c.center);
       \draw [out=45, in=135, relative, decoration={markings, mark=at position 0.5 with {\arrow{<}}}, postaction={decorate}] 
       (c) to (d.center);
       \draw [out=45, in=140, relative, decoration={markings, mark=at position 0.5 with {\arrow{>}}}, postaction={decorate}] 
       (f) to (e.center);
       \draw [out=45, in=135, relative, decoration={markings, mark=at position 0.5 with {\arrow{<}}}, postaction={decorate}] 
       (e) to (f.center);
       \draw [out=45, in=135, relative, decoration={markings, mark=at position 0.5 with {\arrow{>}}}, postaction={decorate}] 
       (j.center) to (k);
       \draw [out=45, in=135, relative, decoration={markings, mark=at position 0.5 with {\arrow{>}}}, postaction={decorate}] 
       (k.center) to (j);
       \end{scope}



  \begin{scope}[scale = 0.9, xshift=4cm]
  
       \node (a) at (-1,1){};
       \node (b) at (1,1){};
       \node (c) at (1,-1) {};
       \node (d) at (-1,-1){};
       \node (e) at (-2, -2) {};
        \draw[]
        (a) -- (b)
        (b) -- (c) 
        (d) -- (a);
        \draw[red]
        (c) -- (d);
        \draw[out=20, in=160, relative]
        (a) to (b)
        (b) to (c)
        (d) to (a)
        (d) to (e)
        (e) to (d);
        \draw[out=20, in=160, relative, double=green]
        (a) to (b)
        (b) to (c)
        (d) to (a);

        \filldraw[color=cyan, fill=cyan!5] 
        (a) circle (1/3)
        (b) circle (1/3)
        (c) circle (1/3)
        (d) circle (1/3)       
        (e) circle (1/3);
        
       \end{scope}

       
       \end{scope}
       
  \end{tikzpicture}
    \caption{Smoothing the crossing that corresponds to the red edge will transform this Burdened type $1$ diagram into a Balanced type $0$ diagram. Smoothing crossings corresponding to the three green edges will transform this Burdened type $1$ diagram into a Balanced type $1$ diagram.  The burdening number is $m=3$.}
    \label{fig:smoothing number vs burdening number}
\end{figure}

%% file: Figures/Knot_and_Fam_coeff_0.tex
\begin{figure}[p]
    \centering
  \begin{tikzpicture}
      \begin{scope}[scale=0.85, >=Stealth]
       \node (a) at (-5,3){};
       \node (b) at (-4,3){};
       \node (c) at (-2,3) {};
       \node (d) at (-1,3){};
       \node (e) at (1,3) {};
       \node (f) at (-4,1){};
       \node (g) at (-2,1) {};
       \node (h) at (-3,0){};
       \node (i) at (-3,-1) {};
       \node (j) at (-3,-2){};
       \node (k) at (-1.5,-3) {};
       \node (l) at (0,0){};
       \node (m) at (0,-1) {};
       \node (n) at (0,-2){};
       \node (o) at (1,-3) {};
       \node (p) at (0,-4){};
       \node (q) at (2,-4) {};
       \node (r) at (3,2){};
       \node (s) at (4,3) {};
       \node (t) at (4,1){};
       
       \node (rr) at (2,-2){};
       \node (ss) at (3.25,-1.5) {};
       \node (tt) at (2.75,-3.5){};

        \node (labelj1) at (-2.75,-2) {\tiny\tcc{$1$}};
       \node (labelj4) at (-3,-2.2){\tiny\tcc{$4$}};

       \draw [out=90, in=90, relative, decoration={markings, mark=at position 0.5 with {\arrow{>}}}, postaction={decorate}] 
       (a.center) to (b);
       \draw [decoration={markings, mark=at position 0.5 with {\arrow{<}}}, postaction={decorate}] 
       (a) to (e.center);
       \draw [out=45, in=90, relative, decoration={markings, mark=at position 0.5 with {\arrow{>}}}, postaction={decorate}] 
       (e) to (s.center);

       \draw [out=0, in=-90, relative, decoration={markings, mark=at position 0.5 with {\arrow{>}}}, postaction={decorate}] 
       (b) to (f.center);
       \draw [ decoration={markings, mark=at position 0.5 with {\arrow{>}}}, postaction={decorate}] 
       (f.center) to (g);
       \draw [out=90, in=90, relative, decoration={markings, mark=at position 0.5 with {\arrow{>}}}, postaction={decorate}] 
       (f) to (g.center);
       \draw [out=15, in=160, relative, decoration={markings, mark=at position 0.5 with {\arrow{>}}}, postaction={decorate}] 
       (g.center) to (j);
       \draw [out=45, in=160, relative, decoration={markings, mark=at position 0.5 with {\arrow{>}}}, postaction={decorate}] 
       (j.center) to (f);
       \draw [ decoration={markings, mark=at position 0.5 with {\arrow{<}}}, postaction={decorate}] 
       (j.center) to (k);

       \draw [out=-120, in=-90, relative, decoration={markings, mark=at position 0.5 with {\arrow{>}}}, postaction={decorate}] 
       (a) ..controls (-7,3) and (-7, -6).. (k.center);
       \draw [decoration={markings, mark=at position 0.5 with {\arrow{>}}}, postaction={decorate}] 
       (k.center) to (n);
       \draw [ out=20, in=180, relative, decoration={markings, mark=at position 0.5 with {\arrow{>}}}, postaction={decorate}] 
       (p) to (k);
       \draw [out=90, in=155, relative, decoration={markings, mark=at position 0.5 with {\arrow{>}}}, postaction={decorate}] 
       (j) to (a.center);
       \draw [out=10, in=170, relative, decoration={markings, mark=at position 0.5 with {\arrow{}}}, postaction={decorate}] 
       (s.center) to (ss.center);
         \draw [out=15, in=155, relative, decoration={markings, mark=at position 0.5 with {\arrow{}}}, postaction={decorate}] 
       (tt) to (q.center);

       \draw [out=10, in=170, relative, decoration={markings, mark=at position 0.5 with {\arrow{>}}}, postaction={decorate}] 
       (g) to (l.center);
       \draw [out=0, in=135, relative, decoration={markings, mark=at position 0.5 with {\arrow{<}}}, postaction={decorate}] 
       (e) to (l);
       \draw [out=45, in=135, relative, decoration={markings, mark=at position 0.5 with {\arrow{>}}}, postaction={decorate}] 
       (l.center) to (m);
       \draw [out=45, in=135, relative, decoration={markings, mark=at position 0.5 with {\arrow{<}}}, postaction={decorate}] 
       (m.center) to (n);
       \draw [out=45, in=135, relative, decoration={markings, mark=at position 0.5 with {\arrow{<}}}, postaction={decorate}] 
       (n.center) to (m);
       \draw [out=45, in=135, relative, decoration={markings, mark=at position 0.5 with {\arrow{>}}}, postaction={decorate}] 
       (m.center) to (l);

       \draw [ decoration={markings, mark=at position 0.5 with {\arrow{>}}}, postaction={decorate}] 
       (n.center) to (q);
       \draw [ decoration={markings, mark=at position 0.5 with {\arrow{>}}}, postaction={decorate}] 
       (q.center) to (p);
       \draw [out=90, in=90, relative, decoration={markings, mark=at position 0.5 with {\arrow{>}}}, postaction={decorate}] 
       (q) to (p.center);
       \draw [out=45, in=180, relative, decoration={markings, mark=at position 0.5 with {\arrow{>}}}, postaction={decorate}] 
       (p.center) to (o);  

       \draw [out=75, in=-160, relative, decoration={markings, mark=at position 0.75 with {\arrow{<}}}, postaction={decorate}] 
       (e.center) to (rr);

        \draw [out=0, in=155, relative, decoration={markings, mark=at position 0.75 with {\arrow{<}}}, postaction={decorate}] 
       (rr.center) to (o);
       \draw [out=15, in =170, relative, decoration={markings, mark=at position 0.5 with {\arrow{>}}}, postaction={decorate}] 
       (rr.center) to (ss);
       \draw [decoration={markings, mark=at position 0.5 with {\arrow{>}}}, postaction={decorate}] 
       (ss.center) to (tt);
       \draw [out=90, in=60, relative, decoration={markings, mark=at position 0.5 with {\arrow{>}}}, postaction={decorate}] 
       (ss) to (tt.center);
       \draw [ decoration={markings, mark=at position 0.5 with {\arrow{>}}}, postaction={decorate}] 
       (tt.center) to (rr);

       \draw[thick, cyan]
       (e) circle (1/3);



       
       \end{scope}
  \end{tikzpicture}
    \caption{A $16$-crossing almost-positive diagram (negative crossing circled) of an almost-positive knot with second Jones coefficient equal to $0$. DT code: $[4, 8, 22, 2, 26, 24, -30, -12, -28, -16, 6, 32, 10, -20, -18, -14]$}
    \label{fig:Knot coeff 0}

       \begin{tikzpicture}
      \begin{scope}[scale=0.85, >=Stealth]
       \node (a) at (-5,3){};
       \node (b) at (-4,3){};
       \node (c) at (-2,3) {};
       \node (d) at (-1,3){};
       \node (e) at (1,3) {};
       \node (f) at (-4,1){};
       \node (g) at (-2,1) {};
       \node (h) at (-3,0){};
       \node (i) at (-3,-1) {};
       \node (j) at (-3,-2){};
       \node (k) at (-1.5,-3) {};
       \node (l) at (0,0){};
       \node (m) at (0,-1) {};
       \node (n) at (0,-2){};
       \node (o) at (1,-3) {};
       \node (p) at (0,-4){};
       \node (q) at (2,-4) {};
       \node (r) at (3,2){};
       \node (s) at (4,3) {};
       \node (t) at (4,1){};
       
       \node (rr) at (2,-2){};
       \node (ss) at (3.25,-1.5) {};
       \node (tt) at (2.75,-3.5){};

       \draw [out=90, in=90, relative, decoration={markings, mark=at position 0.5 with {\arrow{>}}}, postaction={decorate}] 
       (a.center) to (b);
       \draw [decoration={markings, mark=at position 0.5 with {\arrow{<}}}, postaction={decorate}] 
       (a) to (e.center);
       \draw [out=45, in=90, relative, decoration={markings, mark=at position 0.5 with {\arrow{>}}}, postaction={decorate}] 
       (e) to (s.center);

       \draw [out=0, in=-90, relative, decoration={markings, mark=at position 0.5 with {\arrow{>}}}, postaction={decorate}] 
       (b) to (f.center);
       \draw [ decoration={markings, mark=at position 0.5 with {\arrow{>}}}, postaction={decorate}] 
       (f.center) to (g);
       \draw [out=90, in=90, relative, decoration={markings, mark=at position 0.5 with {\arrow{>}}}, postaction={decorate}] 
       (f) to (g.center);
       \draw [out=15, in=160, relative, decoration={markings, mark=at position 0.5 with {\arrow{>}}}, postaction={decorate}] 
       (g.center) to (j);
       \draw [out=45, in=160, relative, decoration={markings, mark=at position 0.5 with {\arrow{>}}}, postaction={decorate}] 
       (j.center) to (f);
       \draw [ decoration={markings, mark=at position 0.5 with {\arrow{<}}}, postaction={decorate}] 
       (j.center) to (k);

       \draw [out=-120, in=-90, relative, decoration={markings, mark=at position 0.5 with {\arrow{>}}}, postaction={decorate}] 
       (a) ..controls (-7,3) and (-7, -6).. (k.center);
       \draw [decoration={markings, mark=at position 0.5 with {\arrow{>}}}, postaction={decorate}] 
       (k.center) to (n);
       \draw [ out=20, in=180, relative, decoration={markings, mark=at position 0.5 with {\arrow{>}}}, postaction={decorate}] 
       (p) to (k);
       \draw [out=90, in=155, relative, decoration={markings, mark=at position 0.5 with {\arrow{>}}}, postaction={decorate}] 
       (j) to (a.center);
       \draw [out=15, in=-170, relative, decoration={markings, mark=at position 0.5 with {\arrow{>}}}, postaction={decorate}] 
       (t) to (ss.center);
         \draw [out=15, in=155, relative, decoration={markings, mark=at position 0.5 with {\arrow{>}}}, postaction={decorate}] 
       (tt) to (q.center);

       \draw [out=10, in=170, relative, decoration={markings, mark=at position 0.5 with {\arrow{>}}}, postaction={decorate}] 
       (g) to (l.center);
       \draw [out=0, in=135, relative, decoration={markings, mark=at position 0.5 with {\arrow{<}}}, postaction={decorate}] 
       (e) to (l);
       \draw [out=45, in=135, relative, decoration={markings, mark=at position 0.5 with {\arrow{>}}}, postaction={decorate}] 
       (l.center) to (m);
       \draw [out=45, in=135, relative, decoration={markings, mark=at position 0.5 with {\arrow{<}}}, postaction={decorate}] 
       (m.center) to (n);
       \draw [out=45, in=135, relative, decoration={markings, mark=at position 0.5 with {\arrow{<}}}, postaction={decorate}] 
       (n.center) to (m);
       \draw [out=45, in=135, relative, decoration={markings, mark=at position 0.5 with {\arrow{>}}}, postaction={decorate}] 
       (m.center) to (l);

       \draw [ decoration={markings, mark=at position 0.5 with {\arrow{>}}}, postaction={decorate}] 
       (n.center) to (q);
       \draw [ decoration={markings, mark=at position 0.5 with {\arrow{>}}}, postaction={decorate}] 
       (q.center) to (p);
       \draw [out=90, in=90, relative, decoration={markings, mark=at position 0.5 with {\arrow{>}}}, postaction={decorate}] 
       (q) to (p.center);
       \draw [out=45, in=180, relative, decoration={markings, mark=at position 0.5 with {\arrow{>}}}, postaction={decorate}] 
       (p.center) to (o);  

       \draw [out=-25, in=-165, relative, decoration={markings, mark=at position 0.75 with {\arrow{<}}}, postaction={decorate}] 
       (r.center) to (rr);
       \draw [ decoration={markings, mark=at position 0.5 with {\arrow{>}}}, postaction={decorate}] 
       (r.center) to (s);
       \draw [decoration={markings, mark=at position 0.5 with {\arrow{>}}}, postaction={decorate}] 
       (s.center) to (t);
       \draw [out=90, in=60, relative, decoration={markings, mark=at position 0.5 with {\arrow{>}}}, postaction={decorate}] 
       (s) to (t.center);
       \draw [ decoration={markings, mark=at position 0.5 with {\arrow{>}}}, postaction={decorate}] 
       (t.center) to (r);
       \draw [out=20, in=160, relative, decoration={markings, mark=at position 0.75 with {\arrow{<}}}, postaction={decorate}] 
       (e.center) to (r);

        \draw [out=0, in=155, relative, decoration={markings, mark=at position 0.75 with {\arrow{<}}}, postaction={decorate}] 
       (rr.center) to (o);
       \draw [out=15, in =170, relative, decoration={markings, mark=at position 0.5 with {\arrow{>}}}, postaction={decorate}] 
       (rr.center) to (ss);
       \draw [decoration={markings, mark=at position 0.5 with {\arrow{>}}}, postaction={decorate}] 
       (ss.center) to (tt);
       \draw [out=90, in=60, relative, decoration={markings, mark=at position 0.5 with {\arrow{>}}}, postaction={decorate}] 
       (ss) to (tt.center);
       \draw [ decoration={markings, mark=at position 0.5 with {\arrow{>}}}, postaction={decorate}] 
       (tt.center) to (rr);

       \filldraw[color=white]
       (3,0) ellipse (1.25 and 0.7);

       \node (dots) at (3,0){$\dots$};

       \draw[very thick, red] (1.8, 3.2) -- (5, 3.2) -- (5, -1) -- (1.8, -1) -- (1.8, 3.2);
       
       \end{scope}
  \end{tikzpicture}
    \caption{An infinite family of almost-positive diagrams of almost-positive knots with second Jones coefficient equal to $0$}
    \label{fig:Family coeff 0}
\end{figure}

%% file: Figures/Knot_and_Fam_coeff_1.tex
\begin{figure}[p]
    \centering
  \begin{tikzpicture}
      \begin{scope}[scale=0.85, >=Stealth]
       \node (a) at (-5,3){};
       \node (b) at (-4,3){};
       \node (c) at (-2,3) {};
       \node (d) at (-1,3){};
       \node (e) at (1,3) {};
       \node (f) at (-4,1){};
       \node (g) at (-2,1) {};
       \node (h) at (-3,0){};

       \node (labelh1) at (-2.75,0) {\tiny\tcc{$1$}};
       \node (labelh4) at (-3,-0.2){\tiny\tcc{$4$}};
       
       \node (i) at (-3,-1) {};
       \node (j) at (-3,-2){};
       \node (k) at (-1.5,-3) {};
       \node (l) at (0,0){};
       \node (m) at (0,-1) {};
       \node (n) at (0,-2){};
       \node (o) at (1,-3) {};
       \node (p) at (0,-4){};
       \node (q) at (2,-4) {};
       \node (r) at (3,2){};
       \node (s) at (4,3) {};
       \node (t) at (4,1){};

       \draw [out=90, in=90, relative, decoration={markings, mark=at position 0.5 with {\arrow{>}}}, postaction={decorate}] 
       (a.center) to (b);
       \draw [decoration={markings, mark=at position 0.5 with {\arrow{<}}}, postaction={decorate}] 
       (a) to (e.center);
       \draw [out=45, in=90, relative, decoration={markings, mark=at position 0.5 with {\arrow{>}}}, postaction={decorate}] 
       (e) to (s.center);

       \draw [out=0, in=-90, relative, decoration={markings, mark=at position 0.5 with {\arrow{>}}}, postaction={decorate}] 
       (b) to (f.center);
       \draw [ decoration={markings, mark=at position 0.5 with {\arrow{>}}}, postaction={decorate}] 
       (f.center) to (g);
       \draw [out=90, in=90, relative, decoration={markings, mark=at position 0.5 with {\arrow{>}}}, postaction={decorate}] 
       (f) to (g.center);
       \draw [out=45, in=180, relative, decoration={markings, mark=at position 0.5 with {\arrow{>}}}, postaction={decorate}] 
       (g.center) to (h);
       \draw [out=45, in=135, relative, decoration={markings, mark=at position 0.5 with {\arrow{<}}}, postaction={decorate}] 
       (h.center) to (i);
       \draw [out=45, in=135, relative, decoration={markings, mark=at position 0.5 with {\arrow{>}}}, postaction={decorate}] 
       (i.center) to (j);
       \draw [out=45, in=135, relative, decoration={markings, mark=at position 0.5 with {\arrow{>}}}, postaction={decorate}] 
       (j.center) to (i);
       \draw [out=45, in=135, relative, decoration={markings, mark=at position 0.5 with {\arrow{<}}}, postaction={decorate}] 
       (i.center) to (h);
       \draw [out=0, in=135, relative, decoration={markings, mark=at position 0.5 with {\arrow{>}}}, postaction={decorate}] 
       (h.center) to (f);
       \draw [ decoration={markings, mark=at position 0.5 with {\arrow{<}}}, postaction={decorate}] 
       (j.center) to (k);

       \draw [out=-120, in=-90, relative, decoration={markings, mark=at position 0.5 with {\arrow{>}}}, postaction={decorate}] 
       (a) ..controls (-7,3) and (-7, -6).. (k.center);
       \draw [decoration={markings, mark=at position 0.5 with {\arrow{>}}}, postaction={decorate}] 
       (k.center) to (n);
       \draw [ out=20, in=180, relative, decoration={markings, mark=at position 0.5 with {\arrow{>}}}, postaction={decorate}] 
       (p) to (k);
       \draw [out=90, in=155, relative, decoration={markings, mark=at position 0.5 with {\arrow{>}}}, postaction={decorate}] 
       (j) to (a.center);
       \draw [out=15, in=115, relative, decoration={markings, mark=at position 0.5 with {\arrow{>}}}, postaction={decorate}] 
       (t.center) to (q.center);

       \draw [out=10, in=170, relative, decoration={markings, mark=at position 0.5 with {\arrow{>}}}, postaction={decorate}] 
       (g) to (l.center);
       \draw [out=0, in=135, relative, decoration={markings, mark=at position 0.5 with {\arrow{<}}}, postaction={decorate}] 
       (e) to (l);
       \draw [out=45, in=135, relative, decoration={markings, mark=at position 0.5 with {\arrow{>}}}, postaction={decorate}] 
       (l.center) to (m);
       \draw [out=45, in=135, relative, decoration={markings, mark=at position 0.5 with {\arrow{<}}}, postaction={decorate}] 
       (m.center) to (n);
       \draw [out=45, in=135, relative, decoration={markings, mark=at position 0.5 with {\arrow{<}}}, postaction={decorate}] 
       (n.center) to (m);
       \draw [out=45, in=135, relative, decoration={markings, mark=at position 0.5 with {\arrow{>}}}, postaction={decorate}] 
       (m.center) to (l);

       \draw [ decoration={markings, mark=at position 0.5 with {\arrow{>}}}, postaction={decorate}] 
       (n.center) to (q);
       \draw [ decoration={markings, mark=at position 0.5 with {\arrow{>}}}, postaction={decorate}] 
       (q.center) to (p);
       \draw [out=90, in=90, relative, decoration={markings, mark=at position 0.5 with {\arrow{>}}}, postaction={decorate}] 
       (q) to (p.center);
       \draw [out=45, in=180, relative, decoration={markings, mark=at position 0.5 with {\arrow{>}}}, postaction={decorate}] 
       (p.center) to (o);  

       \draw [decoration={markings, mark=at position 0.5 with {\arrow{>}}}, postaction={decorate}] 
       (s.center) to (t.center);
       \draw [out=90, in=120, relative, decoration={markings, mark=at position 0.75 with {\arrow{<}}}, postaction={decorate}] 
       (e.center) to (o);



       \draw[thick, cyan]
       (e) circle (1/3);

       
       \end{scope}
  \end{tikzpicture}
    \caption{A $15$-crossing almost-positive knot diagram (negative crossing circled) of an almost-positive knot with second Jones coefficient equal to $-1$. DT code: $[4, 10, 30, 20, 2, 24, 22, -26, -14, 8, 28, 12, -18, -16, 6]$}
    \label{fig:Knot coeff 1}

  \begin{tikzpicture}
      \begin{scope}[scale=0.85, >=Stealth]
       \node (a) at (-5,3){};
       \node (b) at (-4,3){};
       \node (c) at (-2,3) {};
       \node (d) at (-1,3){};
       \node (e) at (1,3) {};
       \node (f) at (-4,1){};
       \node (g) at (-2,1) {};
       \node (h) at (-3,0){};
       \node (i) at (-3,-1) {};
       \node (j) at (-3,-2){};
       \node (k) at (-1.5,-3) {};
       \node (l) at (0,0){};
       \node (m) at (0,-1) {};
       \node (n) at (0,-2){};
       \node (o) at (1,-3) {};
       \node (p) at (0,-4){};
       \node (q) at (2,-4) {};
       \node (r) at (3,2){};
       \node (s) at (4,3) {};
       \node (t) at (4,1){};

       \draw [out=90, in=90, relative, decoration={markings, mark=at position 0.5 with {\arrow{>}}}, postaction={decorate}] 
       (a.center) to (b);
       \draw [decoration={markings, mark=at position 0.5 with {\arrow{<}}}, postaction={decorate}] 
       (a) to (e.center);
       \draw [out=45, in=90, relative, decoration={markings, mark=at position 0.5 with {\arrow{>}}}, postaction={decorate}] 
       (e) to (s.center);

       \draw [out=0, in=-90, relative, decoration={markings, mark=at position 0.5 with {\arrow{>}}}, postaction={decorate}] 
       (b) to (f.center);
       \draw [ decoration={markings, mark=at position 0.5 with {\arrow{>}}}, postaction={decorate}] 
       (f.center) to (g);
       \draw [out=90, in=90, relative, decoration={markings, mark=at position 0.5 with {\arrow{>}}}, postaction={decorate}] 
       (f) to (g.center);
       \draw [out=45, in=180, relative, decoration={markings, mark=at position 0.5 with {\arrow{>}}}, postaction={decorate}] 
       (g.center) to (h);
       \draw [out=45, in=135, relative, decoration={markings, mark=at position 0.5 with {\arrow{<}}}, postaction={decorate}] 
       (h.center) to (i);
       \draw [out=45, in=135, relative, decoration={markings, mark=at position 0.5 with {\arrow{>}}}, postaction={decorate}] 
       (i.center) to (j);
       \draw [out=45, in=135, relative, decoration={markings, mark=at position 0.5 with {\arrow{>}}}, postaction={decorate}] 
       (j.center) to (i);
       \draw [out=45, in=135, relative, decoration={markings, mark=at position 0.5 with {\arrow{<}}}, postaction={decorate}] 
       (i.center) to (h);
       \draw [out=0, in=135, relative, decoration={markings, mark=at position 0.5 with {\arrow{>}}}, postaction={decorate}] 
       (h.center) to (f);
       \draw [ decoration={markings, mark=at position 0.5 with {\arrow{<}}}, postaction={decorate}] 
       (j.center) to (k);

       \draw [out=-120, in=-90, relative, decoration={markings, mark=at position 0.5 with {\arrow{>}}}, postaction={decorate}] 
       (a) ..controls (-7,3) and (-7, -6).. (k.center);
       \draw [decoration={markings, mark=at position 0.5 with {\arrow{>}}}, postaction={decorate}] 
       (k.center) to (n);
       \draw [ out=20, in=180, relative, decoration={markings, mark=at position 0.5 with {\arrow{>}}}, postaction={decorate}] 
       (p) to (k);
       \draw [out=90, in=155, relative, decoration={markings, mark=at position 0.5 with {\arrow{>}}}, postaction={decorate}] 
       (j) to (a.center);
       \draw [out=15, in=115, relative, decoration={markings, mark=at position 0.5 with {\arrow{>}}}, postaction={decorate}] 
       (t) to (q.center);

       \draw [out=10, in=170, relative, decoration={markings, mark=at position 0.5 with {\arrow{>}}}, postaction={decorate}] 
       (g) to (l.center);
       \draw [out=0, in=135, relative, decoration={markings, mark=at position 0.5 with {\arrow{<}}}, postaction={decorate}] 
       (e) to (l);
       \draw [out=45, in=135, relative, decoration={markings, mark=at position 0.5 with {\arrow{>}}}, postaction={decorate}] 
       (l.center) to (m);
       \draw [out=45, in=135, relative, decoration={markings, mark=at position 0.5 with {\arrow{<}}}, postaction={decorate}] 
       (m.center) to (n);
       \draw [out=45, in=135, relative, decoration={markings, mark=at position 0.5 with {\arrow{<}}}, postaction={decorate}] 
       (n.center) to (m);
       \draw [out=45, in=135, relative, decoration={markings, mark=at position 0.5 with {\arrow{>}}}, postaction={decorate}] 
       (m.center) to (l);

       \draw [ decoration={markings, mark=at position 0.5 with {\arrow{>}}}, postaction={decorate}] 
       (n.center) to (q);
       \draw [ decoration={markings, mark=at position 0.5 with {\arrow{>}}}, postaction={decorate}] 
       (q.center) to (p);
       \draw [out=90, in=90, relative, decoration={markings, mark=at position 0.5 with {\arrow{>}}}, postaction={decorate}] 
       (q) to (p.center);
       \draw [out=45, in=180, relative, decoration={markings, mark=at position 0.5 with {\arrow{>}}}, postaction={decorate}] 
       (p.center) to (o);  

       \draw [out=-20, in=155, relative, decoration={markings, mark=at position 0.75 with {\arrow{<}}}, postaction={decorate}] 
       (r.center) to (o);
       \draw [ decoration={markings, mark=at position 0.5 with {\arrow{>}}}, postaction={decorate}] 
       (r.center) to (s);
       \draw [decoration={markings, mark=at position 0.5 with {\arrow{>}}}, postaction={decorate}] 
       (s.center) to (t);
       \draw [out=90, in=60, relative, decoration={markings, mark=at position 0.5 with {\arrow{>}}}, postaction={decorate}] 
       (s) to (t.center);
       \draw [ decoration={markings, mark=at position 0.5 with {\arrow{>}}}, postaction={decorate}] 
       (t.center) to (r);
       \draw [out=20, in=160, relative, decoration={markings, mark=at position 0.75 with {\arrow{<}}}, postaction={decorate}] 
       (e.center) to (r);

       \filldraw[color=white]
       (3,0) circle (1);

       \node (dots) at (3,0){$\dots$};

       \draw[very thick, red] (1.8, 3.2) -- (5, 3.2) -- (5, -1) -- (1.8, -1) -- (1.8, 3.2);
       
       \end{scope}
  \end{tikzpicture}
    \caption{An infinite family of almost-positive diagrams of almost-positive knots with second Jones coefficient equal to $-1$.}
    \label{fig:Family coeff 1}
\end{figure}

%% file: Figures/Family_skein.tex
\begin{figure}
    \centering
  \begin{tikzpicture}
      \begin{scope}[scale=0.43, >=Stealth, thin]

\begin{scope}[]
      \begin{scope}[xshift=0cm]
 
        \node (e) at (2.5, 3.5){};
        \node (ee) at (4, 3.5){};
        
       \node (r) at (3,2){};
       \node (s) at (4,3) {};
       \node (t) at (4,1){};

        \node (rr) at (3,-0.5){};
       \node (ss) at (4,0.5) {};
       \node (tt) at (4,-1.5){};
       
       \node (o) at (2.5, -3.5){};
       \node (q) at (4, -3.5){};

        \node () at (1,1) {$D_w$};

        \draw [] 
       (ee.center) to (s.center);
       \draw [] 
       (t) to (ss.center);
       \draw [out = -40, in = -140, relative, decoration={markings, mark=at position 0.75 with {\arrow{<}}}, postaction={decorate}] 
       (r.center) to (rr);
       \draw [ ] 
       (r.center) to (s);
       \draw [decoration={markings, mark=at position 0.5 with {\arrow{>}}}, postaction={decorate}] 
       (s.center) to (t);
       \draw [out=90, in=60, relative, decoration={markings, mark=at position 0.5 with {\arrow{>}}}, postaction={decorate}] 
       (s) to (t.center);
       \draw [ ] 
       (t.center) to (r);
       \draw [out=10, in = 170, relative, decoration={markings, mark=at position 0.75 with {\arrow{<}}}, postaction={decorate}] 
       (e.center) to (r);

       \draw [ ] 
       (tt) to (q.center);
       \draw [out = -35, in =170, relative, decoration={markings, mark=at position 0.75 with {\arrow{<}}}, postaction={decorate}] 
       (rr.center) to (o.center);
       \draw [ ] 
       (rr.center) to (ss);
       \draw [decoration={markings, mark=at position 0.5 with {\arrow{>}}}, postaction={decorate}] 
       (ss.center) to (tt);
       \draw [out=90, in=60, relative, decoration={markings, mark=at position 0.5 with {\arrow{>}}}, postaction={decorate}] 
       (ss) to (tt.center);
       \draw [] 
       (tt.center) to (rr);

       \filldraw[color=white]
       (3.3,-2.5) circle (0.8);

       \node (dots) at (3.4,-2.5){$\dots$};

       \draw[   red] (1.8, 3.5) -- (5, 3.5) -- (5, -3.5) -- (1.8, -3.5) -- (1.8, 3.5);

        \node (arrowtail) at (6.5,3){};
       \node (arrowhead) at (5, 3){};

       \draw[thick, cyan, ->]
       (arrowtail) -- (arrowhead);

      \end{scope}
\begin{scope}[xshift=3cm, yshift = -4.4cm]
\draw[->]
(0,0 ) to (-1.5,-1.5)
node[above, yshift=6pt]{$-$};
\draw[->]
(1, 0 ) to (2.5, -1.5)
node[above, xshift=4pt, yshift=7pt]{\tiny{smooth}};
    
\end{scope}
\end{scope}

\begin{scope}[yshift=-10cm]
      \begin{scope}[xshift=-4cm]
 
        \node (e) at (2.5, 3.5){};
        \node (ee) at (4, 3.5){};
        
       \node (r) at (3,2){};
       \node (s) at (4,3) {};
       \node (t) at (4,1){};

        \node (rr) at (3,-0.5){};
       \node (ss) at (4,0.5) {};
       \node (tt) at (4,-1.5){};
       
       \node (o) at (2.5, -3.5){};
       \node (q) at (4, -3.5){};

        \node () at (0.7,1) {$D_{w-1}$};

        \draw [] 
       (ee.center) to (ss.center);
       
        \draw [out = -5, in = -170, relative, decoration={markings, mark=at position 0.75 with {\arrow{<}}}, postaction={decorate}] 
       (e.center) to (rr);

       \draw [] 
       (tt) to (q.center);
       \draw [out = -35, in =170, relative, decoration={markings, mark=at position 0.75 with {\arrow{<}}}, postaction={decorate}] 
       (rr.center) to (o.center);
       \draw [ ] 
       (rr.center) to (ss);
       \draw [decoration={markings, mark=at position 0.5 with {\arrow{>}}}, postaction={decorate}] 
       (ss.center) to (tt);
       \draw [out=90, in=60, relative, decoration={markings, mark=at position 0.5 with {\arrow{>}}}, postaction={decorate}] 
       (ss) to (tt.center);
       \draw [] 
       (tt.center) to (rr);

       \filldraw[color=white]
       (3.3,-2.5) circle (0.8);

       \node (dots) at (3.4,-2.5){$\dots$};

       \draw[   red] (1.8, 3.5) -- (5, 3.5) -- (5, -3.5) -- (1.8, -3.5) -- (1.8, 3.5);

      \end{scope}

  \begin{scope}[xshift=4cm]
 
        \node (e) at (2.5, 3.5){};
        \node (ee) at (4, 3.5){};
        
       \node (r) at (3,2){};
       \node (s) at (3.8,2.5) {};
       \node (t) at (4,1){};

        \node (rr) at (3,-0.5){};
       \node (ss) at (4,0.5) {};
       \node (tt) at (4,-1.5){};
       
       \node (o) at (2.5, -3.5){};
       \node (q) at (4, -3.5){};

       \node () at (6.5, 1){$D_{wS}$};

       \draw [] 
       (t) to (ss.center);
       \draw [out = -25, in = -140, relative, decoration={markings, mark=at position 0.75 with {\arrow{<}}}, postaction={decorate}] 
       (r.center) to (rr);
       \draw [out = 30, in = 90, relative] 
       (r.center) to (s.center);
        \draw [out = 15, in = 170, relative] 
       (s.center) to (t);
       \draw [out=1, in= 60, relative, decoration={markings, mark=at position 0.5 with {\arrow{>}}}, postaction={decorate}] 
       (ee.center) to (t.center);
       \draw [ ] 
       (t.center) to (r);
       \draw [out=10, in = 170, relative, decoration={markings, mark=at position 0.75 with {\arrow{<}}}, postaction={decorate}] 
       (e.center) to (r);

       \draw [ ] 
       (tt) to (q.center);
       \draw [out = -35, in =170, relative, decoration={markings, mark=at position 0.75 with {\arrow{<}}}, postaction={decorate}] 
       (rr.center) to (o.center);
       \draw [ ] 
       (rr.center) to (ss);
       \draw [decoration={markings, mark=at position 0.5 with {\arrow{>}}}, postaction={decorate}] 
       (ss.center) to (tt);
       \draw [out=90, in=60, relative, decoration={markings, mark=at position 0.5 with {\arrow{>}}}, postaction={decorate}] 
       (ss) to (tt.center);
       \draw [] 
       (tt.center) to (rr);

       \filldraw[color=white]
       (3.3,-2.5) circle (0.8);

       \node (dots) at (3.4,-2.5){$\dots$};

       \draw[   red] (1.8, 3.5) -- (5, 3.5) -- (5, -3.5) -- (1.8, -3.5) -- (1.8, 3.5);

       


      \end{scope}

\begin{scope}[xshift=7cm, yshift = -4.4cm]
\draw[->]
(0,0 )to (-1,-1.5)
node[above, yshift=6pt]{$-$};
\draw[->]
(1, 0 ) to (2, -1.5)
node[above, xshift=7pt, yshift=7pt]{\tiny{smooth}};
    
\end{scope}

\end{scope}

\begin{scope}[xshift=4cm, yshift=-20cm]

      \begin{scope}[xshift=-3cm]
 
        \node (e) at (2.5, 3.5){};
        \node (ee) at (4, 3.5){};
        
       \node (r) at (3,2){};
       \node (s) at (4,3) {};
       \node (t) at (4,1){};

        \node (rr) at (3,-0.5){};
       \node (ss) at (4,0.5) {};
       \node (tt) at (4,-1.5){};
       
       \node (o) at (2.5, -3.5){};
       \node (q) at (4, -3.5){};

       \node () at (-1.3,1) {$D_{00} \# \underbrace{3_1 \# \dots \# 3_1}_{w-1 \text{ copies}}$};

        \draw [-] 
       (e.center) .. controls (r) and (s).. (ee.center);
       
        \draw [ ] 
       (rr) .. controls (r) and (t).. (ss.center);

       \draw [ ] 
       (tt) to (q.center);
       \draw [out = -35, in =170, relative, decoration={markings, mark=at position 0.75 with {\arrow{<}}}, postaction={decorate}] 
       (rr.center) to (o.center);
       \draw [ ] 
       (rr.center) to (ss);
       \draw [decoration={markings, mark=at position 0.5 with {\arrow{>}}}, postaction={decorate}] 
       (ss.center) to (tt);
       \draw [out=90, in=60, relative, decoration={markings, mark=at position 0.5 with {\arrow{>}}}, postaction={decorate}] 
       (ss) to (tt.center);
       \draw [] 
       (tt.center) to (rr);

       \filldraw[color=white]
       (3.3,-2.5) circle (0.8);

       \node (dots) at (3.4,-2.5){$\dots$};

       \draw[   red] (1.8, 3.5) -- (5, 3.5) -- (5, -3.5) -- (1.8, -3.5) -- (1.8, 3.5);

      \end{scope}

   \begin{scope}[xshift=3cm]
 
        \node (e) at (2.5, 3.5){};
        \node (ee) at (4, 3.5){};
        
       \node (r) at (3,2){};
       \node (s) at (4,3) {};
       \node (t) at (4,1){};

        \node (rr) at (3,-0.5){};
       \node (ss) at (4,0.5) {};
       \node (tt) at (4,-1.5){};
       
       \node (o) at (2.5, -3.5){};
       \node (q) at (4, -3.5){};

       \node () at (6.1,1) {$D_{w-1}$};

        \draw [] 
       (ee.center) to (ss.center);
       
        \draw [out = -5, in = -170, relative, decoration={markings, mark=at position 0.75 with {\arrow{<}}}, postaction={decorate}] 
       (e.center) to (rr);

       \draw [ ] 
       (tt) to (q.center);
       \draw [out = -35, in =170, relative, decoration={markings, mark=at position 0.75 with {\arrow{<}}}, postaction={decorate}] 
       (rr.center) to (o.center);
       \draw [ ] 
       (rr.center) to (ss);
       \draw [decoration={markings, mark=at position 0.5 with {\arrow{>}}}, postaction={decorate}] 
       (ss.center) to (tt);
       \draw [out=90, in=60, relative, decoration={markings, mark=at position 0.5 with {\arrow{>}}}, postaction={decorate}] 
       (ss) to (tt.center);
       \draw [] 
       (tt.center) to (rr);

       \filldraw[color=white]
       (3.3,-2.5) circle (0.8);

       \node (dots) at (3.4,-2.5){$\dots$};

       \draw[   red] (1.8, 3.5) -- (5, 3.5) -- (5, -3.5) -- (1.8, -3.5) -- (1.8, 3.5);

      \end{scope}

\end{scope}
 
           \end{scope}
  \end{tikzpicture}
    \caption{}
    \label{fig:Family skein}
\end{figure}

%% file: Figures/Knot_and_Fam_coeff_2.tex
\begin{figure}[p]
    \centering
  \begin{tikzpicture}
      \begin{scope}[scale=0.85, >=Stealth]
       \node (a) at (-5,3){};
       \node (b) at (-4,3){};
       \node (c) at (-2,3) {};
       \node (d) at (-1,3){};
       \node (e) at (1,3) {};
       \node (f) at (-4,1){};
       \node (g) at (-2,1) {};
       \node (h) at (-3,0){};
       \node (i) at (-3,-1) {};
       \node (j) at (-3,-2){};
       
        \node (labelj1) at (-2.75,-2) {\tiny\tcc{$1$}};
       \node (labelj4) at (-3,-2.2){\tiny\tcc{$4$}};
       
       \node (k) at (-1.5,-3) {};
       \node (l) at (0,0){};
       \node (m) at (0,-1) {};
       \node (n) at (0,-2){};
       \node (o) at (1,-3) {};
       \node (p) at (0,-4){};
       \node (q) at (2,-4) {};
       \node (r) at (3,2){};
       \node (s) at (4,3) {};
       \node (t) at (4,1){};

       \draw [out=90, in=90, relative, decoration={markings, mark=at position 0.5 with {\arrow{>}}}, postaction={decorate}] 
       (a.center) to (b);
       \draw [decoration={markings, mark=at position 0.5 with {\arrow{<}}}, postaction={decorate}] 
       (a) to (c);
       \draw [out=90, in=90, relative, decoration={markings, mark=at position 0.5 with {\arrow{>}}}, postaction={decorate}] 
       (c.center) to (d);
       \draw [decoration={markings, mark=at position 0.5 with {\arrow{<}}}, postaction={decorate}] 
       (c) to (e.center);
       \draw [out=45, in=90, relative, decoration={markings, mark=at position 0.5 with {\arrow{>}}}, postaction={decorate}] 
       (e) to (s.center);

       \draw [out=0, in=-90, relative, decoration={markings, mark=at position 0.5 with {\arrow{>}}}, postaction={decorate}] 
       (b) to (f.center);
       \draw [ decoration={markings, mark=at position 0.5 with {\arrow{>}}}, postaction={decorate}] 
       (f.center) to (g);
       \draw [out=-90, in=180, relative, decoration={markings, mark=at position 0.5 with {\arrow{>}}}, postaction={decorate}] 
       (g) to (c.center);
       \draw [out=90, in=90, relative, decoration={markings, mark=at position 0.5 with {\arrow{>}}}, postaction={decorate}] 
       (f) to (g.center);
       \draw [out=15, in=160, relative, decoration={markings, mark=at position 0.5 with {\arrow{>}}}, postaction={decorate}] 
       (g.center) to (j);
       \draw [out=45, in=160, relative, decoration={markings, mark=at position 0.5 with {\arrow{>}}}, postaction={decorate}] 
       (j.center) to (f);
       \draw [ decoration={markings, mark=at position 0.5 with {\arrow{<}}}, postaction={decorate}] 
       (j.center) to (k);

       \draw [out=-120, in=-90, relative, decoration={markings, mark=at position 0.5 with {\arrow{>}}}, postaction={decorate}] 
       (a) ..controls (-7,3) and (-7, -6).. (k.center);
       \draw [decoration={markings, mark=at position 0.5 with {\arrow{>}}}, postaction={decorate}] 
       (k.center) to (n);
       \draw [ out=20, in=180, relative, decoration={markings, mark=at position 0.5 with {\arrow{>}}}, postaction={decorate}] 
       (p) to (k);
       \draw [out=90, in=155, relative, decoration={markings, mark=at position 0.5 with {\arrow{>}}}, postaction={decorate}] 
       (j) to (a.center);
       \draw [out=15, in=115, relative, decoration={markings, mark=at position 0.5 with {\arrow{>}}}, postaction={decorate}] 
       (t.center) to (q.center);

       \draw [out=0, in=-135, relative, decoration={markings, mark=at position 0.5 with {\arrow{>}}}, postaction={decorate}] 
       (d) to (l.center);
       \draw [out=0, in=135, relative, decoration={markings, mark=at position 0.5 with {\arrow{<}}}, postaction={decorate}] 
       (e) to (l);
       \draw [out=45, in=135, relative, decoration={markings, mark=at position 0.5 with {\arrow{>}}}, postaction={decorate}] 
       (l.center) to (m);
       \draw [out=45, in=135, relative, decoration={markings, mark=at position 0.5 with {\arrow{<}}}, postaction={decorate}] 
       (m.center) to (n);
       \draw [out=45, in=135, relative, decoration={markings, mark=at position 0.5 with {\arrow{<}}}, postaction={decorate}] 
       (n.center) to (m);
       \draw [out=45, in=135, relative, decoration={markings, mark=at position 0.5 with {\arrow{>}}}, postaction={decorate}] 
       (m.center) to (l);

       \draw [ decoration={markings, mark=at position 0.5 with {\arrow{>}}}, postaction={decorate}] 
       (n.center) to (q);
       \draw [ decoration={markings, mark=at position 0.5 with {\arrow{>}}}, postaction={decorate}] 
       (q.center) to (p);
       \draw [out=90, in=90, relative, decoration={markings, mark=at position 0.5 with {\arrow{>}}}, postaction={decorate}] 
       (q) to (p.center);
       \draw [out=45, in=180, relative, decoration={markings, mark=at position 0.5 with {\arrow{>}}}, postaction={decorate}] 
       (p.center) to (o);  

       \draw [decoration={markings, mark=at position 0.5 with {\arrow{>}}}, postaction={decorate}] 
       (s.center) to (t.center);
       \draw [out=90, in=120, relative, decoration={markings, mark=at position 0.75 with {\arrow{<}}}, postaction={decorate}] 
       (e.center) to (o);

       \draw[thick, cyan]
       (e) circle (1/3);

       
       \end{scope}
  \end{tikzpicture}
    \caption{A $15$-crossing almost-positive knot diagram (negative crossing circled) of an almost-positive knot with second Jones coefficient equal to $-2$. DT code: $[4, 8, 22, 2, 20, 26, 24, -28, -14, 10, 6, 30, 12, -18, -16]$}
    \label{fig:Knot coeff 2}

  \begin{tikzpicture}
      \begin{scope}[scale=0.85, >=Stealth]
       \node (a) at (-5,3){};
       \node (b) at (-4,3){};
       \node (c) at (-2,3) {};
       \node (d) at (-1,3){};
       \node (e) at (1,3) {};
       \node (f) at (-4,1){};
       \node (g) at (-2,1) {};
       \node (h) at (-3,0){};
       \node (i) at (-3,-1) {};
       \node (j) at (-3,-2){};
       \node (k) at (-1.5,-3) {};
       \node (l) at (0,0){};
       \node (m) at (0,-1) {};
       \node (n) at (0,-2){};
       \node (o) at (1,-3) {};
       \node (p) at (0,-4){};
       \node (q) at (2,-4) {};
       \node (r) at (3,2){};
       \node (s) at (4,3) {};
       \node (t) at (4,1){};

       \draw [out=90, in=90, relative, decoration={markings, mark=at position 0.5 with {\arrow{>}}}, postaction={decorate}] 
       (a.center) to (b);
       \draw [decoration={markings, mark=at position 0.5 with {\arrow{<}}}, postaction={decorate}] 
       (a) to (c);
       \draw [out=90, in=90, relative, decoration={markings, mark=at position 0.5 with {\arrow{>}}}, postaction={decorate}] 
       (c.center) to (d);
       \draw [decoration={markings, mark=at position 0.5 with {\arrow{<}}}, postaction={decorate}] 
       (c) to (e.center);
       \draw [out=45, in=90, relative, decoration={markings, mark=at position 0.5 with {\arrow{>}}}, postaction={decorate}] 
       (e) to (s.center);

       \draw [out=0, in=-90, relative, decoration={markings, mark=at position 0.5 with {\arrow{>}}}, postaction={decorate}] 
       (b) to (f.center);
       \draw [ decoration={markings, mark=at position 0.5 with {\arrow{>}}}, postaction={decorate}] 
       (f.center) to (g);
       \draw [out=-90, in=180, relative, decoration={markings, mark=at position 0.5 with {\arrow{>}}}, postaction={decorate}] 
       (g) to (c.center);
       \draw [out=90, in=90, relative, decoration={markings, mark=at position 0.5 with {\arrow{>}}}, postaction={decorate}] 
       (f) to (g.center);
       \draw [out=15, in=160, relative, decoration={markings, mark=at position 0.5 with {\arrow{>}}}, postaction={decorate}] 
       (g.center) to (j);
       \draw [out=45, in=160, relative, decoration={markings, mark=at position 0.5 with {\arrow{>}}}, postaction={decorate}] 
       (j.center) to (f);
       \draw [ decoration={markings, mark=at position 0.5 with {\arrow{<}}}, postaction={decorate}] 
       (j.center) to (k);

       \draw [out=-120, in=-90, relative, decoration={markings, mark=at position 0.5 with {\arrow{>}}}, postaction={decorate}] 
       (a) ..controls (-7,3) and (-7, -6).. (k.center);
       \draw [decoration={markings, mark=at position 0.5 with {\arrow{>}}}, postaction={decorate}] 
       (k.center) to (n);
       \draw [ out=20, in=180, relative, decoration={markings, mark=at position 0.5 with {\arrow{>}}}, postaction={decorate}] 
       (p) to (k);
       \draw [out=90, in=155, relative, decoration={markings, mark=at position 0.5 with {\arrow{>}}}, postaction={decorate}] 
       (j) to (a.center);
       \draw [out=15, in=115, relative, decoration={markings, mark=at position 0.5 with {\arrow{>}}}, postaction={decorate}] 
       (t) to (q.center);

       \draw [out=0, in=-135, relative, decoration={markings, mark=at position 0.5 with {\arrow{>}}}, postaction={decorate}] 
       (d) to (l.center);
       \draw [out=0, in=135, relative, decoration={markings, mark=at position 0.5 with {\arrow{<}}}, postaction={decorate}] 
       (e) to (l);
       \draw [out=45, in=135, relative, decoration={markings, mark=at position 0.5 with {\arrow{>}}}, postaction={decorate}] 
       (l.center) to (m);
       \draw [out=45, in=135, relative, decoration={markings, mark=at position 0.5 with {\arrow{<}}}, postaction={decorate}] 
       (m.center) to (n);
       \draw [out=45, in=135, relative, decoration={markings, mark=at position 0.5 with {\arrow{<}}}, postaction={decorate}] 
       (n.center) to (m);
       \draw [out=45, in=135, relative, decoration={markings, mark=at position 0.5 with {\arrow{>}}}, postaction={decorate}] 
       (m.center) to (l);

       \draw [ decoration={markings, mark=at position 0.5 with {\arrow{>}}}, postaction={decorate}] 
       (n.center) to (q);
       \draw [ decoration={markings, mark=at position 0.5 with {\arrow{>}}}, postaction={decorate}] 
       (q.center) to (p);
       \draw [out=90, in=90, relative, decoration={markings, mark=at position 0.5 with {\arrow{>}}}, postaction={decorate}] 
       (q) to (p.center);
       \draw [out=45, in=180, relative, decoration={markings, mark=at position 0.5 with {\arrow{>}}}, postaction={decorate}] 
       (p.center) to (o);  

       \draw [out=-20, in=155, relative, decoration={markings, mark=at position 0.75 with {\arrow{<}}}, postaction={decorate}] 
       (r.center) to (o);
       \draw [ decoration={markings, mark=at position 0.5 with {\arrow{>}}}, postaction={decorate}] 
       (r.center) to (s);
       \draw [decoration={markings, mark=at position 0.5 with {\arrow{>}}}, postaction={decorate}] 
       (s.center) to (t);
       \draw [out=90, in=60, relative, decoration={markings, mark=at position 0.5 with {\arrow{>}}}, postaction={decorate}] 
       (s) to (t.center);
       \draw [ decoration={markings, mark=at position 0.5 with {\arrow{>}}}, postaction={decorate}] 
       (t.center) to (r);
       \draw [out=20, in=160, relative, decoration={markings, mark=at position 0.75 with {\arrow{<}}}, postaction={decorate}] 
       (e.center) to (r);

       \filldraw[color=white]
       (3,0) circle (1);

       \node (dots) at (3,0){$\dots$};

       \draw[very thick, red] (1.8, 3.2) -- (5, 3.2) -- (5, -1) -- (1.8, -1) -- (1.8, 3.2);
       
       \end{scope}
  \end{tikzpicture}
    \caption{An infinite family of almost-positive diagrams of almost-positive knots with second Jones coefficient equal to $-2$}
    \label{fig:Family coeff 2}
\end{figure}